\documentclass{amsart}

\usepackage{amssymb}
\usepackage{overpic}
\usepackage{enumitem}   
\usepackage{graphicx} 
\usepackage{amsrefs}
\usepackage{xcolor}
\usepackage[hidelinks]{hyperref}
\usepackage{multicol}

\graphicspath{{Figures/}}

\newtheorem{theorem}{Theorem}
\newtheorem{proposition}{Proposition} 
 
\newtheorem{remark}{Remark} 
 
\newtheorem{definition}{Definition}

\newtheorem{main}{Theorem} 

\begin{document}

\title[Reversible vector fields of type (2;1)]{Phase portraits of (2;1) reversible vector fields of low codimension}

\author[C. Buzzi, J. Llibre and P. Santana]
{Claudio Buzzi$^1$, Jaume Llibre$^2$ and Paulo Santana$^1$}

\address{$^1$ IBILCE--UNESP, CEP 15054--000, S. J. Rio Preto, S\~ao Paulo, Brazil}
\email{claudio.buzzi@unesp.br; paulo.santana@unesp.br}

\address{$^2$ Universitat Aut\`onoma de Barcelona, 08193 Bellaterra, Barcelona, Spain}
\email{jaume.llibre@uab.cat}



\maketitle

\begin{abstract}
	In this paper we study the phase portraits and bifurcation diagram of the symmetric singularities of codimensions zero, one and two of planar reversible vector fields having a line of reversibility.
\end{abstract}

\section{Introduction and Main Results}\label{sec1}

Given two real $C^k$, $k\geqslant1$, functions of two variables $P$, $Q\colon\mathbb{R}^2\to\mathbb{R}$ we associate the planar $C^k$-differential system given by
\begin{equation}\label{01}
	\dot x=P(x,y), \quad \dot y=Q(x,y),
\end{equation}
where the dot denotes the derivative with respect to the independent variable $t$. By abuse of notation we also call the map $X=(P,Q)$ a planar \emph{vector field}. If $P$ and $Q$ are polynomials, then system \eqref{01} is a \emph{planar polynomial differential system}. In this case we say that system \eqref{01} has \emph{degree n} if the maximum of the degrees of $P$ and $Q$ is $n$. If $n=1$, then system \eqref{01} is called a \emph{linear differential system}. This last class of systems is already completely understood. See~\cite[Chapter $1$]{Perko}. 

However for $n\geqslant2$ we know very few things. This class is too wise and thus it is common to study more specific subclasses and to classify their topological phase portraits. For $n=2$ for example there is a great effort of Artes, Llibre and coauthors~\cites{ArtLli1994,ArtLli1994x2,Cod0,Cod1,Cod2x1,Cod2x2,Cod2x3,Art1,Art2,CaiLli,LliVal2018,Llibre2020,LliRic} to classify as many quadratic vector fields as possible.

In this paper we are concerned with the classification of the \emph{reversible vector fields}. Let $X$ be a $C^k$-vector field (not necessarily planar) and $\varphi\colon\mathbb{R}^m\to\mathbb{R}^m$ a $C^k$-diffeomorphism such that $\varphi^2= \operatorname{Id}_{\mathbb{R}^m}$ (also known as \emph{involution}). Consider also 
	\[\operatorname{Fix}(\varphi)=\{z\in\mathbb{R}^m:\varphi(z)=z\}.\]
We say that $X$ is $\varphi$-reversible of type $(m;s)$ if $\operatorname{Fix}(\varphi)$ is a sub-manifold of $\mathbb{R}^m$ of dimension $s$ and 
	\[D\varphi(z)X(z)=-X(\varphi(z)), \quad \forall \in\mathbb{R}^m,\]
for all $z\in\mathbb{R}^m$. Many reversible vector fields have been studied for several authors. For example in \cite{Buzzi} all the low codimension singularities of systems $(2;0)$-type are classified, with their phase portraits studied at \cite{BuzLliSan}. In \cite{MedTei1998} the low codimension singularities of system of $(3;2)$-type are obtained. In \cites{WeberPessoa2010,WeberPessoa2012} there is a study of the quadratic reversible vector fields of type $(3;2)$ on the sphere $\mathbb{S}^2$. 

In this paper we study the phase portraits of the germs of $C^k$-reversible vector fields of type $(2;1)$. We recall that given a vector field $X$, its \emph{germ} is the equivalence class $[X]$ given by the relation $\sim$ defined by $X\sim Y$ if, and only if, there is a neighborhood $U$ of the origin such that $X$ and $Y$ coincide in $U$. By abuse of notation we denote the germ of $X$ also by $X$.

Let $\mathbb{R}^2,0$ denote plane coordinates in a neighborhood of the origin. Given a $\varphi$-reversible vector field of type $(2;1)$, it is well known that we can choose a coordinate system at $\mathbb{R}^2,0$ such that $\varphi(x,y)=(x,-y)$, see Appendix~\ref{Linearization}. 

Therefore let $\mathfrak{X}$ denote the set of germs of reversible vector fields of type $(2;1)$ at $\mathbb{R}^2,0$ endowed with the $C^k$-topology and observe that given $X\in\mathfrak{X}$, it is no loss of generality to assume that $X$ is $\varphi$-reversible with $\varphi(x,y)=(x,-y)$.

\begin{definition}[Topological equivalence]
	Two germs of vector fields $X$, $Y\in\mathfrak{X}$ are \emph{topologically equivalent} if there are two neighborhoods $U$, $V$ of the origin and a homeomorphism $h\colon U\to V$ which sends orbits of $X$ to orbits of $Y$ preserving or reversing the orientation of all orbits. The homeomorphism $h$ is called a \emph{topological equivalence} betweeen $X$ and $Y$. Moreover we say that $X\in\mathfrak{X}$ is \emph{structurally stable} if there is a neighborhood $N\subset\mathfrak{X}$ of $X$ such that $X$ is topologically equivalent to every $Y\in N$.
\end{definition}

Let $\Sigma_0\subset\mathfrak{X}$ be the family of structurally stable germs and consider the codimension one bifurcation set $\mathfrak{X}_1=\mathfrak{X}\setminus\Sigma_0$, endowed with the subspace topology induced by $\mathfrak{X}$. Let $\Sigma_1\subset\mathfrak{X}_1$ be the family of structurally stable germs relatively to $\mathfrak{X}_1$ and consider the codimension two bifurcation set $\mathfrak{X}_2=\mathfrak{X}_1\setminus\Sigma_1$. Finally, let $\Sigma_2\subset\mathfrak{X}_2$ be the family of structurally stable germs relatively to $\mathfrak{X}_2$.

\begin{definition}[$k$-parameter families]\label{Def2}
	Given $k\in\{1,2\}$ let $\Theta_k$ be the space of $C^1$-mappings $\xi:\mathbb{R}^k\to\mathfrak{X}$ endowed with the $C^1$-topology. Its elements $\xi$ are called \emph{$k$-parameter families} of germs of vector fields of $\mathfrak{X}$. Given $\xi\in\Theta_k$ and $\lambda\in\mathbb{R}^k$, we say that the germ of vector field $\xi(\lambda)\in\mathfrak{X}$ is an \emph{element} of the family $\xi$.
\end{definition}

In the following theorem Teixeira~\cite{Teixeira} provided a characterization of the normal forms of the elements of $\Sigma_0$ and also for dense sets of $\Phi_1$ and $\Phi_2$.

\begin{theorem}[Teixeira~\cite{Teixeira}]\label{T1}
	The following statements hold.
	\begin{enumerate}[label=(\alph*)]
		\item (Codimension zero classification) If $X\in\Sigma_0$, then it is topologically equiva\-lent to one of the following germs of vector fields.
		\begin{enumerate}[label=\arabic*.]
			\item $X_{01}=\left(0,\frac{1}{2}\right)$.
			\item $X_{02}=\left(y,\delta x\right)$, $\delta\in\{-1,1\}$.
		\end{enumerate}
		\item (Codimension one classification) In $\Theta_1$ a dense set if formed by the families whose each element is topologically equivalent to one of the elements of the following families.
		\begin{enumerate}[label=\arabic*.]
			\item The normal forms of item $(a)$.
			\item $X_{11}=\bigl(y,\frac{1}{2}\left(\lambda+x^2\right)\bigr)$.
			\item $X_{12}=\bigl(\delta x y, \frac{1}{2}\bigl(2\delta y^2+x+\lambda\bigr)\bigr)$, $\delta\in\{-1,1\}$.
			\item $X_{13}=\bigl(xy,\frac{1}{2}\left(-y^2+x+\lambda\right)\bigr)$.
			\item $X_{14}=\bigl(xy+y^3,\frac{1}{2}\left(-x+y^2+\lambda\right)\bigr)$.
		\end{enumerate}
		\item (Codimension two classification) In $\Theta_2$ a dense set if formed by the families whose each element is topologically equivalent to one of the elements of the following families.
		\begin{enumerate}[label=\arabic*.]
			\item All the normal forms listed item $(b)$.
			\item $X_{21}=\bigl(y,\frac{1}{2}\bigl(bx^3+\beta x + \alpha\bigr)\bigr)$, $b\in\{-1,1\}$.
			\item \begin{enumerate}[label=(\roman*)]
				\item $X_{22a}=\bigl(ay(x-y^2)+\beta y(x+y^2),\frac{1}{2}(\alpha+(x+y^2)^2)\bigr)$, $a\in\{-1,1\}$.
				\item $X_{22b}=\bigl(y(x-y^2)+\beta y(x+y^2),\frac{1}{2}(\alpha+a(x+y^2)^2)\bigr)$, $a\in\{-1,1\}$.
			\end{enumerate}
			\item $X_{23}=\bigl(-y^3+axy(\alpha+x^2+y^4),\frac{1}{2}(ax-y^2(\alpha+x^2+y^4)+\beta)\bigr)$, $a\in\{-1,1\}$.
			\item $X_{24}=\bigl(axy+\alpha y^3,\frac{1}{2}(x+ay^2+\beta)\bigr)$, $a\in\{-1,1\}$.
			\item \begin{enumerate}[label=(\roman*)]
				\item $X_{25a}=\bigl(xy,\frac{1}{2}(\alpha x-y^2+ax^2+\beta)\bigr)$, $a\in\{-1,1\}$.
				\item $X_{25b}=\bigl(axy,\frac{1}{2}(\alpha x+by^2+\delta x^2+\beta)\bigr)$, $ab>0$, $\delta\in\{-3,3\}$ and $\bigl\{(a\in\{-1,1\},\;b\in\{-3,3\}) \text{ or } (a\in\{-3,3\},\;b\in\{-1,1\})\bigr\}$.
			\end{enumerate}
		\end{enumerate}
	\end{enumerate}
\end{theorem}

The families $X_{ij}$ that appears at Theorem~\ref{T1}$(b)$ and $(c)$ are the topological normal forms that characterizes all possible ways to cross $\Sigma_1$ and $\Sigma_2$ transversely in the space $\mathfrak{X}$ of all germs. In particular the dense set of families stated at Theorem~\ref{T1} are those characterized by the fact that whenever it crosses $\mathfrak{X}_k$, then it does transversely and at structurally stable elements.

We note that such families are polynomial and thus we are motivated to obtain the bifurcation diagrams and the phase portraits of such vector fields in order to complement the results of Teixeira~\cite{Teixeira}. Therefore our first main result is the following.

\begin{main}\label{Main}
	The bifurcation diagram and the phase portraits in the Poincar\'e disk of the vector fields in Theorem~\ref{T1} are given in Figures~\ref{cod0and1}-\ref{23b2}.
\end{main}

We observe that the bifurcation diagram and phase portraits presented at Theorem~\ref{Main} are \emph{global} and defined for all values of $\lambda\in\mathbb{R}$ and $(\alpha,\beta)\in\mathbb{R}^2$.

\begin{definition}
	Let $X$ be a $\varphi$-reversible vector field of type $(m;s)$. We say that $p$ is a symmetric singularity of $X$ if $\varphi(p)=p$.
\end{definition}

It follows from \cite{Teixeira} that the families $X_{ij}$ that appear at Theorem~\ref{T1}$(b)$ and $(c)$ are also the topological normal forms of the symmetric singularities of codimension one and two.

Therefore we observe that if we restrict the phase portraits of Theorem~\ref{Main} to a neighborhood of the origin and for $\lambda\approx0$ or $(\alpha,\beta)\approx(0,0)$, then we have a characterization of the local phase portraits of the symmetric singularities of low codimension of reversible vector fields of type $(2;1)$. Therefore our second main result is the following.

\begin{main}\label{Main2}
	A symmetric singularity of codimension zero, one or two of a reversible vector field of type $(2;1)$ is topologically equivalent to one of the phase portraits given in Figures~\ref{cod0and1}-\ref{23b2}, at the origin and for $\lambda\approx0$ or $(\alpha,\beta)\approx(0,0)$.
\end{main}

We observe that the phase portraits given in Figures~\ref{cod0and1}-\ref{23b2} does not represent the global phase portraits of all reversible vector fields of type $(2;1)$ of low codimension. In fact, unless we are given a maximum degree, a reversible vector field of type $(2;1)$ may have as many singularities as desired and thus there is infinitely many of such phase portraits. For a classification of all the phase portraits of the quadratic reversible vector fields of type $(2;1)$ we refer to~\cite{LlibreMedrado}.

\begin{remark}
	The phase portraits work as follows. The thicker lines represent the separatrices of the phase portrait and the thin lines represent generic orbits of the canonical regions (see Appendix~\ref{MN}). The bigger dots represent isolated singularities. The dotted line at some phase portraits of the bifurcation diagram of $X_{12}$ and $X_{24a}$ represents a line of singularities. Moreover if a phase portrait at the bifurcation diagram has no pointing lines, then it represents the open connected region in which it is inserted.
\end{remark}

\begin{remark}
	For simplicity we call the origin of the charts $U_2$, $V_2$, $U_1$ and $V_1$ of the Poincar\'e compactification (see Appendix~\ref{PC}) as \emph{north pole}, \emph{south pole}, \emph{east pole} and \emph{west pole}, respectively.
\end{remark}

The paper is organized as follows. In Section~\ref{sec2} we present the backbone of the approaching method. In Sections~\ref{sec3}-\ref{sec6} we study the main peculiarities of the families of codimension two. In the Appendix we have a brief survey on some technical results used in the proves of our results.

\section{Approaching Method}\label{sec2}

The backbone of our approach is first to study all the local phase portrait at every finite singularity and at the infinity (see Appendix~\ref{PC}), then to discover when it may exist some limit cycle and finally to study when it may happen some bifurcation as the saddle-node, the center-focus, Hopf bifurcation and etc. After this we study the $\alpha$ and $\omega$-limits of all the separatrices (see Appendix~\ref{MN}). To do this we use some convenient straight lines and curves to see at which direction the flow crosses it. 

Let us take system $X_{12}$ as an example. We remember it is given by 
	\[\dot x=\delta xy, \quad \dot y=\frac{1}{2}\left(2\delta y^2+x+\lambda\right),\]
with $\delta\in\{-1,1\}$. First we observe that if $\delta=-1$, then with the change of variables and parameter $(x,y;\lambda)\mapsto(-x_1,-y_1;-\lambda_1)$ we return to the case $\delta=1$ and thus both cases are equivalent. Therefore it is no loss of generality to assume $\delta=1$. Moreover we observe that the only possible finite singularities are given by 
	\[p=(-\lambda,0), \quad q^\pm=\pm\left(0,\sqrt{-\frac{\lambda}{2}}\right).\]
	
Therefore as anticipated in the begining of this section, our first step is to obtain the local phase portraits at the finite singularities of $X_{12}$.

\begin{proposition}\label{X12finite}
	The following statements hold.
	\begin{enumerate}[label=(\alph*)]
		\item $p$ is a hyperbolic saddle if $\lambda<0$ and a center if $\lambda>0$.
		\item $q^+$ (resp. $q^-$) is a unstable node (resp. stable node) if $\lambda<0$.
		\item The origin is the unique finite singularity if $\lambda=0$ and its phase portrait is given by Figure~\ref{12origin}.
		\item $X_{12}$ does not have limit cycles.
	\end{enumerate}
	\begin{figure}[ht]	
		\begin{center}
			\includegraphics[height=3cm]{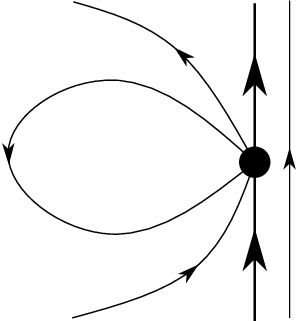}
		\end{center}	
		\caption{Local phase portrait of $X_{12}$ at the origin with $\lambda=0$.}\label{12origin}
	\end{figure}
\end{proposition}

\begin{proof} Knowing that a symmetric singularity cannot be a focus (see~\cite[Section~$2$]{Teixeira2}), statements $(a)$ and $(b)$ follows from the linear parts of $X_{12}$ at the singularities $p$ and $q^\pm$,
	\[DX_{12}(p)=\left(\begin{array}{cc} 0 & -\lambda \\ \frac{1}{2} & 0 \end{array}\right), \quad DX_{12}(q^\pm)=\pm\left(\begin{array}{cc} \frac{1}{\sqrt{2}}\sqrt{-\lambda} & 0 \\ \star & \sqrt{2}\sqrt{-\lambda} \end{array}\right),\]
in addition with the Stable Manifold Theorem, see \cite[Sec~$2.7$]{Perko}.

Doing a quasihomogeneous blow up at the origin when $\lambda=0$ with weights $(\alpha,\beta)=(2,1)$ (see Appendix~\ref{BlowUp}), we obtain the vector field $X_0=X_0(r,\theta)$ given by,
	\[\dot r = rR_1(r,\theta), \quad \dot \theta = \cos\theta\left(\cos\theta+\sin^2\theta\right)+rR_2(r,\theta).\]
Therefore the singularities of $X_0$ such that $r=0$ are given by $\theta\in\left\{\pm\frac{1}{2}\pi,\theta^\pm\right\}$, where $\theta^\pm=\pm\arccos\left(\frac{1}{2}\left(1-\sqrt{5}\right)\right)$. Hence statement $(c)$ follows from
\[\begin{array}{l}
	DX_0\left(0,\pm\frac{\pi}{2}\right)=\pm\left(\begin{array}{cc} 1 & 0 \\ 0 & -1 \end{array}\right), \vspace{0.2cm} \\
	DX_0(0,\theta^\pm)=\pm\left(\begin{array}{cc} \frac{1}{2}\sqrt{5\left(\sqrt{5}-2\right)} & 0 \\ 0 & \sqrt{5\left(\sqrt{5}-2\right)} \end{array}\right).
\end{array}\]
Statement $(d)$ follows from the invariance of the $y$-axis and the fact that a limit cycle cannot surround only a unique saddle, see Appendix~\ref{index}. \end{proof}

Now the next step is to find the local phase portrait at the infinity of the Poincar\'e compactification of $X_{12}$, see Appendix~\ref{PC}.

\begin{proposition}\label{X12infinite}
	The infinity of $X_{12}$ is filled up of singularities and its local phase portrait is given by Figure~\ref{12infinity0}.
	\begin{figure}[ht]	
		\begin{center}
				\includegraphics[width=3cm]{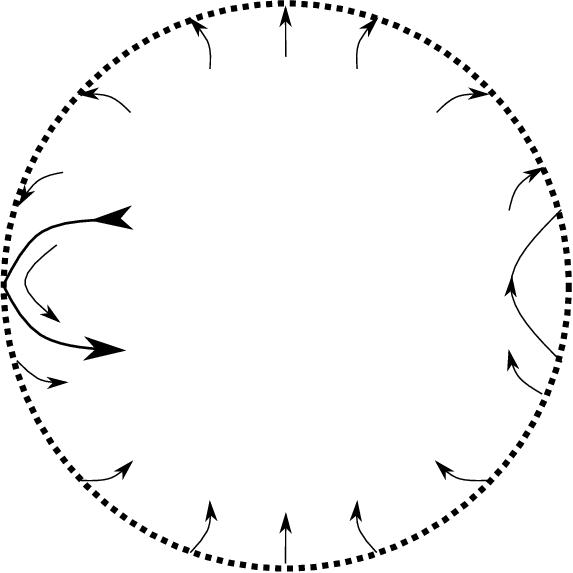}
		\end{center}	
		\caption{Local phase portrait at the infinity of $X_{12}$.}\label{12infinity0}
	\end{figure}
\end{proposition}

\begin{proof} In the chart $U_1$ of the Poincar\'e compactification system $X_{12}$ becomes
	\[\dot u = \frac{1}{2}v+\frac{\lambda}{2}v^2, \quad \dot v = -uv.\]
And in the chart $U_2$ we have
	\[\dot u = -\frac{1}{2}u^2 v -\frac{\lambda}{2}u v^2, \quad \dot v = -v-\frac{1}{2}u v^2-\frac{\lambda}{2}v^3.\]
Therefore we conclude that the infinity is filled up of singularities. Doing a regu\-larization in $v$, i.e. dividing both systems by $v$, one can see that the flow of the regularized system at $v=0$ is given by Figure~\ref{12infinity}. More precisely, let $X=X(u,v)$ be the regularized version of system $X_{12}$ at the chart $U_1$, i.e. $X$ is given by
	\[\dot u = \frac{1}{2}+\frac{\lambda}{2}v, \quad \dot v = -u.\]
Let $h(u,v)=v$, $Xh(p)=\left<X(p),\nabla h(p)\right>$ and $X^nh(p)=\left<X(p),\nabla X^{n-1}h(p)\right>$, where $\left<\cdot,\cdot\right>$ denotes the standard inner product of $\mathbb{R}^2$. One can conclude that $h(0,0)=Xh(0,0)=0$, $X^2h(0,0)=-\frac{1}{2}$, and thus obtain Figure~\ref{12infinity}.
\begin{figure}[ht]	
		\begin{center}
			\includegraphics[width=3cm]{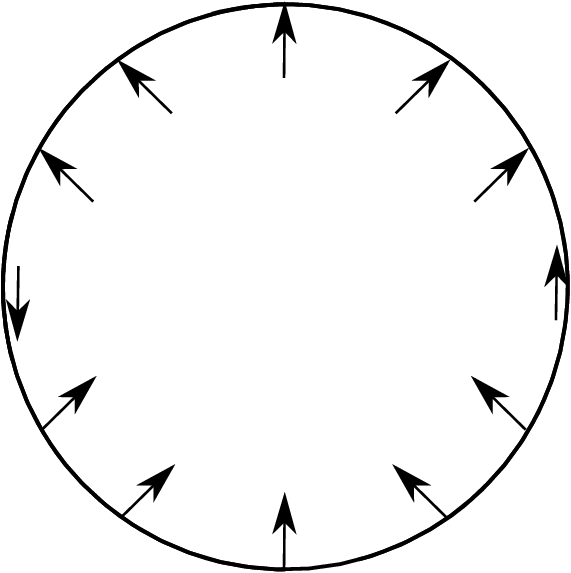}
		\end{center}	
	\caption{Local phase portrait of the regularized infinity of $X_{12}$.}\label{12infinity}
\end{figure} \end{proof}

Now we have the local behavior of $X_{12}$ in all its singularities and we know that it does not have limit cycles. Hence we can start working on its separatrices. 

\begin{proposition}
	The phase portrait of $X_{12}$ for $\lambda<0$ is the one given in Figure~\ref{cod0and1}.
\end{proposition}

\begin{proof} From Propositions~\ref{X12finite} and \ref{X12infinite} we conclude Figure~\ref{Field12Local1}.
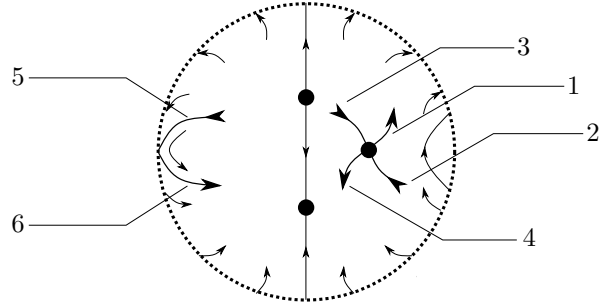
\begin{figure}[ht]	
	\begin{center}
		\begin{overpic}[height=4cm]{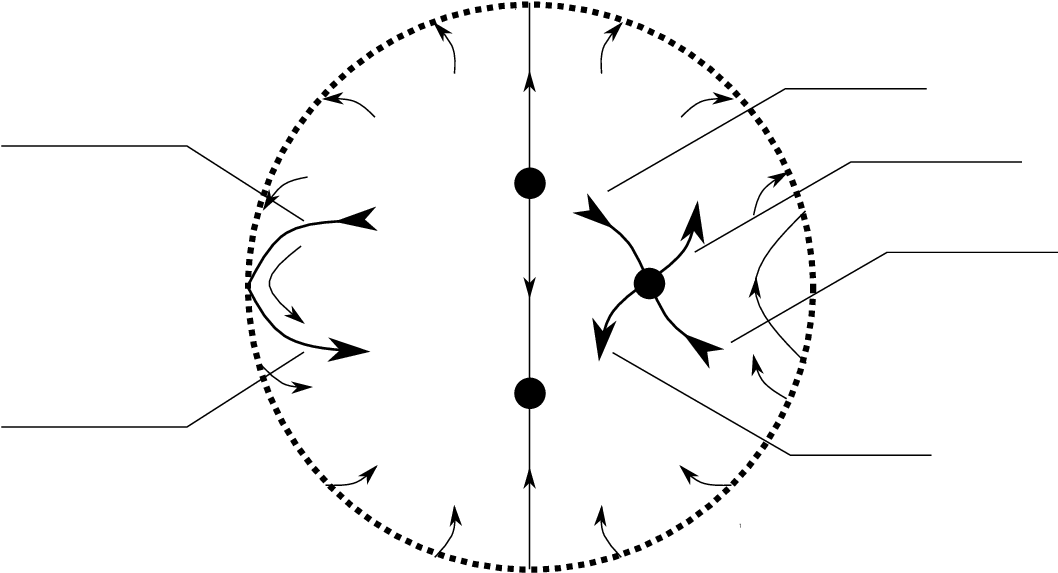} 
			\put(97,37.5){$1$}
			\put(100.5,29){$2$}
			\put(88,44.5){$3$}
			\put(89,10){$4$}
			\put(-3,39){$5$}
			\put(-3,12.5){$6$}	
		\end{overpic}
	\end{center}	
	\caption{Local behavior of the phase portrait of $X_{12}$ for $\lambda<0$.}\label{Field12Local1}
\end{figure}
First we observe that the flow crosses the $x$-axis downwards if $x<-\lambda$ and upwards if $x>-\lambda$. Therefore, separatrix $4$ cannot cross the $x$-axis, otherwise separatrix $2$ would have no $\alpha$-limit. Hence, separatrix $4$ has no other option than ending at the stable node $q^-$. By symmetry separatrix $3$ borns at the unstable node. Now separatrix $1$ has no other option than ending at some singularity of the first quadrant of the infinity and thus separatrix $2$ borns at the symmetric singularity at the forth quadrant. Clearly separatrix $6$ has no option than ending at the stable node and therefore separatrix $5$ borns at the unstable node. \end{proof}

\begin{proposition}
	The phase portrait of $X_{12}$ for $\lambda=0$ is the one given in Figure~\ref{cod0and1}.
\end{proposition}

\begin{proof} From Propositions~\ref{X12finite} and \ref{X12infinite} we conclude Figure~\ref{Field12Local2}.
\begin{figure}[ht]	
	\begin{center}
		\begin{overpic}[height=4cm]{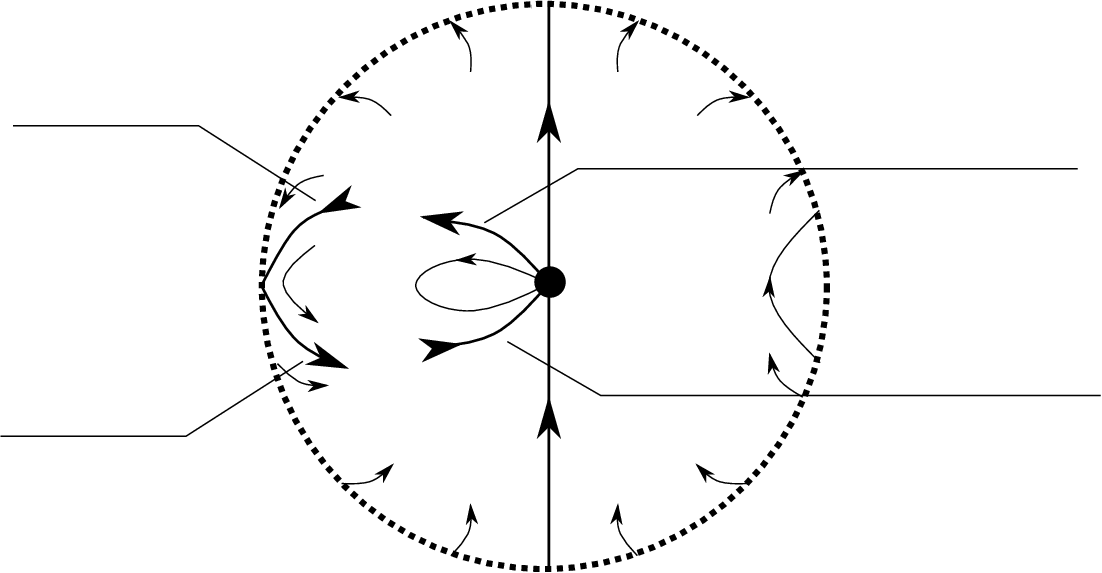} 
			\put(98,35){$1$}
			\put(101,15){$2$}
			\put(-2,39){$3$}
			\put(-3,11){$4$}
		\end{overpic}
	\end{center}	
	\caption{Local behavior of the phase portrait of $X_{12}$ for $\lambda=0$.}\label{Field12Local2}
\end{figure}
We know that the flow crosses the $x$-axis downwards if $x<0$, therefore separatrix $4$ has no other option than glue up to separatrix $2$ and thus by symmetry separatrix $1$ glues up to separatrix $3$. \end{proof}

\begin{proposition}
	The phase portrait of $X_{12}$ for $\lambda>0$ is the one given in Figure~\ref{cod0and1}.
\end{proposition}

\begin{proof} From Propositions~\ref{X12finite} and \ref{X12infinite} we conclude Figure~\ref{Field12Local3}.
\begin{figure}[ht]	
	\begin{center}
		\begin{overpic}[height=4cm]{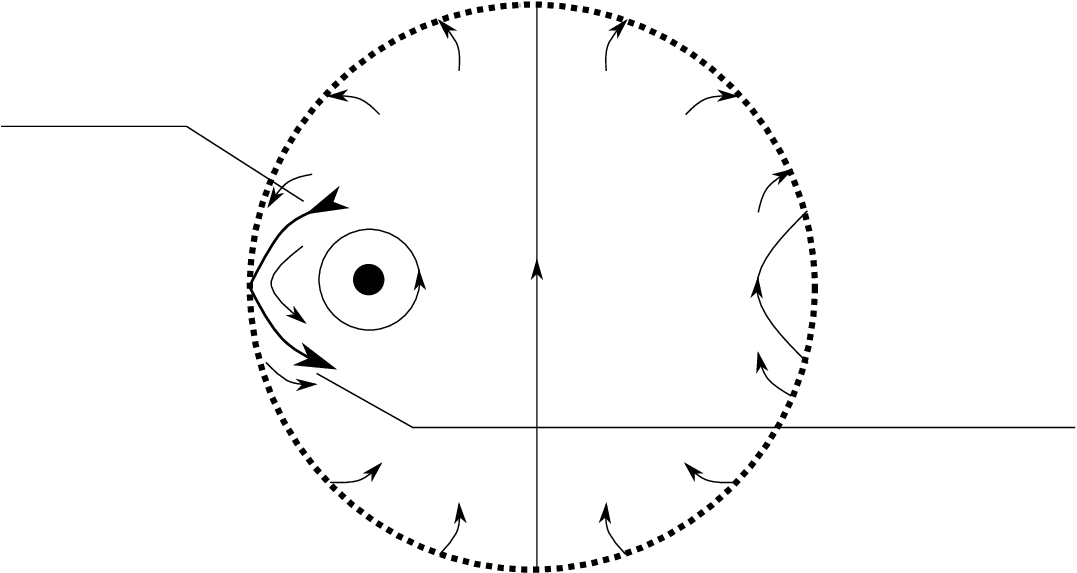} 
			\put(-3,40){$2$}
			\put(101,12){$1$}
		\end{overpic}
	\end{center}	
	\caption{Local behavior of the phase portrait of $X_{12}$ for $\lambda>0$.}\label{Field12Local3}
\end{figure}
Here we just need to observe that separatrix $1$ must cross the $x$-axis and thus by symmetry it must glue to separatrix $2$. \end{proof}

With this same approach one can also conclude the phase portraits of $X_{01}$, $X_{02}$, $X_{11}$, $X_{13}$ and $X_{14}$.

Before we study the systems of codimension $2$ we must know when they are topologically equivalent, at some finite singularity, to the origin of $X_{11}$, $X_{12}$, $X_{13}$ or $X_{14}$, for $\lambda\approx0$. To this end we use a characterization provided by Teixeira~\cite{Teixeira}.

Let $h\colon\mathbb{R}^2\to\mathbb{R}$ be given by $h(x,y)=y$ and recall that the $n$-th \emph{Lie derivative} of $h$ at $p$ in the direction of a vector field $X$ is given by,
	\[Xh(p)=\left<X(p),\nabla h(p)\right>, \quad X^nh(p)\left<X(p),\nabla X^{n-1}h(p)\right>,\]
where $\left<\cdot,\cdot\right>$ denotes the standard inner product of $\mathbb{R}^2$.

\begin{definition}\label{Def4}
   Let $p$ be a symmetric singularity ($S$-singularity, for short) of $X\in\mathfrak{X}$. We say that $p$ is a:
	\begin{enumerate}[label=(\alph*)]
		\item \emph{cusp} $S$-singularity if $h(p)=Xh(p)=X^2h(p)=0$ and $X^3h(p)\neq0$.
		\item \emph{nodal} $S$-singularity if $X(p)=0$, the eigenvalues $\lambda_1$ and $\lambda_2$ of $DX(p)$ are real, distinct, $\lambda_1\lambda_2>0$ and its eigenspaces are transverse to $y=0$ at $p$.
		\item \emph{saddle} $S$-singularity if $X(p)=0$, the eigenvalues $\lambda_1$ and $\lambda_2$ of $DX(p)$ are real, distinct, $\lambda_1\lambda_2<0$ and its eigenspaces are transverse to $y=0$ at $p$.
		\item \emph{focal} $S$-singularity if $p$ is a hyperbolic singularity of $X$ and the eigenvalues of $DX(p)$ are given by $\lambda=a+ib$ with $b\neq0$.
	\end{enumerate}
\end{definition}

It follows from Lemma $4.2$ of \cite{Teixeira} that if $p$ is a cusp, nodal, saddle or focal $S$-singularity, then $X$ is locally topologically equivalent at $p$ to the origin of $X_{11}$, $X_{12}$, $X_{13}$ or $X_{14}$, respectively, for $\lambda\approx0$.

\section{System $X_{21}$}\label{sec3}

Let us remember that system $X_{21}$ is given by
	\[\dot x = y, \quad \dot y = \frac{1}{2}\left(bx^3+\beta x + \alpha\right),\]
with $b\in\{-1,1\}$. From now on we assume $b=1$. Clearly the finite singularities of $X_{21}$ are given by $(x_0,0)$, where $x_0$ is a real root of the polynomial $p(x)=x^3+\beta x+\alpha$. Following the Cardano-Tartaglia formula (see Appendix~\ref{TGSection}) we define $D=\frac{1}{27}\beta^3+\frac{1}{4}\alpha^2$ and observe that $D=0$ if, and only if, $\beta=\beta(\alpha)$ where
	\[\beta(\alpha)=-\frac{3}{\sqrt[3]{4}}\sqrt[3]{\alpha^2}.\]
Therefore from Appendix~\ref{TGSection} we conclude the following statements.
\begin{enumerate}[label=(\alph*)]
	\item If $\beta>\beta(\alpha)$ we have a unique finite singularity.
	\item If $\beta<\beta(\alpha)$ we have three finite singularities.
	\item If $\beta=\beta(\alpha)$ and $\alpha\neq0$ we have two finite singularities.
	\item If $\beta=\beta(\alpha)$ and $\alpha=0$, then the origin is the unique finite singularity.
\end{enumerate}
From Appendix~\ref{TGSection} we have that we can denote the zeros of $p(x)=0$ by 
	\[\begin{array}{l}
		\displaystyle x_1=S+T, \vspace{0.2cm} \\
		\displaystyle x_2=-\frac{1}{2}(S+T)+\frac{1}{2}\sqrt{3}(S-T)i, \vspace{0.2cm} \\
		\displaystyle x_3=-\frac{1}{2}(S+T)-\frac{1}{2}\sqrt{3}(S-T)i,
	\end{array}\]
where $i^2=-1$ and
	\[S=f\left(-\frac{\alpha}{2}+D^\frac{1}{2}\right), \quad T=f\left(-\frac{\alpha}{2}-D^\frac{1}{2}\right), \quad f(x)=\begin{cases} \sqrt[3]{x}, \text{ if } D\geqslant0, \\ x^\frac{1}{3}, \text{ if } D<0. \end{cases}\]
	
With this machinery we can now study the relative position of the real solutions of the polynomial $p(x)=x^3+\beta x+\alpha$.

\begin{proposition}
	The relative position of the real solutions of $p$ is the following.
	\begin{enumerate}[label=(\alph*)]
		\item If $D<0$ then $x_2<x_3<x_1$.
		\item If $D=0$ and $\alpha<0$, then $x_2=x_3<x_1$.
		\item If $D=0$ and $\alpha>0$, then $x_1<x_2=x_3$.
		\item Otherwise $x_1$ is the unique real solution.
	\end{enumerate}
\end{proposition}
\begin{proof} We start with statement $(a)$, which we separate in three parts. If $D<0$ and $\alpha<0$, then $S=\left(a+ib\right)^\frac{1}{3}$, where $a=-\frac{1}{2}\alpha>0$ and $b=\sqrt{-D}>0$. Therefore if we denote $r=\sqrt{a^2+b^2}$ and $\theta=\frac{1}{3}\arctan\left(\frac{b}{a}\right)$ we obtain $S=\sqrt{-\frac{\beta}{3}}(\cos\theta+i\sin\theta)$. In a similar way one can see that $T=\sqrt{-\frac{\beta}{3}}(\cos\theta-i\sin\theta)$. Hence,
	\[\begin{array}{l}
		x_1=2\sqrt{-\frac{\beta}{3}}\cos\theta, \\ 			x_2=-\sqrt{-\frac{\beta}{3}}(\cos\theta+\sqrt{3}\sin\theta), \\
		x_3=-\sqrt{-\frac{\beta}{3}}(\cos\theta-\sqrt{3}\sin\theta).
	\end{array}\]
Observe that 
	\[\frac{b}{a}=\frac{\sqrt{-D}}{-\frac{1}{2}\alpha}=-\frac{2}{\alpha}\sqrt{-\frac{\beta^3}{27}-\frac{\alpha^2}{4}}=\frac{2}{|\alpha|}\sqrt{-\frac{\beta^3}{27}-\frac{\alpha^2}{4}}=\sqrt{\frac{4}{27}\frac{(-\beta)^3}{\alpha^2}-1}.\]
Given $\beta_0<0$ fixed we know that $\alpha\in(\alpha(\beta_0),0)$, where $\alpha(\beta)=-\sqrt{-\frac{4}{27}\beta^3}$. Similarly, given $\alpha_0<0$ fixed we know that $\beta\in(-\infty,\beta(\alpha))$. Therefore we have that $\theta=\theta(\alpha,\beta)\in\left(0,\frac{1}{6}\pi\right)$ and 
	\[\lim\limits_{D\rightarrow0}\theta(\alpha,\beta)=0, \quad \lim\limits_{\beta\rightarrow-\infty}\theta(\alpha_0,\beta)= \lim\limits_{\alpha\rightarrow0}\theta(\alpha,\beta_0)=\frac{\pi}{6}.\]
Hence we conclude that in this case we have $x_2<x_3<0<x_1$ and also,
	\[\lim\limits_{D\rightarrow0}|x_2-x_3|=0, \quad \lim\limits_{\beta\rightarrow-\infty}x_3(\alpha_0,\beta)=\lim\limits_{\alpha\rightarrow0}x_3(\alpha,\beta_0)=0.\]
If $D<0$ and $\alpha>0$ we have 
	\[\begin{array}{l}x_1=2\sqrt{-\frac{\beta}{3}}\cos\theta, \\
		x_2=-\sqrt{-\frac{\beta}{3}}(\cos\theta+\sqrt{3}\sin\theta), \\
		x_3=-\sqrt{-\frac{\beta}{3}}(\cos\theta-\sqrt{3}\sin\theta),
	\end{array}\]
where $\theta=\frac{1}{3}\arctan\left(-\frac{2}{\alpha}\sqrt{-D}\right)+\frac{1}{3}\pi$. Since we have $\alpha>0$ it follows that 
	\[-\frac{2}{\alpha}\sqrt{-D}=-\sqrt{\frac{4}{27}\frac{(-\beta)^3}{\alpha^2}-1}.\]
Given $\beta_0<0$ fixed we know that $\alpha\in(0,\alpha(\beta_0))$, where $\alpha(\beta)=\sqrt{-\frac{4}{27}\beta^3}$. Similarly, given $\alpha_0>0$ fixed we know that $\beta\in(-\infty,\beta(\alpha))$. Therefore it follows that $\theta=\theta(\alpha,\beta)\in\left(\frac{1}{6}\pi,\frac{1}{3}\pi\right)$ and 
	\[\lim\limits_{D\rightarrow0}\theta(\alpha,\beta)=\frac{\pi}{3}, \quad \lim\limits_{\beta\rightarrow-\infty}\theta(\alpha_0,\beta)= \lim\limits_{\alpha\rightarrow0}\theta(\alpha,\beta_0)=\frac{\pi}{6}.\]
Hence we conclude that in this case we have $x_2<0<x_3<x_1$ and also,
	\[\lim\limits_{D\rightarrow0}|x_1-x_3|=0, \quad \lim\limits_{\beta\rightarrow-\infty}x_3(\alpha_0,\beta)=\lim\limits_{\alpha\rightarrow0}x_3(\alpha,\beta_0)=0.\]
If $D<0$ and $\alpha=0$, then
	\[S=\frac{1}{2}\sqrt{-\frac{\beta}{3}}\left(\sqrt{3}+i\right), \quad  T=\frac{1}{2}\sqrt{-\frac{\beta}{3}}\left(\sqrt{3}-i\right),\]
and thus
	\[x_1=\sqrt{-\beta}, \quad x_2=-\sqrt{-\beta}, \quad x_3=0.\]
Moreover, observe that in this case we have
	\[\lim\limits_{D\rightarrow0}x_1(\alpha,\beta)=\lim\limits_{D\rightarrow0}x_2(\alpha,\beta)=0.\]
Hence one can conclude statement $(a)$.

If $D=0$ then $S=T=\sqrt[3]{-\frac{1}{2}\alpha}$ and thus 
\begin{equation}\label{1}
	x_1 = 2\sqrt[3]{-\frac{1}{2}\alpha}, \quad x_2=x_3=\sqrt[3]{\frac{1}{2}\alpha}.
\end{equation}
Statements $(b)$ and $(c)$ follow from \eqref{1} if $\alpha\neq0$. If $\alpha=0$ one obtain $x_1=x_2=x_3$ as in statement $(d)$.

Finally if $D>0$ or $D=\alpha=0$, then the unique real zero is given by
	\[x_1=\sqrt[3]{-\frac{\alpha}{2}+\sqrt{D}}+\sqrt[3]{-\frac{\alpha}{2}-\sqrt{D}},\]
as in statement $(d)$. 	\end{proof}

Now that we have information about the real roots of $p(x)=x^3+\beta x+\alpha$ we study the phase portraits of the singularities of $X_{21}$.

\begin{proposition}
	Let $p_i=(x_i,0)$, $i\in\{1,2,3\}$ be the singularities of $X_{21}$. The local phase portraits of $X_{21}$ at these singularities are the following.
	\begin{enumerate}[label=(\alph*)]
		\item If $D>0$, then $p_1$ is a saddle.
		\item If $D<0$, then $p_1$ and $p_2$ are saddles and $p_3$ is a center.
		\item If $D=0$ and $\alpha\neq0$, then $p_1$ is a saddle and $p_2=p_3$ is a cusp, as in Figure~\ref{21D0}.
		\item If $D=\alpha=0$, then $p_1$ is a non-hyperbolic saddle.
	\end{enumerate}
\begin{figure}[ht]
	\begin{multicols}{2}		
		\begin{center}
			\includegraphics[height=3cm]{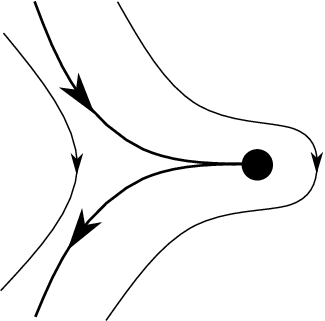}
			
			$\alpha<0$
		\end{center}		
		\columnbreak		
		\begin{center}
			\includegraphics[height=3cm]{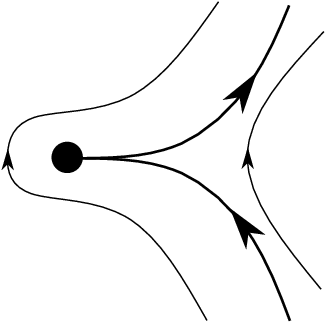}
			
			$0<\alpha$
		\end{center}				
	\end{multicols}
	\caption{Local phase portrait of $X_{21}$ at $p_2=p_3$ when $D=0$.}\label{21D0}
\end{figure}
\end{proposition}

\begin{proof} Observe that the Jacobian matrix at a singularity $p_i$ is given by 
	\[DX_{21}(p_i)=\left(\begin{array}{cc} 0 & 1 \\ \frac{1}{2}\beta+\frac{3}{2}x_i^2 & 0 \end{array}\right).\]

We first work statement $(a)$. In this case $p_1$ is the unique finite singularity. Clearly it is a saddle if $\beta\geqslant0$. If $\beta<0$ then
	\[\frac{1}{2}\beta+\frac{3}{2}x_1^2=\frac{3}{2}(S^2+T^2)-\frac{1}{2}\beta>0,\]
with the last inequality due to the fact that $\beta<0$. In particular if $\beta<0$ then $p_1$ is still a saddle. 

Let us now work statement $(b)$. Since $\beta<0$ is a necessary conditions for $D$, we have that
	\[\frac{1}{2}\beta+\frac{3}{2}x_1^2=-\beta\left(2\cos^2\theta-\frac{1}{2}\right)>0.\]
Therefore $p_1$ is a hyperbolic saddle. Observe now that 
	\[\frac{1}{2}\beta+\frac{3}{2}x_2^2=-\beta\sin\theta(\sin\theta+\sqrt{3}\cos\theta)>0,\]
because $\beta<0$ and $\theta\in\left(0,\frac{1}{3}\pi\right)$. Therefore, $p_2$ is also a hyperbolic saddle. Finally, 
	\[\frac{1}{2}\beta+\frac{3}{2}x_3^2=-\beta\sin\theta(\sin\theta-\sqrt{3}\cos\theta)<0,\]
because $\beta<0$ and $\theta\in\left(0,\frac{1}{3}\pi\right)$. This, in addition with the fact that $p_3$ is a symmetric singularity, ensures that it is a center (see~\cite[Section $2$]{Teixeira2}).

We now work statement $(c)$. Observe that
	\[\frac{1}{2}\beta+\frac{3}{2}x_1^2=\frac{9}{2}\sqrt[3]{\frac{\alpha^2}{4}}>0, \quad \frac{1}{2}\beta+\frac{3}{2}x_2^2=0,\]
because $\beta=-3\sqrt[3]{\frac{\alpha^2}{4}}$. Therefore, $(x_1,0)$ is a hyperbolic saddle and $(x_2,0)$ is degenerated. Translating $(x_2,0)$ to the origin and then doing a quasihomogenous blow up with weight $(2,3)$ we obtain Figure~\ref{21D0}. 

Finally, assuming the hypothesis of statement $(d)$ we obtain $\alpha=\beta=0$ and then the origin is the unique finite singularity and it is clearly a non-hyperbolic saddle. \end{proof}

We now study the phase portraits of $X_{21}$.

\begin{proposition}
	The phase portrait of $X_{21}$ with $b=1$ is the one given in Figure~\ref{21a}.
\end{proposition}

\begin{proof} First we observe that $X_{21}$ cannot have any limit cycle because every finite singularity is symmetric. Using the same approach as in Section~\ref{sec2} one can conclude the phase portraits of $X_{21}$ when $D\geqslant0$. Knowing that $X_{21}$ is also $\varphi$-reversible with $\varphi(x,y)=(-x,y)$ when $\alpha=0$ one can conclude its phase portrait when $D<0$ and $\alpha=0$. Knowing that singularity $p_2=p_3$ is a \emph{cusp $S$-singularity} of $X_{21}$ when $D=0$ and $\alpha\neq0$ one can also conclude the phase portrait for $D<0$, but near the boundary $D=0$. Now we must prove that this last phase portrait holds for any $(\alpha,\beta)$ such that $D<0$ and $\alpha\neq0$.

If $D<0$ we know that we have three finite singularities, a center $p_3$ in the middle and a hyperbolic saddle $p_2$ on its left side and another hyperbolic saddle $p_1$ on its right side. Let $\mu_0=(\alpha_0,\beta_0)\in\mathbb{R}^2$ be such that there is a heteroclinic orbit $\Gamma_0$ connecting both hyperbolic saddles, $x_0\in\mathbb{R}^2$ the intersection of $\Gamma_0$ with the $y$-axis and $l_0$ a transversal section of $\Gamma_0$ passing through $x_0$. Following \cite{PerkoPaper} we define $n$ to be the coordinate along the normal line $l_0$ such that $n>0$ outside the polycycle (observe that there is another heteroclinic connection due to the symmetry) and $n<0$ inside the polycycle. We also denote $\Gamma_s$ and $\Gamma_u$ the perturbations of $\Gamma_0$, for $|(\alpha-\alpha_0,\beta-\beta_0)|$ small enough, such that $\Gamma_s$ is contained in the stable manifold of $p_1$ and $\Gamma_u$ is contained in the unstable manifold of $p_2$. Let $x_s$ and $x_u$ be the intersections of $\Gamma_s$ and $\Gamma_u$ with $l_0$, respectively, and let $n_s$ and $n_u$ be its coordinates along the line $l_0$. We now define the displacement map $d(\alpha,\beta)=n_u-n_s$. See Figure~\ref{21Poly}.
\begin{figure}[ht]
	\begin{multicols}{3}
		\begin{center}
			\begin{overpic}[width=3.5cm]{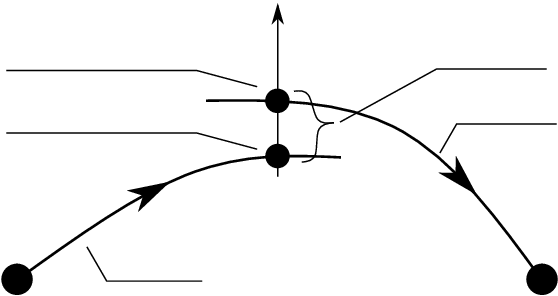} 
				\put(-2,43){$x_s$}
				\put(-3,32){$x_u$}
				\put(80,42){$d<0$}
				\put(37,0){$\Gamma_u$}
				\put(100,27){$\Gamma_s$}
			\end{overpic}
			$(\alpha,\beta)\neq(\alpha_0,\beta_0)$
		\end{center}
		\columnbreak
		\begin{center}
			\begin{overpic}[width=3.5cm]{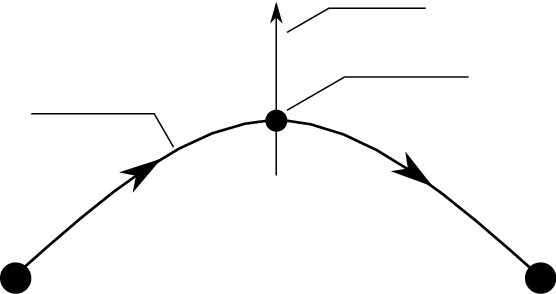} 
				\put(5,35){$\Gamma_0$}
				\put(85,37){$x_0$}
				\put(78,50){$l_0$}
			\end{overpic}
			$(\alpha,\beta)=(\alpha_0,\beta_0)$
		\end{center}
		\columnbreak
		\begin{center}
			\begin{overpic}[width=3.5cm]{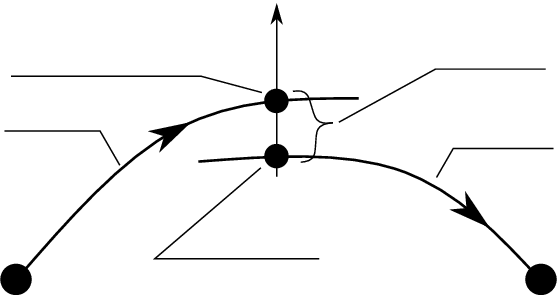} 
				\put(-6,42){$x_u$}
				\put(-7,26){$\Gamma_u$}
				\put(80,42){$d>0$}
				\put(94,28){$\Gamma_s$}
				\put(59,4){$x_s$}
			\end{overpic}
			$(\alpha,\beta)\neq(\alpha_0,\beta_0)$
		\end{center}
	\end{multicols}
	\caption{The displacement map $d(\alpha,\beta)$ defined near $(\alpha_0,\beta_0)$.}\label{21Poly}
\end{figure}
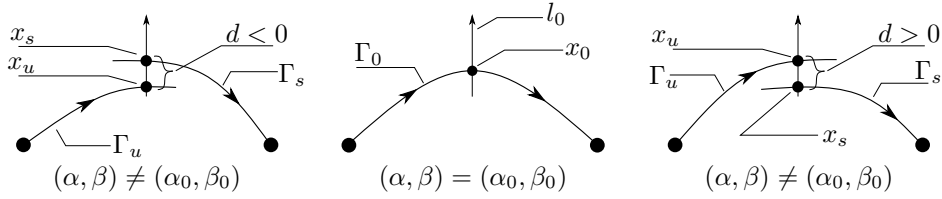

Let $\gamma(t)$ be a parametrization of $\Gamma_0$ with $\gamma(0)=x_0$ and $f(t;\alpha,\beta)=X_{21}(\gamma(t);\alpha,\beta)$. It follows from \cite{PerkoPaper} that
	\[\frac{\partial d}{\partial \alpha}(\mu_0)=\frac{1}{|f(0;\mu_0)|}\int_{-\infty}^{+\infty}\left(e^{-\int_{0}^{t}Div(f(s;\mu_0))\;ds}\right)f(t;\mu_0)\land\dfrac{\partial f}{\partial \alpha}(t;\mu_0)\;dt,\]
where $\mu_0=(\alpha_0,\beta_0)$ and $(x_1,x_2)\land(y_1,y_2)=x_1y_2-x_2y_1$. Knowing that
	\[X_{21}(x,y;\alpha,\beta)\land\frac{\partial X_{21}}{\partial \alpha}(x,y;\alpha,\beta)=\frac{1}{2}y,\]
one can see that $\frac{\partial d}{\partial \alpha}(\mu_0)>0$ whenever $d$ is defined. Hence, given $\beta=\beta_0<0$ fixed and $\alpha_0=\alpha_0(\beta_0)$ such that $d(\alpha_0,\beta_0)=0$ we conclude that $d(\cdot,\beta_0)$ increases for $\alpha\approx\alpha_0$ and thus if $\alpha_0$ exists it is unique. But we already know that $d(0,\beta)=0$ for every $\beta<0$ and therefore given $\beta=\beta_0<0$ fixed it follows that $\alpha=0$ is the unique value that satisfy $d(\alpha_0,\beta_0)=0$. \end{proof}

In a similar way one can prove that the phase portrait of $X_{21}$ with $b=-1$ is the one given by Figure~\ref{21b}.

\section{System $X_{22}$}\label{sec4}

We remember that system $X_{22a}$ is given by
	\[\dot x = (\beta+a)xy+(\beta-a)y^3, \quad \dot y = \frac{\alpha}{2}+\frac{1}{2}x^2+xy^2+\frac{1}{2}y^4,\]
with $a\in\{-1,1\}$. On the other hand, system $X_{22b}$ is given by 
	\[\dot x = (\beta+1)xy+(\beta-1)y^3, \quad \dot y = \frac{\alpha}{2}+\frac{a}{2}x^2+axy^2+\frac{a}{2}y^4,\]
with $a\in\{-1,1\}$. Observe that they are both equal if we replace $a=1$. Moreover, if we replace $a=-1$ on $X_{22b}$ and then apply the change of variables and parameters 
	\[(x,y;\alpha,\beta)\mapsto(x_1,-y_1;-\alpha_1,-\beta_1)\]
we obtain system $X_{22a}$ with $a=-1$. Hence, both systems are equivalent and we focus on $X_{22a}$.

Let $S_x=\{(x,y)\in\mathbb{R}^2:\dot x(x,y)=0\}$ and $S_y=\{(x,y)\in\mathbb{R}^2:\dot y(x,y)=0\}$ and observe that
	\[\begin{array}{l}
		S_x=\{(x,y)\in\mathbb{R}^2:y=0 \text{ or } x=\frac{a-\beta}{a+\beta}y^2\}, \vspace{0.2cm} \\
		S_y=\{(x,y)\in\mathbb{R}^2:x=-y^2\pm\sqrt{-\alpha}\}.
	\end{array}\]

\begin{proposition}
	The phase portrait of $X_{22a}$ with $a=1$ is given by Figure~\ref{22a}.
\end{proposition}

\begin{proof} Here we point out only some particular reasonings. Approaching system $X_{22a}$ as in Section~\ref{sec2} and knowing that the \emph{Bendixson Criterion} prevents the existence of any limit cycles in this case, one can work with the separatrices. But here one must look also at the flows at sets $S_x$ and $S_y$. Take for example Figure~\ref{Field22aLocal1Special}.
\begin{figure}[ht]	
	\begin{center}
	\begin{overpic}[height=5cm]{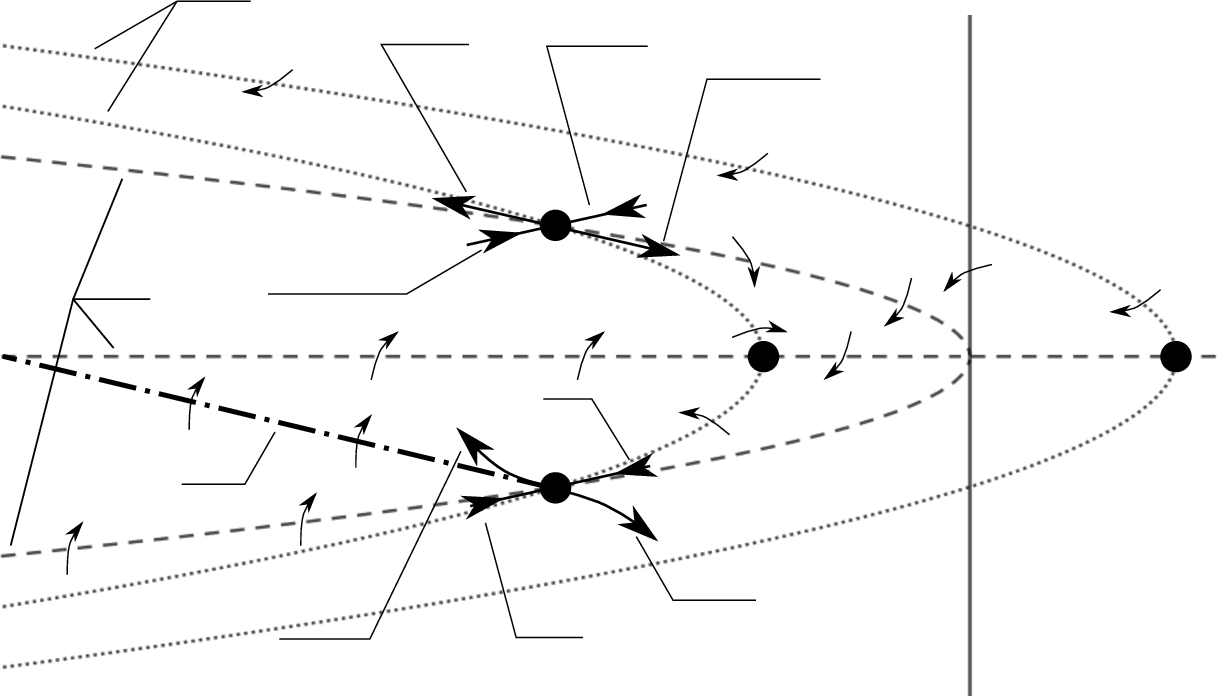} 
			\put(67,49){$1$}
			\put(53,52){$2$}
			\put(39,52){$3$}
			\put(24,33.5){$4$}
			\put(64,7){$5$}
			\put(48.5,3.5){$6$}
			\put(20,3.5){$7$}
			\put(42,23){$8$}
			\put(80,54){$y$}
			\put(13,31){$S_x$}
			\put(21,55.5){$S_y$}
			\put(13,16){$l$}
		\end{overpic}
	\end{center}	
	\caption{Illustration of the flow of $X_{22a}$ with $a=1$ and $\beta<-1$ at the sets $S_x$ and $S_y$.}\label{Field22aLocal1Special}
\end{figure}
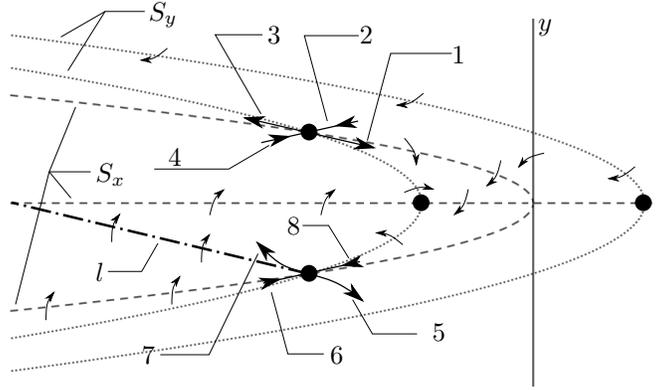
One can obtain it with an analysis of the flow on the sets $S_x$ and $S_y$ together with an analysis at the eigenvectors of the hyperbolic saddles. Now one can conclude that separatrix $1$ must cross the $x$-axis and thus it must glue with separatrix $8$ due to symmetry. At other hand, separatrix $3$ is trapped between the $x$-axis and $S_y$ and thus it must end at infinity. Finally, separatrix $7$ has no other option than either ending at infinity (precisely at the west pole) or crossing the $x$-axis. Let $q=q(\alpha,\beta)$ be the hyperbolic saddle inside the half-plane $y<0$ and $\lambda^-<0<\lambda^+$, $\lambda^\pm=\lambda^\pm(\alpha,\beta)$, the eigenvalues of $DX_{22a}(q)$. Let $v=v(\alpha,\beta)$ be an eigenvector of $DX_{22a}(q)$ with respect to $\lambda^+$, $l(t)=q+tv$ and $t_0=t_0(\alpha,\beta)$ such that $l(t_0)$ is the intersection of $l$ and the $x$-axis. Calculations shows that separatrix $7$ is above $l$ and if $t$ is between $0$ and $t_0$, then the flow of $X_{22a}$ is transversal to $l$ and it points upwards. Therefore separatrix $7$ cannot cross $l$ and thus it must cross the $x$-axis and hence it glues up with serparatrix $4$ due to the symmetry.

The local phase portrait of $X_{22a}$ with $a=1$, $\beta-1$ and $\alpha\leqslant-16$ is shown in Figure~\ref{Field22aLocal10}.
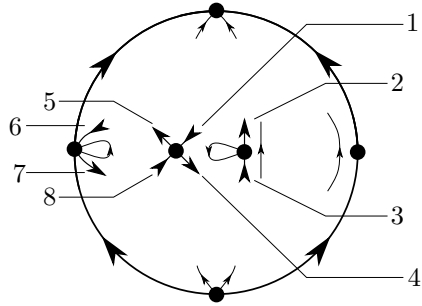
\begin{figure}[ht]	
	\begin{center}
		\begin{overpic}[height=4cm]{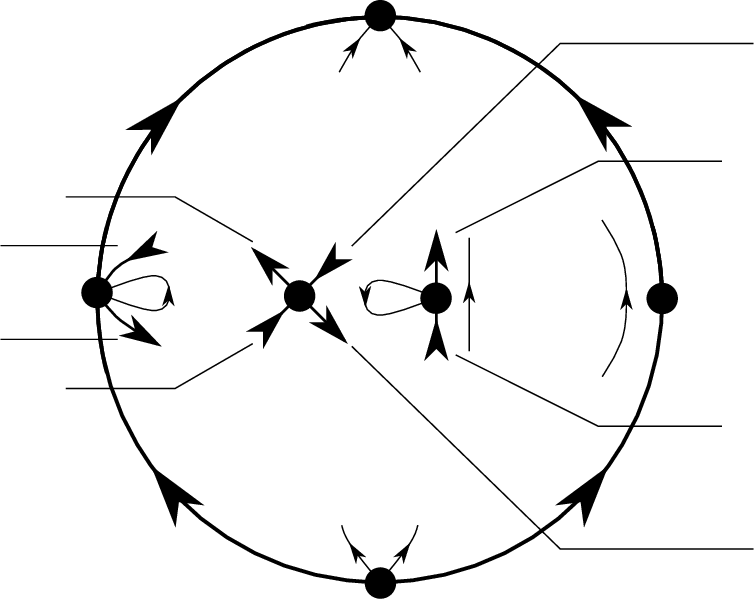} 
			\put(100.5,72){$1$}
			\put(96.5,56){$2$}
			\put(96.5,21){$3$}
			\put(101,4.5){$4$}
			\put(4.5,51){$5$}
			\put(-4.5,44.5){$6$}
			\put(-3.5,32){$7$}
			\put(4.5,26){$8$}
		\end{overpic}
	\end{center}	
	\caption{Local phase portrait of $X_{22a}$ with $a=1$, $\beta-1$ and $\alpha\leqslant-16$.}\label{Field22aLocal10}
\end{figure}
To prove that separatrix $2$, which born at the singularity $p=(\sqrt{-\alpha},0)$, ends at the north pole we calculate the flow on the parabola $x=-\frac{1}{2}y^2+\sqrt{-\alpha}$ and observe that it points upwards if $y\neq0$. We also calculate that separatrix $2$ near $p$ is given by $(f(y),y)$, where
	\[f(y)=-\frac{\sqrt{-4\sqrt{-\alpha}-\alpha}+\alpha}{2\alpha}y^2+O(y^4).\]
Knowing that
	\[-\frac{1}{2}\leqslant-\frac{\sqrt{-4\sqrt{-\alpha}-\alpha}+\alpha}{2\alpha},\]
with the right hand side equals $-\frac{1}{2}$ if and only if $\alpha=-16$, one can concludes that if $\alpha<-16$, then separatrix $2$ ends at the north pole. Moreover, if $\alpha=-16$, then 
	\[f(y)=-\frac{1}{2}y^2+\frac{1}{32}y^4+O(y^6),\]
and thus separatrix $2$ goes to the north pole in this case too. \end{proof}

\begin{proposition}
	The phase portrait of $X_{22a}$ with $a=-1$ is the one given by Figure~\ref{22b}.
\end{proposition}

\begin{proof} Here we point out two things. First we assume $\alpha<0$ and $\beta\geqslant1+2\sqrt{-\alpha}$. In this case one can see that the local phase portrait is given by Figure~\ref{Field22bLoca1}.
\begin{figure}[ht]	
	\begin{center}
		\begin{overpic}[height=4cm]{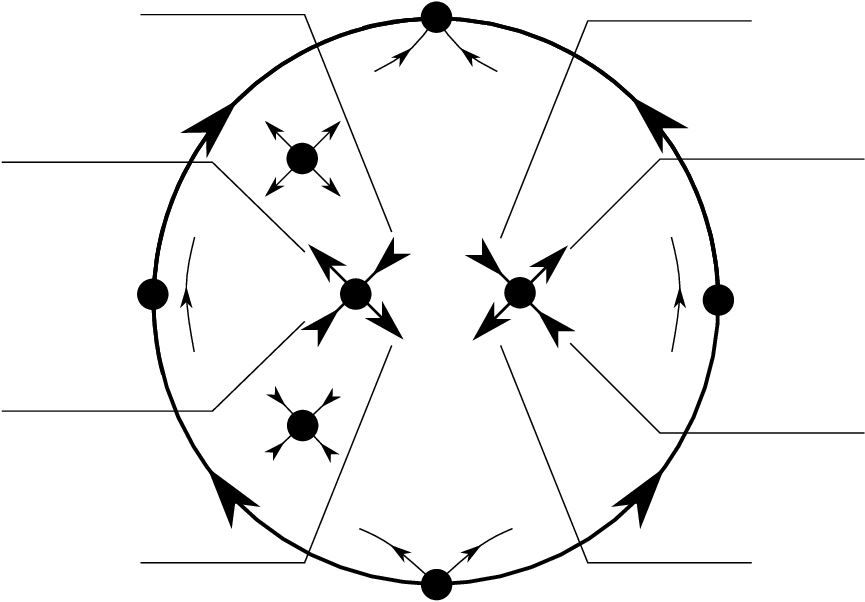} 
			\put(88,65.5){$1$}
			\put(101,49.5){$2$}
			\put(101,17.5){$3$}
			\put(88,2.5){$4$}
			\put(12.5,66){$5$}
			\put(-3.5,49){$6$}
			\put(-3.5,20){$7$}
			\put(12.5,2.5){$8$}
		\end{overpic}
	\end{center}	
	\caption{Local behavior of the phase portrait of $X_{22a}$ with $a=-1$, $\alpha<0$ and $\beta\geqslant1+2\sqrt{-\alpha}$.}\label{Field22bLoca1}
\end{figure}
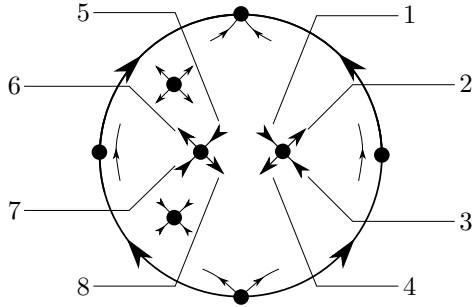
We want to prove that for $\beta$ big enough separatrix $6$ must end at the north pole. 
Doing the change of coordinates $(x,y)=(x_1,\sqrt{y_1})$ for $y>0$ one get the vector field $X_1=X_1(x_1,y_1)$ whose differential system is
	\[\dot x_1 = \sqrt{y_1}\bigl((\beta-1)x_1+(\beta+1)y_1\bigr), \quad \dot y_1 = \sqrt{y_1}\bigl(\alpha+x_1^2+2x_1y_1+y_1^2\bigr).\]
The fact that separatrix $6$ at system $X_{22a}$ crosses the set $S_x$ above the unstable node implies, for the system $X_1$, that separatrix $6$ also crosses the set $S_{x_1}$ above the unstable node $q^+$. See Figure~\ref{Field22bSpecial1}.
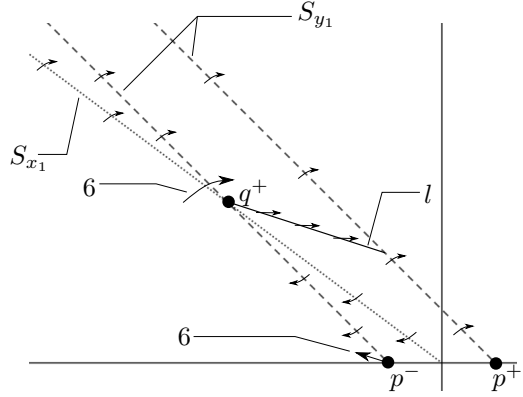
\begin{figure}[ht]	
	\begin{center}
		\begin{overpic}[height=5cm]{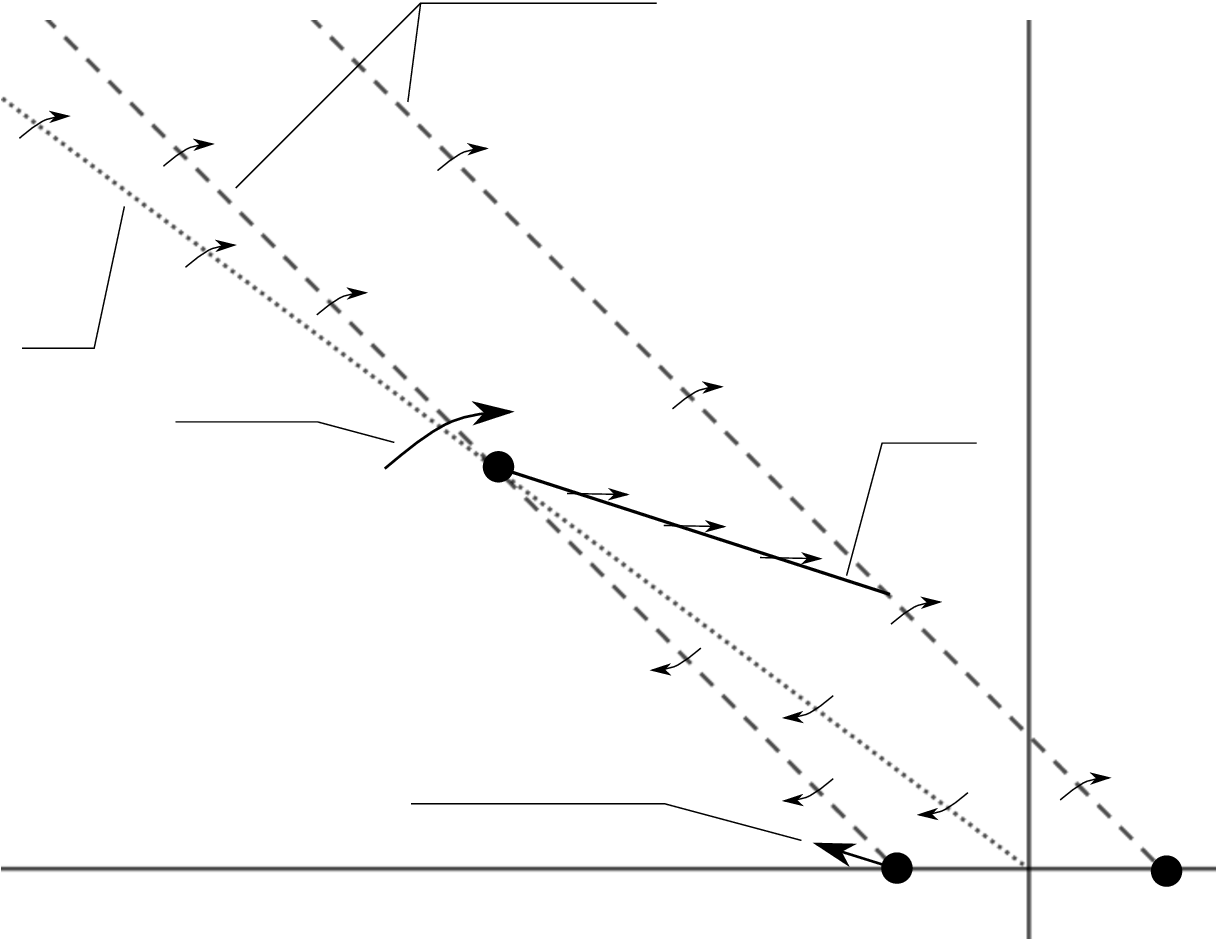} 
			\put(30.5,9.5){$6$}
			\put(74,1){$p^-$}
			\put(95,1){$p^+$}
			\put(43,39){$q^+$}
			\put(-4,46.5){$S_{x_1}$}
			\put(55,76){$S_{y_1}$}
			\put(11,40){$6$}
			\put(81,38.5){$l$}
		\end{overpic}
	\end{center}	
	\caption{Illustration of the flow of $X_1$ at the sets $S_{x_1}$ and $S_{y_1}$}\label{Field22bSpecial1}
\end{figure}
Taking $\beta\geqslant1+2\sqrt{-\alpha}+4\sqrt[4]{-\alpha}$ and defining the line 			
	\[l=q^++t\left(1,\frac{\alpha}{(-\alpha)^\frac{3}{4}-\alpha}\right)\]
one can conclude that the flow crosses $l$ upwards and thus separatrix $6$ must cross the line of $S_{y_1}$ which contains $p^+$. But at the right hand side of this line we have $\dot y_1>0$ and thus separatrix $6$ cannot end at $p^+$ and hence it must end at the north pole.

Second, we must prove that the relative position of the curves $\beta_1(\alpha)$ and $\beta_h(\alpha)=1+2\sqrt{-\alpha}$ is as it is shown in Figure~\ref{22b}, i.e. we must prove $\beta_1<\beta_h$ for $\alpha<0$ large enough and $\beta_h<\beta_1$ for $\alpha<0$ small enough. We remember that $\beta_1$ represents the moment when it exists a heteroclinic connection between both saddles. From now on we assume $(\alpha,\beta)=(\lambda,1+2\sqrt{-\lambda})$, $\lambda\leqslant0$, and we prove that for $\alpha<0$ large enough the connection already broke and thus $\beta_1<\beta_h$. Translating the saddle to the left of the origin and then doing a simple blow up in the $y$ direction one obtain the vector field $X_2=X_2(x_2,y_2)$ whose differential system is 
	\[\begin{array}{l}
		\dot x_2 = 2\lambda+\sqrt{-\lambda}x_2^2+3\sqrt{-\lambda}x_2y_2+2(1+\sqrt{-\lambda})y_2^2-\frac{1}{2}x_2^3y_2-x_2^2y_2^2-\frac{1}{2}x_2y_2^3, \vspace{0.2 cm} \\
		\dot y_2 = -\sqrt{-\lambda}x_2 y_2-\sqrt{-\lambda} y_2^2+\frac{1}{2} x_2^2 y_2^2+x_2 y_2^3+\frac{1}{2}y_2^4.
	\end{array}\]
In these coordinates set $S_{x_2}$ has a closed component. Let $p_1=(x_0,y_0)$ be the higher point in this closed component. Calculations shows that $p_1$ lies on the left of the straight line $r$ given by $x_2+y_2=-\sqrt{-\lambda}$. Let $p_2=(-y_0,y_0)$ be the projection of $p_1$ on $r$ and observe that $p_2$ is above the focus $p_3$, which lies in the same line. See Figure~\ref{Field22bSpecial2}. It follows from an analysis on the set $S_{x_2}$ that the separatrix $s$ must cross $r$ above the point $p_2$. Let $l$ be the segment given by $p_2+t\left(1,-\frac{1}{2}\right)$ such that it ends at the other component $r_1$ of $S_{y_2}$. For $\alpha<0$ large enough calculations shows that the flow on $l$ points upwards and thus separatrix $s$ must cross $r_1$. But at the right side of $r_1$ we have $\dot y_1 > 0$ and thus $s$ must end at the north pole. Therefore, for $\alpha<0$ large enough we have $\beta_1<\beta_h$.
\begin{figure}[ht]	
	\begin{center}
		\begin{overpic}[height=5cm]{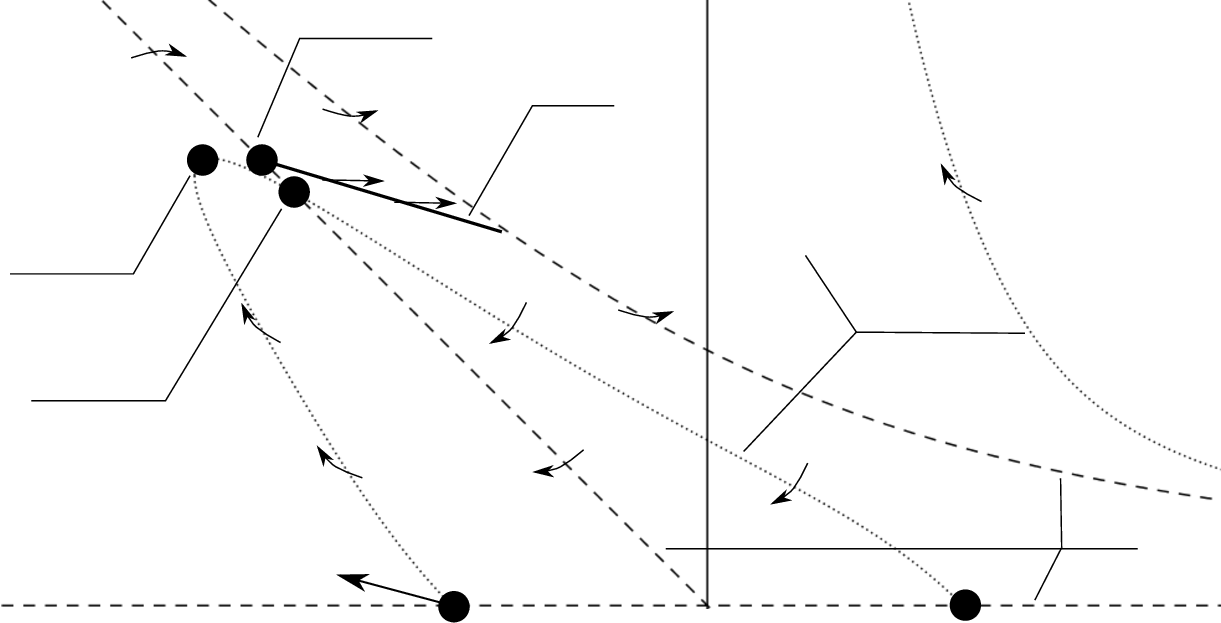} 
			\put(94,5){$S_{y_2}$}
			\put(64.5,30.5){$S_{x_2}$}
			\put(25,3){$s$}
			\put(51,41){$l$}
			\put(-3,28){$p_1$}
			\put(36,48){$p_2$}
			\put(-1,17.5){$p_3$}
			\put(5,50){$r$}
			\put(20,50){$r_1$}
			\put(58.5,50){$y$}
		\end{overpic}
	\end{center}	
	\caption{Illustration of the flow of $X_2$ at the sets $S_{x_2}$ and $S_{y_2}$.}\label{Field22bSpecial2}
\end{figure}
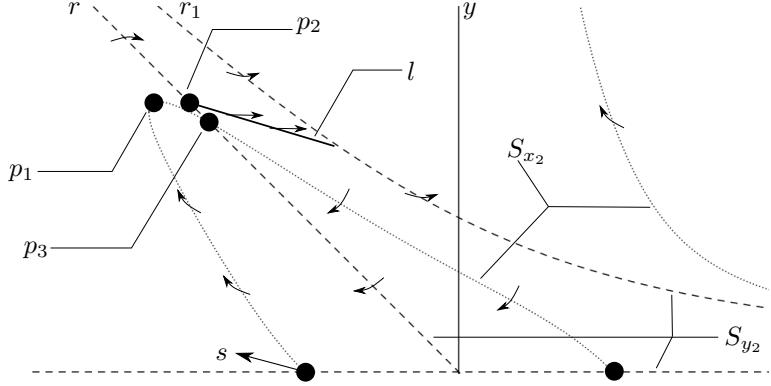
Once we cannot give $\beta_1$ explicitly and that there is nothing that prevents either $\beta_1<\beta_h$ or $\beta_h<\beta_1$ it turns out to be a very difficult task to understand completely and in a analytic way the relative position of $\beta_h$ and $\beta_1$. But numerical computations (see chapters $9$ and $10$ of \cite{DumLliArt2006}) shows that for $\lambda_0\approx-\frac{2}{3}$ is the unique intersection of $\beta_1$ and $\beta_h$. Moreover it also shows that $\beta_h<\beta_1$ if $\lambda\in(\lambda_0,0)$ and $\beta_1<\beta_h$ if $\lambda\in(-\infty,\lambda_0)$. So to provide an analytic proof of these facts is an open problem. \end{proof}

\section{System $X_{23}$}\label{sec5}

Let us remember that system $X_{23}$ is given by
	\[\dot x = a\alpha x y-y^3+ax^3y+axy^5, \quad \dot y = \frac{\beta}{2}+\frac{a}{2}x-\frac{\alpha}{2}y^2-\frac{1}{2}x^2y^2-\frac{1}{2}y^6,\]
with $a\in\{-1,1\}$.

Since this system has a particular higher degree than the previous systems it is not practical to find analytic expressions for the finite singularities as we did with the previous systems. Therefore, we change our approach. First we do the change of variables given by $(x,y)=(x_1,\sqrt{y_1})$, obtaining the vector field 
	\[\dot x_1 = \sqrt{y_1}(-y_1+x_1(x_1^2+y_1^2+\alpha)), \quad \dot y_1 = \sqrt{y_1}(x_1-y_1(x_1^2+y_1^2+\alpha)+\beta).\]
Dividing both equations by $\sqrt{y_1}$ we obtain the vector field $Y=(\dot x_2,\dot y_2)$ given by
\begin{equation}\label{3}
	\dot x_2 = -y_2+x_2(x_2^2+y_2^2+\alpha), \quad \dot y_2 = x_2-y_2(x_2^2+y_2^2+\alpha)+\beta.
\end{equation}
Observe that $X_{23}$ and $Y$ are topologically equivalent for $y>0$ (which is equivalent to $y_2>0$) and $Y$ has degree $3$ while $X_{23}$ has degree 6. Therefore, the approach is the following. We study the vector field $Y$ at $y_2>0$ and then draw conclusions for $X_{23}$ for $y>0$ (and thus for $y<0$ too due to its symmetry) and study locally the unique singularity $p=(-a\beta,0)$ of $X_{23}$ at the $x$-axis. From now on we focus on the case $a=1$.

\begin{proposition}\label{prop1X23}
	The following statement hold.
	\begin{enumerate}[label=(\alph*)]
		\item $p$ is a hyperbolic saddle if $\beta(\alpha+\beta^2)<0$ and a center if $\beta(\alpha+\beta^2)>0$.
		\item $p$ is a center if $\{\beta=0,\, -1\leqslant\alpha<1\}$ or $\{\alpha+\beta^2=0,\, |\beta|<1\}$.
		\item Otherwise the local phase portrait of $X_{23}$ at $p$ is given in Figure~\ref{23aOrigin}.
	\end{enumerate}
\begin{figure}[ht]
	\begin{multicols}{3}		
		\begin{center}
			\includegraphics[height=3cm]{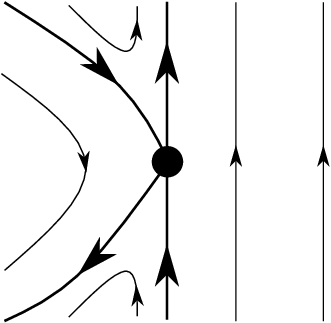}
			
			$\beta=0$ and $\alpha<-1$
		\end{center}		
		\columnbreak		
		\begin{center}
			\includegraphics[height=3cm]{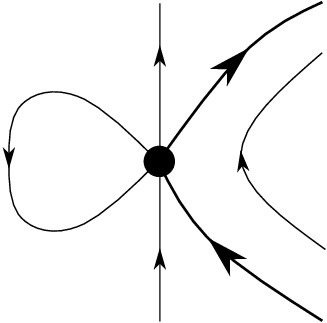}
			
			$\alpha+\beta^2=0$ and $|\beta|\geqslant1$
		\end{center}
		\columnbreak		
		\begin{center}
			\includegraphics[height=3cm]{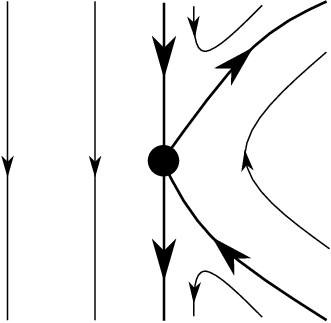}
			
			$\beta=0$ and $\alpha\geqslant1$
		\end{center}				
	\end{multicols}
	\caption{Local phase portraits of $X_{23}$ at $p$.}\label{23aOrigin}
\end{figure}
\end{proposition}

\begin{proof} Statement $(a)$ follows from 
	\[DX_{23}(p)=\left( \begin{array}{cc} 0 & -\beta(\alpha+\beta^2) \\ \frac{1}{2} & 0 \end{array}\right),\]
and the fact that $p$ is a symmetric singularity. Statements $(b)$ and $(c)$ follow from a quasihomogeneous blow up with weight $(2,1)$. \end{proof}

Let us define the following functions.
\begin{equation}\label{2}
	\begin{array}{rl}
		p_1(x_2)=&2x_2^2+\beta x_2+\alpha. \vspace{0.2cm} \\
		p_2(x_2)=&4x_2^5+5\beta x_2^4+(4\alpha+\beta^2)x_2^3+2\alpha\beta x_2^2+(\alpha^2-1)x_2-\beta; \vspace{0.2cm} \\
		f(x_2)=&\sqrt{x_2(x_2+\beta)}. \vspace{0.2cm} \\
		R(\alpha,\beta)=&-256(\alpha^2-1)^3-192\alpha(\alpha^4+7\alpha^2+28)\beta^2 \vspace{0.2cm} \\ 
		&+60(\alpha^4-28\alpha^2-72)\beta^4-4\alpha(\alpha^2-108)\beta^6-27\beta^8.
	\end{array}
\end{equation}

In what follows we list some analytic properties of the finite singularities of $Y$ when $y_2>0$.

\begin{proposition}\label{prop2X23}
		Let $q=(x_0,y_0)\in\mathbb{R}^2$ such that $y_0>0$. The following statements hold.
	\begin{enumerate}[label=(\alph*)]
		\item $q$ is a finite singularity of $Y$ if, and only if, 
			\[y_0=\frac{x_0+\beta}{p_1(x_0)}=x_0 p_1(x_0)=f(x_0).\]
		\item If $q$ is a finite singularity of $Y$, then $p_2(x_0)=0$ and $x_0 p_1(x_0)>0$.
		\item If $q_0=(x_0,f(x_0))$ is such that $p_2(x_0)=0$ and $x_0p_1(x_0)>0$, then $q_0$ is a finite singularity of $Y$.
		\item If $q$ is a non-hyperbolic finite singularity of $Y$, then $\beta=0$ or $R(\alpha,\beta)=0$.
	\end{enumerate}
\end{proposition}

\begin{proof} Isolating $u=x_2^2+y_2^2+\alpha$ from $\dot x_2=0$ and $\dot y_2=0$ we see that a necessary condition for a point $(x_0,y_0)$ be a singularity of $Y$ is that it satisfies
	\[\frac{y_0}{x_0}=\frac{x_0+\beta}{y_0}.\]
Knowing that $y_0>0$ we obtain $y_0=f(x_0)$. Statement $(a)$ now follows from $Y(x_0,f(x_0))=(0,0)$.

From statement $(a)$ one concludes that if $q$ is a finite singularity of $Y$, then $x_0 p_1^2(x_0)-(x_0+\beta)=0$ and $x_0 p_1(x_0)=f(x_0)\geqslant0$. Statement $(b)$ now follows from the fact that $x_0\neq0$ and $p_2(x_0)=x_0 p_1^2(x_0)-(x_0+\beta)$.

To prove statement $(c)$ first observe that $p_2(x_0)=0$ implies $x_0p_1(x_0)^2=x_0+\beta$ and thus
	\[x_0p_1(x_0)^2=x_0+\beta=\frac{f(x_0)^2}{x_0}.\]
Hence, $x_0^2p_1(x_0)^2=f(x_0)^2$. Squaring both sides and knowing that $x_0p_1(x_0)>0$ we obtain statement $(c)$.

Now observe that the Jacobian matrix of $Y$ is given by
	\[DY(x_2,y_2)=\left(\begin{array}{cc} 3x_2^2+y_2^2+\alpha & 2x_2y_2-1 \\ 1-2x_2y_2 & -(x_2^2+3y_2^2+\alpha)\end{array}\right).\]
Replacing $y_0=f(x_0)$ and calculating the trace and the determinant and then replacing $f(x_0)=x_0 p_1(x_0)$ we obtain
\begin{equation}\label{4}
	Tr(x_0)=-2x_0\beta, \quad Det(x_0)=-p_2'(x_0).
\end{equation}

It follows from the \emph{Trace-Determinant Characterization} (see Section 4.1 of \cite{Devaney}) that a singularity $q$ is non-hyperbolic if $Tr(x_0)=0$, $Det(x_0)=0$ or both. Therefore, if $q=(x_0,y_0)$, $y_0>0$, is a singularity, then it follows from \eqref{3} that $x_0\neq0$ and thus $Tr(x_0)=0$ if, and only if, $\beta=0$. Moreover, since $R(\alpha,\beta)$ is the \emph{resultant} (except by a constant) with respect to the variable $x_2$ between $p_2$ and $p'_2$ it follows from \eqref{4} and from statement $(b)$ that $Det(x_0)=0$ if, and only if, $R(\alpha,\beta)=0$. \end{proof}

In what follows we list some properties of $R(\alpha,\beta)$ (see \eqref{2}) and the parabola $\alpha+\beta^2=0$.

\begin{proposition}
	The following statement holds.
	\begin{enumerate}[label=(\alph*)]
		\item $R(\alpha,\beta)=0$ has four branches (two positives and two negatives) if $\alpha\leqslant-1$, two branches (one positive and one negative) if $-1<\alpha\leq1$ and no branches at all if $\alpha>1$.
		\item  Let $\beta_+$ be the negative branch that borns at $\alpha=1$ and $\beta_-$ the negative branch that borns at $\alpha=-1$. Then $\mu_1\approx(-3.398..,-1.849..)$ is the unique intersection between $\beta_+$ and $\beta_-$.
		\item $\mu_2\approx(-3.389..,-1.841..)$ is the unique intersection between the para\-bola $\alpha+\beta^2=0$ and $R(\alpha,\beta)=0$ at $\beta<0$. Moreover it occurs at $\beta_-$. See Figure~\ref{Field23aSpecial2}.
	\end{enumerate}
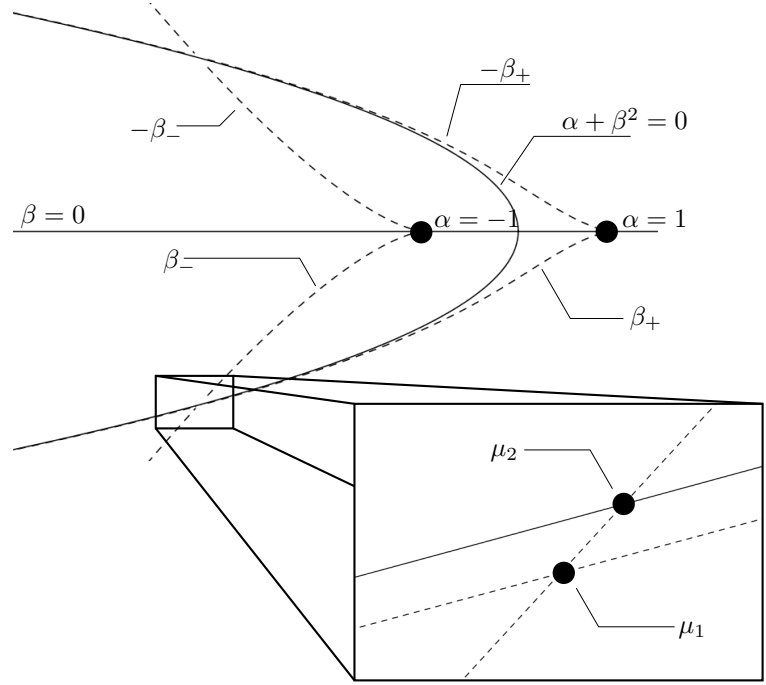
\begin{figure}[ht]	
	\begin{center}
	\begin{overpic}[width=10cm]{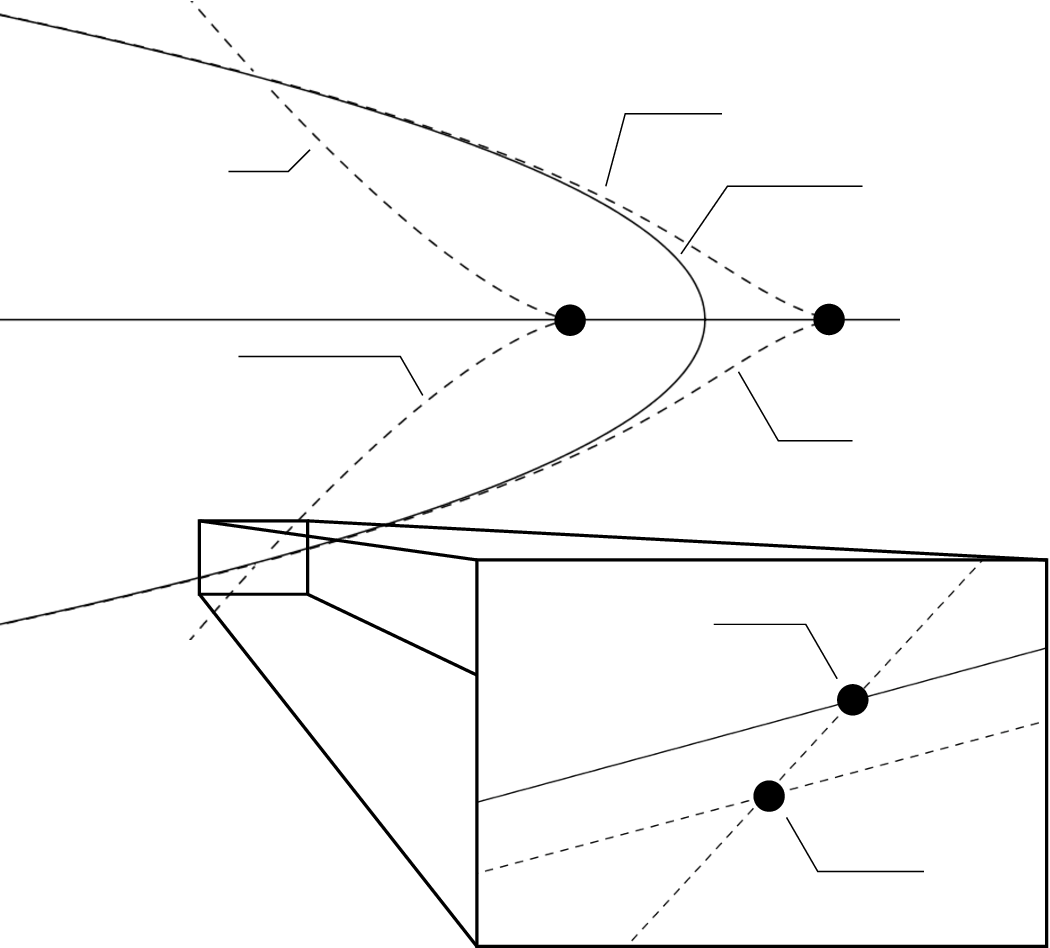} 
		\put(1,61){$\beta=0$}
		\put(56,60.5){$\alpha=-1$}
		\put(81,60.5){$\alpha=1$}
		\put(73,73.5){$\alpha+\beta^2=0$}
		\put(62,80.5){$-\beta_+$}
		\put(15.5,72.5){$-\beta_-$}
		\put(82,47.5){$\beta_+$}
		\put(20,55){$\beta_-$}
		\put(63.5,30){$\mu_2$}
		\put(88.5,6.5){$\mu_1$}
	\end{overpic}
	\end{center}	
	\caption{Plot of the sets $R(\alpha,\beta)=0$ and $\alpha+\beta^2=0$.}\label{Field23aSpecial2}
\end{figure}
\end{proposition}

\begin{proof} Fortunately $R(\alpha,\beta)$ is biquadratic at $\beta$ and thus we can do the change of variables $\beta\mapsto\pm\sqrt{b}$ (both give the same result) and obtain the quartic polynomial in $b$ 
	\[\begin{array}{rl}
		R_1(\alpha,b)=&-256(\alpha^2-1)^3-192\alpha(\alpha^4+7\alpha^2+28)b \vspace{0.2cm} \\ 
		&+60(\alpha^4-28\alpha^2-72)b^2-4\alpha(\alpha^2-108)b^3-27b^4.
	\end{array}\]
Given $\alpha_0\in\mathbb{R}$ fixed we want to know how many positive real roots the polynomial $R_1(\alpha_0,b)$ has. We recall that if
	\[P(x)=a_4x^4+a_3x^3+a_2x^2+a_1x+a_0\]
is a real polynomial of degree four and $r_1$, $r_2$, $r_3$ and $r_4$ are its roots, then its discriminant (see \cite{Dickson} p. 157) is given by
	\[a_4^6\prod_{i<j}(r_i-r_j)^2.\]
The discriminant in $b$ of $R_1$ is given (see chapter $12$ of \cite{Dickson} for a expression of the discriminant in function of the coefficients), except by a constant, by
	\[\begin{array}{rl}
		D(\alpha)=& -1\;871\;773\;696-10\;034\;479\;104\;\alpha^2-19\;980\;402\;688\;\alpha^4 \vspace{0.2cm} \\ 
		&-17\;398\;321\;152\;\alpha^6-5\;393\;464\;832\;\alpha^8+250\;599\;680\;\alpha^{10} \vspace{0.2cm} \\ 
		&+64\;492\;352\;\alpha^{12}+870\;480\;\alpha^{14}-190\;633\;\alpha^{16}-9\;567\;\alpha^{18} \vspace{0.2cm} \\
		&-167\;\alpha^{20}-\alpha^{22}.
	\end{array}\]
Computations shows that $D(\alpha)\leqslant0$ with the equality happening only if $\alpha=\pm3.398..$. It is well known (see \cite{Dickson} p. 45) that if $D<0$, then $R_1$ has two distinct real roots and two complex conjugate roots. If $\{b_1,b_2,b_3,b_4\}$ are the four roots in $b$ of $R_1$, then 
	\[R_1(\alpha,\beta)=-27(b-b_1)(b-b_2)(b-b_3)(b-b_4)\]
and thus $27b_1b_2b_3b_4=256(\alpha^2-1)^3$. Supposing that $b_1$, $b_2$ are the real solutions and $b_3$, $b_4$ the complex solutions we conclude that $sign(b_1b_2)=sign(\alpha^2-1)$. Hence, there is a unique positive solution when $-1<\alpha<1$. Moreover $b_1$ and $b_2$ can change signal only if $\alpha=\pm1$. Choosing arbitrarily $\alpha=\pm2$ we see that there is no positive solutions for $\alpha=2$ and there is two positive real solutions for $\alpha=-2$. Hence we conclude that there is no branch of positive real solutions of $R_1$ if $\alpha>1$; one branch if $-1<\alpha\leqslant1$ and two branches if $\alpha\leqslant-1$. Squaring those branches and then reflecting them at the $x$-axis one can conclude statement $(a)$. For more details about the nature of the roots of a polynomial of degree four see \cite{Quartic}.

If the two branches $\beta_\pm$ intersect each other, then we have a double positive real solution of $R_1$ which requires $D(\alpha)=0$ and thus $\alpha=-3.398..$. Replacing this at $R_1$ one can see that there is two complex conjugate solutions and a double positive real root and thus we have statement $(b)$.

Knowing that $R(\alpha,-\sqrt{-\alpha})=256+288\alpha^2-27\alpha^4$ one can calculate its roots and see that the unique root $\alpha_0\leqslant-1$ is given by $\alpha_0=-\frac{4}{3}\sqrt{3+2\sqrt{3}}$ and thus $\mu_2=(\alpha_0,-\sqrt{-\alpha_0})$ is the unique intersection between the sets $\alpha+\beta^2=0$ and $R(\alpha,\beta)=0$. The relative position of $\mu_1$, $\mu_2$ and the parabola $\alpha+\beta^2=0$ and the fact that $\beta_\pm$ are the graphs of some continuous function implies that $\mu_2\in\beta_-$ and thus we have statement $(c)$. \end{proof}

Numerical computations show that the branches $-\beta_\pm$ of $R(\alpha,\beta)=0$ perturb the double roots $x_0$ of $p_2$ which do not satisfy $x_0p_1(x_0)>0$ and thus perturb the roots that are not related to the finite singularities of $Y$ and thus can be ignored. Moreover, the negative branches $\beta_\pm$ perturb the double roots $x_0$ of $p_2$ which do satisfy $x_0p_1(x_0)>0$ and thus perturb the roots that are related to the finite singularities. One can now draw the backbone of the bifurcation diagram of $X_{23}$, with $a=1$, and thus obtain the solid lines of Figure~\ref{23a1}.

\begin{proposition}
	Let $(\alpha_0,\beta_0)\in\mathbb{R}^2$ such that $\beta_0=\beta_\pm(\alpha_0)$ (i.e. it is a point in one of the brunches $\beta_\pm$), with $\beta_0<0$, and $\gamma(t)=(t+\alpha_0,\beta_0)$, $|t|<\varepsilon$, a transversal segment through $\beta_\pm$. Then a saddle-node bifurcation happens at $\gamma(0)$ in such way that it vanishes when $t$ increases.
\end{proposition}

\begin{proof} We know that at $t=0$ we have a double real root $x_0$ of $p_2$ which satisfy $x_0 p_1(x_0)>0$. Therefore, it follows from statement $(c)$ of Proposition~\ref{prop2X23} that $q_0=(x_0,f(x_0))$ is a non-hyperbolic finite singularity of $X_{23}$. Since $\beta_0<0$ we know that the trace of $DX_{23}(q_0)$ is not zero and thus $DX_{23}(q_0)$ has only one eigenvalue equal zero. Hence, $q_0$ is a non-hyperbolic saddle, a non-hyperbolic node or a saddle-node.

Computations show that if $t<0$, then the double root $x_0$ splits in two simple real roots $x_-$, $x_+$ satisfying $x_\pm p_1(x_\pm)>0$. Hence, it follows from Proposition~\ref{prop2X23} that $q_\pm=(x_\pm,f(x_\pm))$ are both hyperbolic finite singularities of $X_{23}$. The proof now follows from the fact that if $t>0$, then the double real root $x_0$ goes to the complex realm and thus the finite singularity $q_0$ vanishes. \end{proof}

Calculating the infinity of $X_{23}$, with $a=1$, one can see that the only singularities are the origins of the charts $U_1$ and $U_2$. While the origin of $U_2$ is an unstable node the local phase portrait at the origin of $U_1$ is given by Figure~\ref{Field23aEastPole}.
\begin{figure}[ht]	
	\begin{center}
		\begin{overpic}[width=3cm]{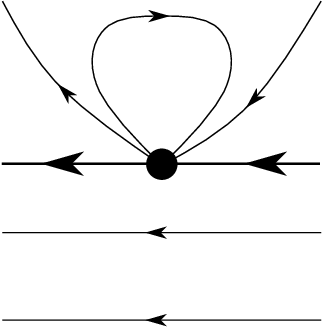} 
		\end{overpic}
	\end{center}	
	\caption{Local phase portrait of $p(X_{23})$, with $a=1$, at the origin of $U_1$.}\label{Field23aEastPole}
\end{figure}
To prove this last claim one must do a quasihomogeneous directional blow up at the direction $x^+$ with weight $(3,2)$ and then a quasihomogenous blow up with $(\alpha,\beta)=(1,4)$ at the unique singularity that appear. Hence it follows from Theorem~\ref{PH} that \emph{the sum of the Poincar\'e indices of all the finite singularities must always be equal to $-1$}.

It follows from the characterization given at Definition~\ref{Def4} that $X_{23}$ with $a=1$ is topologically equivalent at $p=(-\lambda,0)$ to the origin of:
\begin{enumerate}[label=(\alph*)]
	\item $X_{12}$ when $\alpha_0+\beta_0^2=0$, $|\beta_0|\geqslant1$ and $|\beta-\beta_0|<\varepsilon$.
	\item $X_{13}$ when $\beta_0=0$, $|\alpha_0|>1$ and $|\beta|<\varepsilon$.
	\item $X_{14}$ when $\alpha_0+\beta_0^2=0$, $0<|\beta_0|<1$ and $|\beta-\beta_0|<\varepsilon$.
\end{enumerate}
Although we have the local equivalence between $p$ at $X_{23}$ and the origin of $X_{12}$, $X_{13}$ or $X_{14}$ we do not have the direction of the bifurcation, i.e. as $\beta$ increases do $\lambda$ increase or decrease? The following proposition solves this question.

\begin{proposition}
 	Let $\lambda$ be the parameter of $X_{12}$, $X_{13}$ and $X_{14}$ and $(\alpha_0,\beta_0)\in\mathbb{R}^2$ fixed. The following statement holds.
\begin{enumerate}[label=(\alph*)]
	\item If $\alpha_0+\beta_0^2=0$, $|\beta_0|\geqslant1$ and $|\beta-\beta_0|<\varepsilon$, then as $|\beta|$ increases $\lambda$ also increases.
	\item If $\beta_0=0$, $|\alpha_0|>1$ and $|\beta|<\varepsilon$, then as $\beta$ increases $\lambda$ increases if $\alpha_0>1$ and decreases if $\alpha_0<-1$.
	\item If $\alpha_0+\beta_0^2=0$, $0<|\beta_0|<1$ and $|\beta-\beta_0|<\varepsilon$, then as $\beta$ increases $\lambda$ also increases.
\end{enumerate}
\end{proposition}

\begin{proof} All statements follow from statement $(a)$ of Proposition~\ref{prop1X23} and the fact that if $\lambda>0$, then $X_{12}$, $X_{13}$ and $X_{14}$ has a lonely center while if $\lambda<0$ it have a hyperbolic saddle. \end{proof} 

Finally we point out that at $\beta=0$ and $\alpha<-1$ we have a center-focus problem at the singularity $p_0=\frac{1}{\sqrt{2}}(-\sqrt{-1-\alpha},\sqrt{-1-\alpha})$ of $Y$. Translating this singularity to the origin one obtain the vector field $Y_1$ whose differential system is
\[\begin{array}{rl}
	\dot x =& \displaystyle -(2+\alpha)x+\alpha y-\frac{3}{\sqrt{2}}\sqrt{-1-\alpha}x^2+\sqrt{2}\sqrt{-1-\alpha}xy-\frac{\sqrt{-1-\alpha}}{\sqrt{2}}y^2 \vspace{0.2cm} \\
	&+x^3+xy^2 \vspace{0.2cm} \\
	\dot y =& \displaystyle -\alpha x+(2+\alpha)y-\frac{\sqrt{-1-\alpha}}{\sqrt{2}}x^2+\sqrt{2}\sqrt{-1-\alpha}xy-\frac{3}{\sqrt{2}}\sqrt{-1-\alpha}y^2 \vspace{0.2cm} \\
	&-x^2y-y^3.
\end{array}\]
Knowing that the vector field $Y_1$ is $\varphi$-reversible with $\varphi(x,y)=(-y,-x)$ we conclude that $p_0$ is a center.

\begin{proposition}
	The phase portraits of $X_{23}$ with $a=1$ are the one given in Figures~\ref{23a1}, \ref{23a2}, \ref{23a3} and \ref{23a4}.
\end{proposition}

\begin{proof} With an analysis of $f(x_2)$ and $x_2p_1(x_2)$ (see the functions defined at \eqref{2}) one can see that in region $1$ of Figure~\ref{23a1} there is only one pair of singularities aside $p=(-\lambda,0)$ and in region $23$ there is no other singularity other than $p$.  
	
Using the same techniques as in the previous systems, and the information that we already have, one can obtain the phase portraits $1$, $2$, $3$, $4$, $5$, $13$, $23$, $31$ and $32$. Looking at $13$ and $23$ and knowing that at $22$ we have the collapse of a node and a saddle in a saddle-node one can obtain $22$. Again if we collapse the saddle and the node in $5$ in a saddle-node we obtain the phase portrait $14$. As we get down at $\beta_-$ this saddle-node get away from $p$ and phase portraits $15$ and $16$ rise and with them $6$ and $7$, respectively.

At the intersection of $\beta_+$ and the $\beta=0$ if get down at $\beta_+$, then a saddle-node born at $p$. Calculations shows that near $\beta=0$ the direction of the unstable manifold of this saddle-node is very near the direction of the straight line $y=x$ and thus we obtain $24$. As we follow $\beta_+$ the direction of the unstable manifolds approach the direction of $y=0$ and then we can get $25$ and $26$. Looking at $26$ and $22$ we can have $21$. If we split the right hand saddle-node of $21$ in a saddle and a node we obtain $20$. Again if we split the left hand saddle-node of $20$ we obtain $11$. From $20$ and $32$ and knowing the bifurcation that happens at $p$ at $\alpha+\beta^2=0$ we can obtain $19$. Now we can get $17$ and $18$ if we look at $16$ and $19$. From them one can obtain $8$, $9$ and $10$. From $11$ and $13$ we have $12$. From $24$ we obtain $27$. From $26$ we obtain $29$. From $27$ and $29$ we obtain $28$. From $29$ and $31$ we obtain $30$.
	
We observe that it follows from the continuity that it is impossible that the curves defined by $8$ and $6$ intercept each other. On the other hand we do not know if the curves defined by $28$ and $30$ intercept each other. But combining the bifurcation of $28$ and $30$ we know that if the curves intercept each other, then we have the phase portrait of Figure~\ref{23a4}, and thus region defined by $29$ is disconnected. \end{proof}

Let $X_{23a}=(P,Q)$ be $X_{23}$ with $a=1$ and $X_{23b}$ be $X_{23}$ with $a=-1$. Observe that if we do the change of coordinates $(x,y)\mapsto(-x,y)$ at $X_{23a}$, then we obtain $X_{23b}=(-P,Q)$. Hence a huge amount of information of $X_{23a}$ can be carry on to $X_{23b}$. The most important ones are the following.
\begin{enumerate}[label=\arabic*.]
	\item The relative position of the finite singularities are the same. Hence, the solid lines of Figure~\ref{23a1} are carried onto the bifurcation diagram of $X_{23b}$.
	\item The determinant of the Jacobian matrices at the finite singularities are the same except by a factor of $-1$. Therefore, if $q$ is a saddle (focus/node) of $X_{23a}$, then it is a focus/node (saddle) of $X_{23b}$.
	\item If $\beta(\alpha+\beta^2)\neq0$, then $p=(-\beta,0)$ is center (saddle) at $X_{23a}$ if, and only if, it is a saddle (center) at $X_{23b}$.
\end{enumerate}
With the same approach as the previous cases, mainly $X_{23a}$, one can conclude the following proposition.

\begin{proposition}
	The phase portraits of $X_{23}$ with $a=-1$ are the ones given in Figures~\ref{23b1} and \ref{23b2}.
\end{proposition}

\section{Systems $X_{24}$ and $X_{25}$}\label{sec6}

Let us remember that $X_{24}$ is given by 
	\[\dot x = axy+\alpha y^3, \quad \dot y = \frac{\beta}{2}+\frac{1}{2}x+\frac{a}{2}y^2,\]
where $a\in\{-1,1\}$. First we observe that if we replace $a=-1$ and then do the change of variables
	\[(x,y,\alpha,\beta)=(-x_1,-y_1,\alpha_1,-\beta_1),\]
then we get the same system with $a=1$. Therefore from now on we assume $a=1$. In this system we point out that it is the unique system which the maximum degree depends on the parameters. Observe that the maximum degree is three if $\alpha\neq0$ and two if $\alpha=0$. Therefore, one must do an analysis at each case. Here we only point out some information about the local phase portrait at the origin of chart $U_1$ for $\alpha=1$ and $\beta\neq0$. In this case the local phase portrait of the blow up is incomplete due to the existence of two singularities with a unique eigenvalue (the radial one) equal zero. See Figure~\ref{Field24EastPole3}.
\begin{figure}[ht]	
	\begin{center}
		\begin{overpic}[height=3cm]{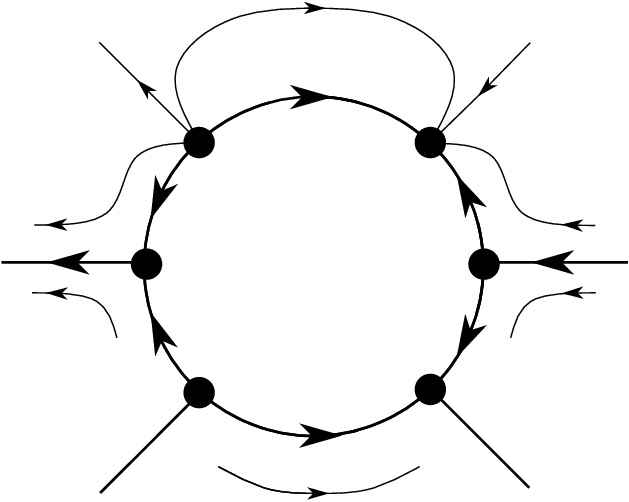} 
		\end{overpic}
	\end{center}	
	\caption{Unfinished blow up of the origin of chart $U_1$ of $X_{24}$ with $\alpha=1$.}\label{Field24EastPole3}
\end{figure}
Calculations show that in this case a center manifold of the origin of the chart $U_1$ such that $v<0$, $|v|<\varepsilon$, is given by the graph of
	\[v(u)=-u^2-\frac{\beta}{2}u^4+O(u^6).\]
This in addition with the fact that
	\[\left.\dot v\right|_{\alpha=1} = -uv(u^2+v),\]
it is enough to calculate the direction of the flow on the center manifold and thus to complete the blow up when $\beta\neq0$. See Figure~\ref{Field24EastPole4}.
\begin{figure}[ht]
	\begin{multicols}{2}				
		\begin{center}
			\includegraphics[height=3cm]{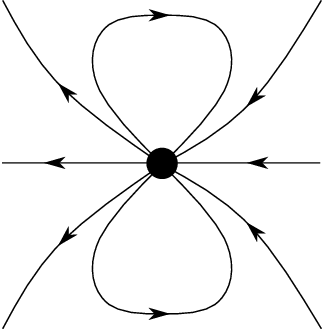}
			
			$\beta<0$
		\end{center}
		\columnbreak		
		\begin{center}
			\includegraphics[height=3cm]{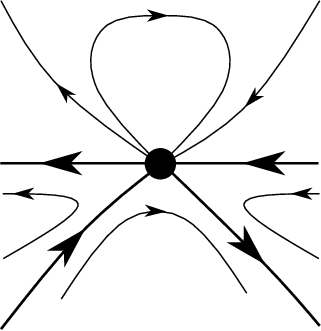}
			
			$\beta>0$
		\end{center}				
	\end{multicols}
	\caption{Local phase portrait of $X_{24}$ at the origin of the chart $U_1$ with $a=1$, $\alpha=1$ and $\beta\neq0$.}\label{Field24EastPole4}
\end{figure}

Let us remember that system $X_{25b}$ is given by
	\[\dot x = axy, \quad \dot y = \frac{\beta}{2}+\frac{\alpha}{2}x+\frac{\delta}{2}x^2+\frac{b}{2}y^2,\]
where $ab>0$, $\delta\in\{-3,3\}$ and 
	\[\bigl\{\left(a\in\{-1,1\},\, b\in\{-3,3\}\right) \text{ or } \left(a\in\{-3,3\},\, b\in\{-1,1\}\right)\bigr\}.\]
Here we observe that with the change of variables
	\[(x,y,\alpha,\beta)\rightarrow(x_1,-y_1,-\alpha_1,-\beta_1),\]
it is enough to study only the four cases given by
\begin{enumerate}[label=(\alph*)]
	\item $a=1$, $b=3$ and $\delta=3$.
	\item $a=1$, $b=3$ and $\delta=-3$.
	\item $a=3$, $b=1$ and $\delta=3$.
	\item $a=3$, $b=1$ and $\delta=-3$.
\end{enumerate}

\section*{Acknowledgments}

This work is supported by Agencia Estatal de Investigaci\'on grant PID2022-136613NB-100, AGAUR (generalitat de Catalunha) grant 2021SGR00113 and by the Reial Acadèmia de Cièncias i Arts de Barcelona, by CNPq grant 304798/2019-3, by Agence Nationale de la Recherche (ANR), project ANR-23-CE40-0028, and S\~ao Paulo Research Foundation (FAPESP), grants 2021/01799-9 and 2023/02959-5.

\appendix

\section{Local linearization of the involution}\label{Linearization}

Let $Q\colon\mathbb{R}^m\to\mathbb{R}^m$ be a linear involution, i.e. a linear map such that $Q=Q^{-1}$ (also known as \emph{involuntory matrix}). Then the following statements hold.
\begin{enumerate}[label=(\roman*)]
	\item If $\lambda\in\mathbb{C}$ is a eigenvalue of $Q$, then $\lambda\in\{-1,1\}$.
	\item $Q$ is diagonalizable.
\end{enumerate}

Indeed, let $\lambda\in\mathbb{C}$ be a eigenvalue of $Q$ and $x\in\mathbb{R}^m$, $x\neq0$, one of its eigenvectors. Since $Q^2$ is the identity map it follows that,
	\[x=Q^2x=Q(Qx)=\lambda Qx=\lambda^2x.\]
In particular we have that $\lambda^2=1$ and thus $\lambda\in\{-1,1\}$, proving statement $(i)$. 

We now prove statement $(ii)$. Suppose by contradiction that $Q$ is not diagonalizable. Then its Jordan canonical form (see~\cite[Chapter $4.5$]{Jordan}) has at least one Jordan block of size $k\geqslant2$. Hence there is a eigenvalue $\lambda\in\mathbb{C}$ of $Q$ and vectors $u$, $v\in\mathbb{R}^m\setminus\{0\}$ such that 
	\[Qu=\lambda u, \quad Qv=u+\lambda v.\]
Now observe that,
	\[v=Q^2v=Q(u+\lambda v)=\lambda u+\lambda(u+\lambda v)=2\lambda u+\lambda^2v.\]
From statement $(i)$ we know that $\lambda^2=1$ and thus we obtain $0=2\lambda u$, which in turn implies $\lambda=0$ contradicting $(i)$.

\begin{proposition}
	Let $\varphi\colon\mathbb{R}^m\to\mathbb{R}^m$ be an involution such that $\operatorname{Fix}(\varphi)$ is an embedded sub-manifold of dimension $r\in\{0,\dots,m\}$. If $p\in\operatorname{Fix}(\varphi)$, then there is a neighborhood of $p$ such that $\varphi$ is $C^k$-conjugated to
		\[\Phi(x_1,\dots,x_m)=(x_1,\dots,x_r,-x_{r+1},\dots,-x_m).\]
\end{proposition}

\begin{proof} 
	Let $T(x)=x+p$ and $\psi(x)=T^{-1}\circ \varphi \circ T$. Observe that $\psi(0)=0$ and $\psi=\psi^{-1}$. Therefore $\varphi$ is conjugate to another involution $\psi$ such that $\psi(0)=0$ and thus we can suppose $p=0$. Let $Q=D\varphi(0)$ and $\sigma(x)=x+Q\varphi(x)$. Since $\varphi=\varphi^{-1}$ and $\varphi(0)=0$ it follows that,
	\begin{equation}\label{5}
		Q^{-1}=\left[D\varphi(0)\right]^{-1}=D\varphi^{-1}(\varphi(0))=D\varphi(\varphi(0))=D\varphi(0)=Q.
	\end{equation}
	In particular we have that $D\sigma(0)=2\operatorname{Id}_{\mathbb{R}^m}$ and thus from the Inverse Function Theorem it follows that $\sigma$ is locally a $C^k$-difeomorphism at the origin. Observe that
		\[\sigma(\varphi(x))=\varphi(x)+Qx=Q(Q\varphi(x)+x)=Q\sigma(x).\]
	Therefore $\sigma$ conjugates $\varphi$ and $Q$. In particular it follows that $\operatorname{Fix}(Q)\subset\mathbb{R}^m$ is sub vector space of dimension $r$. From \eqref{5} we know that $Q$ is an involution. Thus from statements $(i)$ and $(ii)$ we have that there is a linear change of variables such that in these new variables $Q$ is given by,
		\[Q(x_1,\dots,x_m)=(x_1,\dots,x_r,-x_{r+1},\dots,-x_m).\]
	The result now follows from the fact that $\varphi$ is $C^k$-conjugated to $Q$ by $\sigma$. 
\end{proof}

\section{Poincar\'e Compactification}\label{PC}

Let $X$ be a planar \emph{polynomial} vector field of degree $n$ as our polynomial differential systems of Theorem~\ref{T1}. The \emph{Poincar\'e compactified vector field} $p(X)$ is an analytic vector field on $\mathbb{S}^2$ constructed as follow (for more details see~\cite[Chapter~$5$]{DumLliArt2006}).

First we identify $\mathbb{R}^2$ with the plane $(x_1,x_2,1)$ in $\mathbb{R}^3$ and define the \emph{Poincar\'e sphere} as $\mathbb{S}^2=\{y=(y_1,y_2,y_3)\in\mathbb{R}^3:y_1^2+y_2^2+y_3^2=1\}$. We define the \emph{northern hemisphere}, the \emph{southern hemisphere} and the \emph{equator} respectively by $H_+=\{y\in\mathbb{S}^2:y_3>0\}$, $H_-=\{y\in\mathbb{S}^2:y_3<0\}$ and $\mathbb{S}^1=\{y\in\mathbb{S}^2:y_3=0\}$.

Consider now the projections $f_\pm:\mathbb{R}^2\rightarrow H_\pm$ given by
	\[f_\pm(x_1,x_2)=\pm \Delta(x_1,x_2)(x_1,x_2,1),\]
where $\Delta(x_1,x_2)=(x_1^2+x_2^2+1)^{-\frac{1}{2}}$. These two maps define two copies of $X$, one copy $X^+$ in $H_+$ and one copy $X^-$ in $H_-$. Consider the vector field $X'=X^+\cup X^-$ defined in $\mathbb{S}^2\backslash\mathbb{S}^1$. Note that the \emph{infinity} of $\mathbb{R}^2$ is identified with the equator $\mathbb{S}^1$. The Poincar\'e compactified vector field $p(X)$ is the analytic extension of $X'$ from $\mathbb{S}^2\backslash\mathbb{S}^1$ to $\mathbb{S}^2$ given by $y_3^{n-1}X'$. The \emph{Poincar\'e disk} $\mathbb{D}$ is the projection of the closed northern hemisphere to $y_3=0$ under $(y_1,y_2,y_3)\mapsto(y_1,y_2)$ (the vector field given by this projection is also denoted by $p(X)$). Note that to know the behavior $p(X)$ near $\mathbb{S}^1$ is the same than to know the behavior of $X$ near the infinity. We define the local charts of $\mathbb{S}^2$ by $U_i=\{y\in\mathbb{S}^2:y_i>0\}$ and $V_i=\{y\in\mathbb{S}^2:y_i<0\}$ for $i\in\{1,2,3\}$. In these charts we define $\phi_i:U_i\rightarrow\mathbb{R}^2$ and $\psi_i:V_i\rightarrow\mathbb{R}^2$ by 
	\[\phi_i(y_1,y_2,y_3)=-\psi_i(y_1,y_2,y_3) = \left(\frac{y_m}{y_i},\frac{y_n}{y_i}\right),\]
where $m\neq i$, $n\neq i$ and $m<n$. Denoting by $(u,v)$ the image of $\phi_i$ and $\psi_i$ in every chart (therefore $(u,v)$ play different roles in each chart) one can see the following expressions for $p(X)$:
\begin{equation*}
	\begin{aligned}
		&v^n\;m(u,v)\left(Q\left(\frac{1}{v},\frac{u}{v}\right)-uP\left(\frac{1}{v},\frac{u}{v}\right),-vP\left(\frac{1}{v},\frac{u}{v}\right)\right) \text{ in } U_1, \\
		&v^n\;m(u,v)\left(P\left(\frac{u}{v},\frac{1}{v}\right)-uQ\left(\frac{u}{v},\frac{1}{v}\right),-vQ\left(\frac{u}{v},\frac{1}{v}\right)\right) \text{ in } U_2, \\
		&m(u,v)(P(u,v),Q(u,v)) \text{ in } U_3,
	\end{aligned}
\end{equation*}
where $m(u,v)=(u^2+v^2+1)^{-\frac{1}{2}(n-1)}$. We can omit the term $m(u,v)$ by a time rescaling of $p(X)$. Therefore, we obtain a polynomial expression of $p(X)$ in each $U_i$. The expressions of $p(X)$ in each $V_i$ is the same as that for each $U_i$, except by a multiplicative factor of $(-1)^{n-1}$. In these coordinates for $i\in\{1,2\}$, $v=0$ always represents the points of $\mathbb{S}^1$ and thus the infinity of $\mathbb{R}^2$. Note that $\mathbb{S}^1$ is invariant under the flow of $p(X)$.

\section{Blow Up Technique}\label{BlowUp}

If the origin is an isolated singularity of a \emph{polynomial} vector field $X$, then we can apply the change of coordinates $\phi:\mathbb{R}_+\times \mathbb{S}^1\rightarrow \mathbb{R}^2$ given by $\phi(\theta,r)=(r\cos\theta,r\sin\theta)=(x,y)$, where $\mathbb{R}_+=\{r\in\mathbb{R}:r>0\}$. Therefore, we can induce a vector field $X_0$ in $\mathbb{R}_+\times\mathbb{S}^1$ by pullback, i.e. $X_0=D\phi^{-1}  X$. One can see that if the $k$-jet of $X$ (i.e. the Taylor expansion of order $k$ of $X$, denoted by $j_k$) is zero at the origin, then the $k$-jet of $X_0$ is also zero in every point of $\{0\}\times\mathbb{S}^1$. Thus, taking the first $k\in\mathbb{N}$ satisfying $j_k(0,0)=0$ and $j_{k+1}(0,0)\neq0$ we can define the vector field $\hat X=\frac{1}{r^k}X_0$. Therefore, to know the behavior of $\hat X$ near $\mathbb{S}^1$ is the same than to know the behavior of $X$ near the origin. One can also see that $\mathbb{S}^1$ is invariant under the flow of $\hat X$. Observe that $\hat X$ is given by 
	\[\dot r = \frac{x\dot x+y\dot y}{r^{k+1}}, \quad \dot \theta = \frac{x\dot y-y\dot x}{r^{k+2}}.\]
There is a generalization of the Blow Up Technique, known as \emph{Quasihomogeneous Blow Up}, where the change of variables is given by 
	\[\psi(\theta,r)=(r^\alpha \cos\theta,r^\beta \sin\theta)=(x,y),\]
for $(\alpha,\beta) \in \mathbb{N}^2$. The pair $(\alpha,\beta)$ is called the \emph{weight} of the quasihomogeneous blow up. Similarly to the previous technique, we can induce a vector field $X_0$  in $\mathbb{R}_+\times \mathbb{S}^1$. For some $k \in \mathbb{N}$ maximal one can define $X_{\alpha,\beta}=\frac{1}{r^k}X_0$ and see that this vector field is given by
	\[\dot r = \xi(\theta)\frac{\cos\theta\;r^\beta\dot x+\sin\theta\;r^\alpha\dot y}{r^{\alpha+\beta+k-1}}, \quad \dot \theta = \xi(\theta)\frac{\alpha \cos\theta\;r^\alpha\dot y-\beta \sin\theta\;r^\beta\dot x}{r^{\alpha+\beta+k}},\]
where $\xi(\theta)=(\beta \sin^2\theta+\alpha \cos^2\theta)^{-1}$. Observe that the factor $\xi(\theta)$ can be cancel out by time rescaling. Similarly to the previous technique, to know the behavior of $X_{\alpha,\beta}$ near $\mathbb{S}^1$ (which is invariant) is the same than to know the behavior of $X$ near the origin. 

There is another blow up technique, known as \emph{Quasihomogeneous directional Blow Up}. This time we consider the change of coordinates 
\[\begin{cases}
	(x,y) \mapsto (x_1^\alpha, x_1^\beta y_1) \quad \text{positive $x$-direction} \\
	(x,y) \mapsto (-x_1^\alpha, x_1^\beta y_1) \quad \text{negative $x$-direction} \\	
	(x,y) \mapsto (x_1 y_1^\alpha, y_1^\beta) \quad \text{positive $y$-direction} \\
	(x,y) \mapsto (x_1 y_1^\alpha, -y_1^\beta) \quad \text{negative $y$-direction}, \\
\end{cases}\]
where $(\alpha,\beta)\in\mathbb{N}^2$. Again the pair $(\alpha,\beta)$ is called the \emph{weight} of the quasihomogeneous directional blow up. For example, if we do a Quasihomogeneous blow up at the positive $x$-direction, leading to a vector field $X_+^x$, then to understand the positive $x$-direction, i.e. $x>0$, of $X$ is the same than to understand the positive $x_1$-direction of vector field $X_+^x$. Similarly to the previous techniques, for some $k\in\mathbb{N}$ maximal one can divide $X_+^x$ and $X_-^x$ by $x^k$ (and $X_+^y$, $X_-^y$ by $y_k$) and study a regularized version of these systems. Moreover, these vectors fields (not regularized) are given by
\[\begin{array}{lll}
	\dot x_1 = \dfrac{1}{\alpha} x_1^{1-\alpha} \dot x, & \dot y_1 = \dfrac{x_1^\alpha \dot y - \frac{\beta}{\alpha}\dot x y}{x_1^{\alpha+\beta}} & \text{ at the positive $x$-direction,} \\
	&& \\
	\dot x_1 = -\dfrac{1}{\alpha} x_1^{1-\alpha} \dot x, & \dot y_1 = \dfrac{x_1^\alpha \dot y + \frac{\beta}{\alpha}\dot x y}{x_1^{\alpha+\beta}} & \text{ at the negative $x$-direction,} \\
	&& \\
	\dot x_1 = \dfrac{y_1^\beta \dot x - \frac{\alpha}{\beta}\dot y x}{y_1^{\alpha+\beta}}, & \dot y_1 = \dfrac{1}{\beta} y_1^{1 - \beta} \dot y & \text{ at the positive $y$-direction,} \\
	&& \\
	\dot x_1 = \dfrac{y_1^\beta \dot x + \frac{\alpha}{\beta}\dot y x}{y_1^{\alpha+\beta}}, & \dot y_1 = -\dfrac{1}{\beta} y_1^{1 - \beta} \dot y & \text{ at the negative $y$-direction.} \\
\end{array}\]

For a more detailed study of the blow up technique we refer to~\cite[Chapter~$3$]{DumLliArt2006} and~\cite{Alvarez}. For a nice generalization of this technique in which the infinity is ``exploded'' instead of the origin, we refer to the recent work~\cite{Otavio}.

\section{Markus-Neumann Theorem}\label{MN}

Let $X$ be a \emph{polynomial} vector field and $p(X)$ be its compactification defined on $\mathbb{D}$ and $\phi$ the flow defined by $p(X)$. The separatrices of $p(X)$ are given by the following.

\begin{enumerate}[label=\arabic*.]
	\item All the orbits contained in $\mathbb{S}^1$, i.e. at infinity.
	\item All the singular points.
	\item All the separatrices of the hyperbolic sectors of the finite and infinite singular points.
	\item All the limit cycles of $X$.
\end{enumerate}

Denote by $\mathcal{S}$ the set of all separatrices. It is known that $\mathcal{S}$ is closed, see for instance \cite{DumLliArt2006}. Each connected component of $\mathbb{D}\backslash\mathcal{S}$ is called a \emph{canonical region} of the flow $(\mathbb{D},\phi)$. The \emph{separatrix configuration} $\mathcal{S}_c$ of a flow $(\mathbb{D},\phi)$ is the union of all the separatrices $\mathcal{S}$ of the flow together with one orbit belonging to each canonical region. The separatrix configuration $\mathcal{S}_c$ of the flow $(\mathbb{D},\phi)$ is topologically equivalent to the separatrix configuration $\mathcal{S}_c^*$ of the flow $(\mathbb{D},\phi^*)$ if there exists a homeomorphism from $\mathbb{D}$ to $\mathbb{D}$ which transforms orbits of $\mathcal{S}_c$ into orbits of $\mathcal{S}_c^*$ and preserves or reverses the orientation of all these orbits.

\begin{theorem}[Markus-Neumann]
	Let $p(X)$ and $p(Y)$ be two Poin\-car\'e compactifications in the Poincar\'e disk $\mathbb{D}$ of two polynomial vector fields $X$ and $Y$ with finitely many singularities, respectively. Then the phase portraits of $p(X)$ and $p(Y)$ are topologically equivalent if and only if their separatrix configurations are topologically equivalent.
\end{theorem}

For a proof of this theorem see \cite{Mar1954,Neu1975,CorrectionMN}. For a nice digestion of the complete proof we also refer to the recent work~\cite{Braun}.

\begin{remark}
	In Figures~\ref{cod0and1} to \ref{23b2} we have written the separatrix configurations of the corresponding Poincar\'e compactifications.
\end{remark}

\section{Index of Singularities of a Vector Field}\label{index}

Let $p$ be an isolated singularity of a \emph{polynomial} vector field $X$. Let $e$ and $h$ denote the number of elliptical and hyperbolic sectors of $p$, respectively. The \emph{Poincar\'e index} of $p$ is given by
	\[i_p=\frac{e-h}{2}+1.\]
It is known that $i_p\in\mathbb{Z}$. See for instance~\cite[Chapter~$6$]{DumLliArt2006}.

\begin{proposition}\label{LC}
	Let $\Gamma$ be a limit cycle of a planar polynomial vector field $X$. Then there is at least one singularity in the bounded region limited by it. Moreover if there is a finite number of singularities in the bounded region limited by $\Gamma$, then the sum of their Poincar\'e indices is $1$.
\end{proposition}

\begin{theorem}[Poincar\'e-Hopf Theorem]\label{PH}
	Let $X$ be a planar polynomial vector field and $p(X)$ its compactification defined on $\mathbb{S}^2$. If $p(X)$ has a finite number of singularities, then the sum of their Poincar\'e indices is $2$.
\end{theorem}
For a proof of Proposition~\ref{LC} and Theorem~\ref{PH} we refer to~\cite[Chapter~$6$]{DumLliArt2006}. 

\section{The Cardano-Tartaglia Formula}\label{TGSection}

Given a cubic equation
\begin{equation}\label{TG}
	x^3+px +q=0,
\end{equation}
if we make the change of variable $x=u+v$ on \eqref{TG} we obtain
	\[u^3+v^3+(3uv+p)(u+v)+q=0.\]
Assuming that $3uv+p=0$ one obtain $u^3+v^3=-q$. The combination of the equations
	\[u^3v^3=-\frac{1}{27}p^3, \quad u^3+v^3=-q\]
imply that $u^3$ and $v^3$ are the zeros of 
	\[z^2+qz-\frac{1}{27}p^3=0.\]
Therefore,
	\[u^3=-\frac{q}{2}+\left(\frac{q^2}{4}+\frac{p^3}{27}\right)^\frac{1}{2}, \quad v^3=-\frac{q}{2}-\left(\frac{q^2}{4}+\frac{p^3}{27}\right)^\frac{1}{2}.\]

\begin{remark}
	We observe that $x^\frac{1}{n}$ denotes the standard $nth$ root of $x$, i.e. if $x=re^{i\theta}$, then $x^\frac{1}{n}=\sqrt[n]{r}e^{i\frac{\theta}{n}}$, where $\sqrt[n]{x}$ denotes the $nth$ real root (when it exists) of $x$. For example, $8^\frac{1}{3}=\sqrt[3]{8}=2$, but $(-8)^\frac{1}{3}=1+i\sqrt{3}$ while $\sqrt[3]{-8}=-2$. In general both functions coincide when $x$ is a non-negative real number.
\end{remark}

Let $D=\frac{1}{4}q^2+\frac{1}{27}p^3$ and observe that if $D\geqslant0$, then $u^3$ and $v^3$ are real numbers and thus we can take
	\[\mu=\sqrt[3]{u^3}, \quad \nu=\sqrt[3]{v^3}.\]
Observe that these values do satisfy $\mu^3+\nu^3=-q$ and $3\mu\nu+p=0$. Hence 
	\[x=\mu+\nu=\sqrt[3]{-\frac{q}{2}+\left(\frac{q^2}{4}+\frac{p^3}{27}\right)^\frac{1}{2}}+\sqrt[3]{-\frac{q}{2}-\left(\frac{q^2}{4}+\frac{p^3}{27}\right)^\frac{1}{2}}\]
is a (real) solution of \eqref{TG}. Let $\xi = e^{i\frac{\pi}{3}}$ be a root of the unity and observe that the three solutions of \eqref{TG}, when $D\geqslant0$, are given by
	\[x_k=\xi^k\sqrt[3]{-\frac{q}{2}+\left(\frac{q^2}{4}+\frac{p^3}{27}\right)^\frac{1}{2}}+\xi^{2k}\sqrt[3]{-\frac{q}{2}-\left(\frac{q^2}{4}+\frac{p^3}{27}\right)^\frac{1}{2}},\]
where $k\in\{0,1,2\}$. If $D<0$ we cannot take any third root of $u^3$ and $v^3$ because not every choosing of roots satisfy $3\mu\nu+p=0$. To satisfy this it is enough that $u$ be the conjugate of $v$. Therefore we can take
	\[\mu=\left(u^3\right)^\frac{1}{3}, \quad \nu=\left(v^3\right)^\frac{1}{3}.\]
Hence the solutions of \eqref{TG} when $D<0$ are given by 
	\[x_k=\xi^k\left(-\frac{q}{2}+\left(\frac{q^2}{4}+\frac{p^3}{27}\right)^\frac{1}{2}\right)^\frac{1}{3}+\xi^{2k}\left(-\frac{q}{2}-\left(\frac{q^2}{4}+\frac{p^3}{27}\right)^\frac{1}{2}\right)^\frac{1}{3},\]
where $k\in\{0,1,2\}$. Observe that if we take these roots when $D\geqslant0$, then $\mu$ is not necessarily the conjugate of $\nu$ (for example if both $u^3$ and $v^3$ are negative real numbers) and thus we would not have a solution of \eqref{TG}. At the other hand, we cannot always apply the $\sqrt[3]{\cdot}$ when $D<0$ because not every complex number has a third real root.

If we denote by $S$ and $T$ the \emph{convenient root} of $u^3$ and $v^3$ (i.e. denotes the root according to the logic of this section), respectively, then the solutions of \eqref{TG} are given by
\[\begin{cases}
	x_1=S+T. \\
	x_2=-\frac{1}{2}(S+T)+\frac{1}{2}\sqrt{3}(S-T)i. \\
	x_3=-\frac{1}{2}(S+T)-\frac{1}{2}\sqrt{3}(S-T)i.
\end{cases}\]
One can see that if $D<0$, then all solutions are real and simple. If $D=0$ and $q\neq0$, then $x_1$ is a real simple solution and $x_2=x_3$ is a double real solution. If $D=q=0$, then $x_1=x_2=x_3$ is a real triple solution. Finally if $D>0$, then $x_1$ is a real simple solution and $x_2$, $x_3$ are two complex conjugates solutions.

\begin{figure}[ht]
	\begin{center}
		\begin{minipage}{3cm}
			\begin{center}
				\includegraphics[width=3cm]{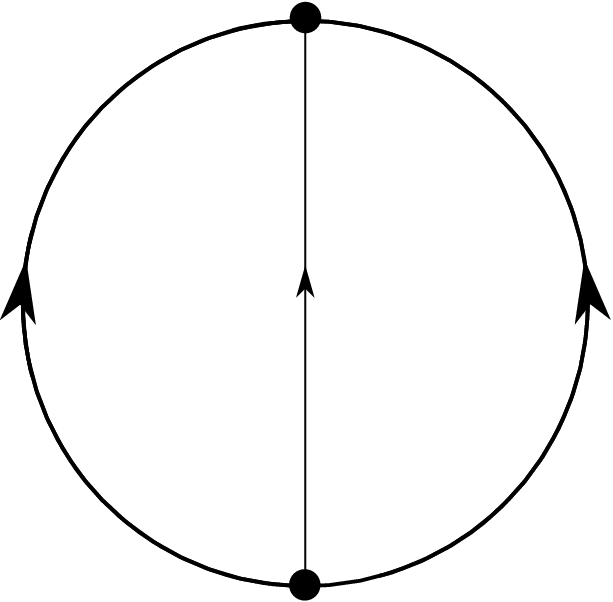}
				
				$X_{01}$
			\end{center}	
		\end{minipage}
		\begin{minipage}{3cm}
			\begin{center}
				\includegraphics[width=3cm]{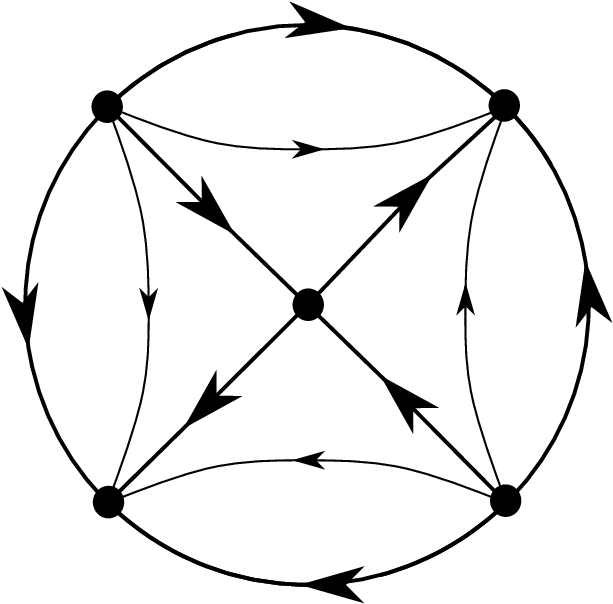}
				
				$X_{02}$, $\delta=1$
			\end{center}
		\end{minipage}
		\begin{minipage}{3cm}
			\begin{center}
				\includegraphics[width=3cm]{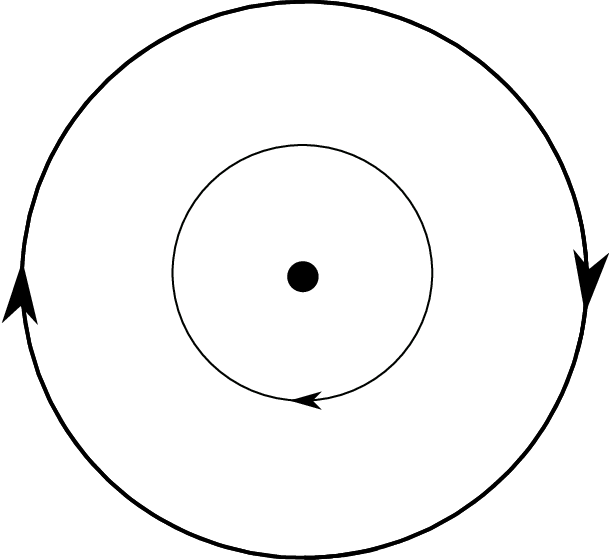}
				
				$X_{02}$, $\delta=-1$
			\end{center}
		\end{minipage}
	\end{center}
	
	\begin{center}
		\begin{minipage}{3cm}		
			\begin{center}
				\includegraphics[width=3cm]{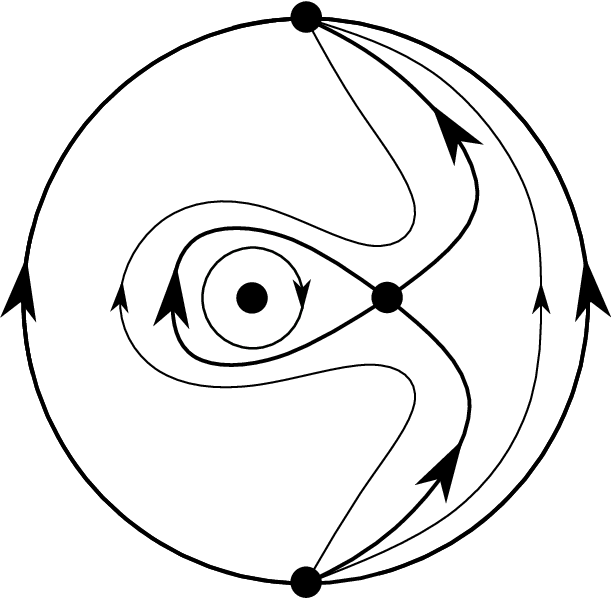}
				
				$X_{11}$, $\lambda<0$
			\end{center}		
		\end{minipage}
		\begin{minipage}{3cm}	
			\begin{center}
				\includegraphics[width=3cm]{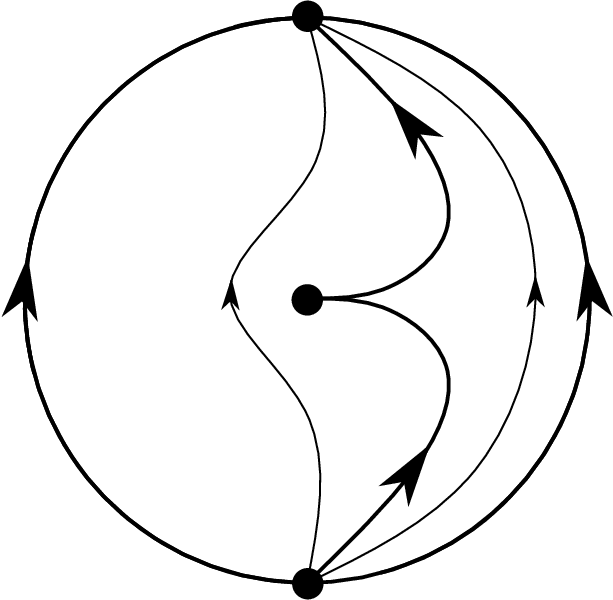}
				
				$X_{11}$, $\lambda=0$
			\end{center}		
		\end{minipage}
		\begin{minipage}{3cm}		
			\begin{center}
				\includegraphics[width=3cm]{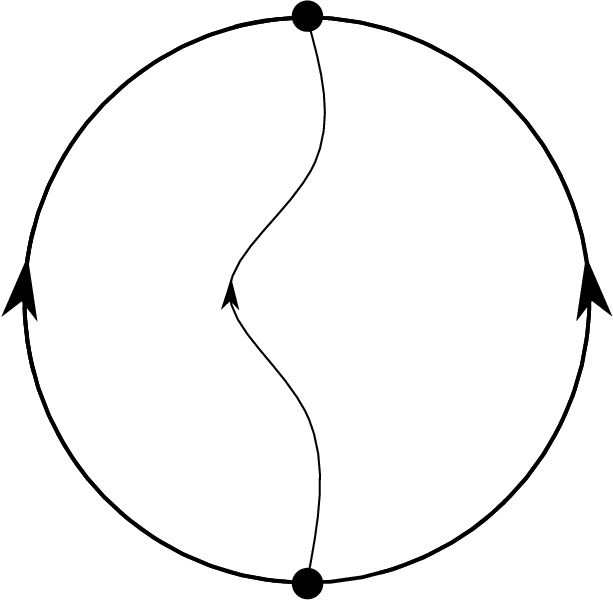}
				
				$X_{11}$, $0<\lambda$
			\end{center}		
		\end{minipage}
	\end{center}
	
	\begin{center}
		\begin{minipage}{3cm}
			\begin{center}
				\includegraphics[width=3cm]{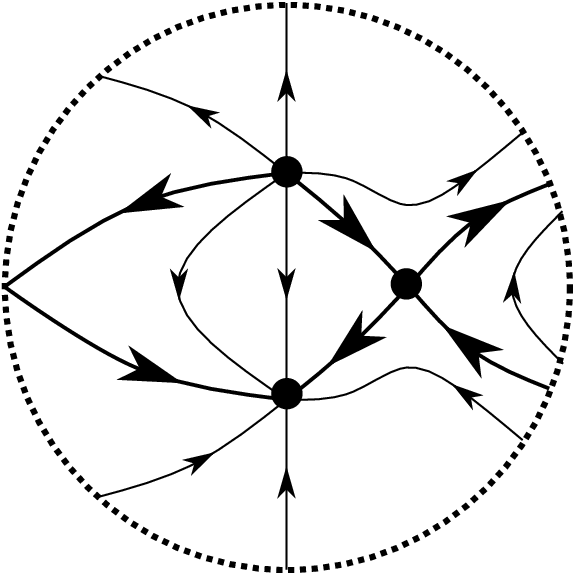}
				
				$X_{12}$
				
				$\lambda<0, \; \delta=1$
			\end{center}		
		\end{minipage}
		\begin{minipage}{3cm}		
			\begin{center}
				\includegraphics[width=3cm]{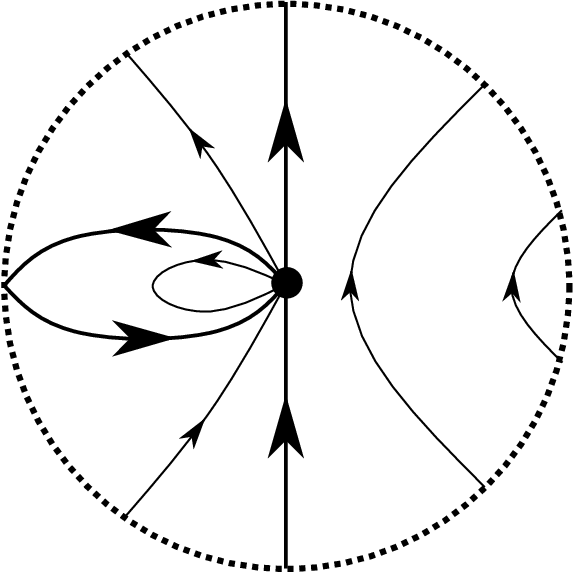}
				
				$X_{12}$
				
				$\lambda=0, \; \delta=1$
			\end{center}		
		\end{minipage}
		\begin{minipage}{3cm}		
			\begin{center}
				\includegraphics[width=3cm]{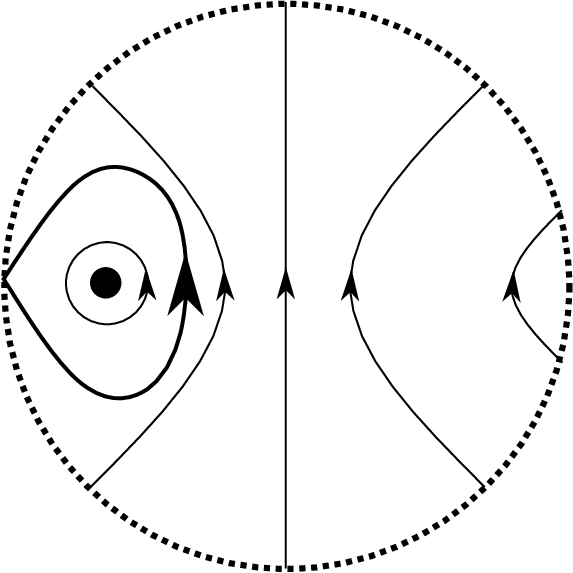}
				
				$X_{12}$
				
				$0<\lambda, \; \delta=1$
			\end{center}		
		\end{minipage}
	\end{center}
	
	\begin{center}
		\begin{minipage}{3cm}	
			\begin{center}
				\includegraphics[width=3cm]{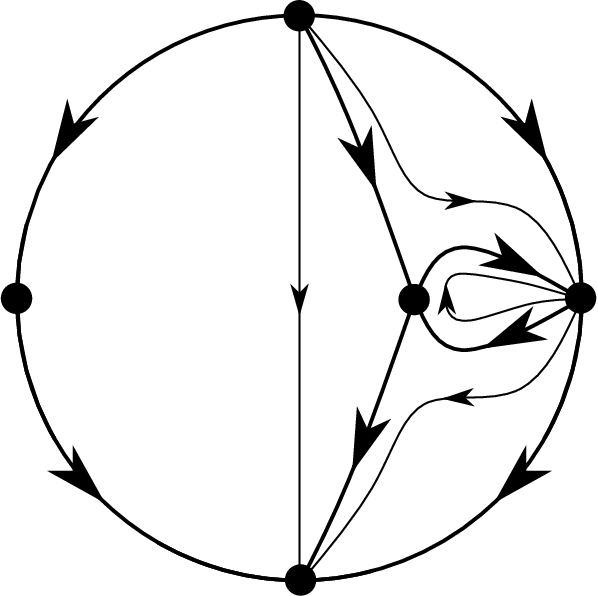}
				
				$X_{13}$, $\lambda<0$
			\end{center}		
		\end{minipage}
		\begin{minipage}{3cm}		
			\begin{center}
				\includegraphics[width=3cm]{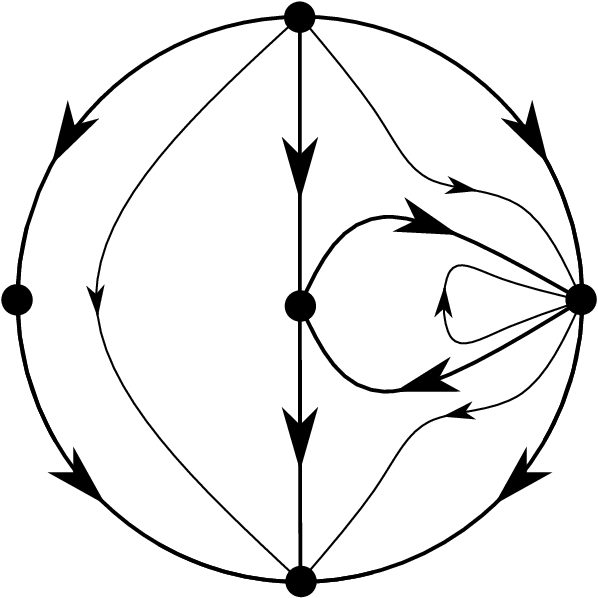}
				
				$X_{13}$, $\lambda=0$
			\end{center}		
		\end{minipage}
		\begin{minipage}{3cm}	
			\begin{center}
				\includegraphics[width=3cm]{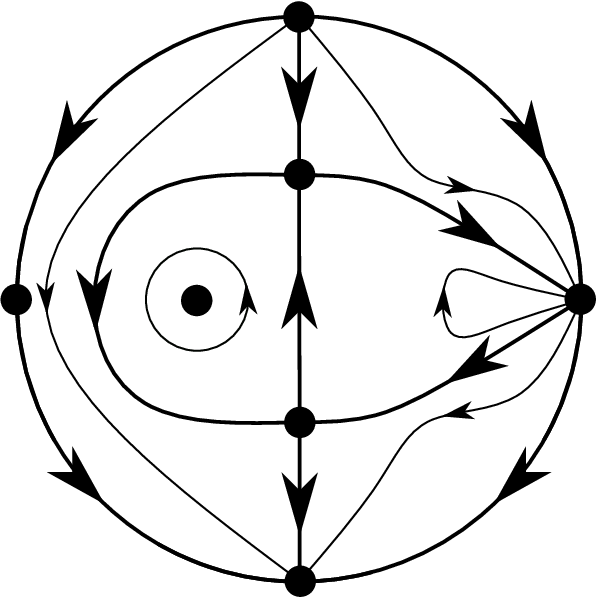}
				
				$X_{13}$, $0<\lambda$
			\end{center}		
		\end{minipage}
	\end{center}
	
	\begin{center}
		\begin{minipage}{3cm}
			\begin{center}
				\includegraphics[width=3cm]{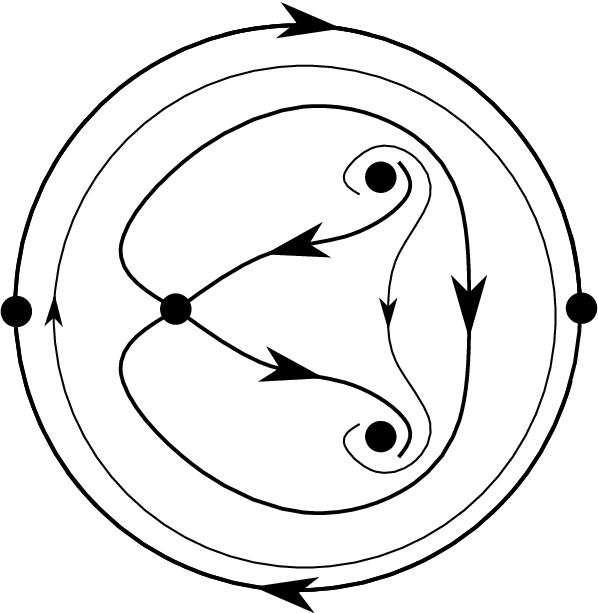}
				
				$X_{14}$, $\lambda<0$
			\end{center}		
		\end{minipage}	
		\begin{minipage}{3cm}	
			\begin{center}
				\includegraphics[width=3cm]{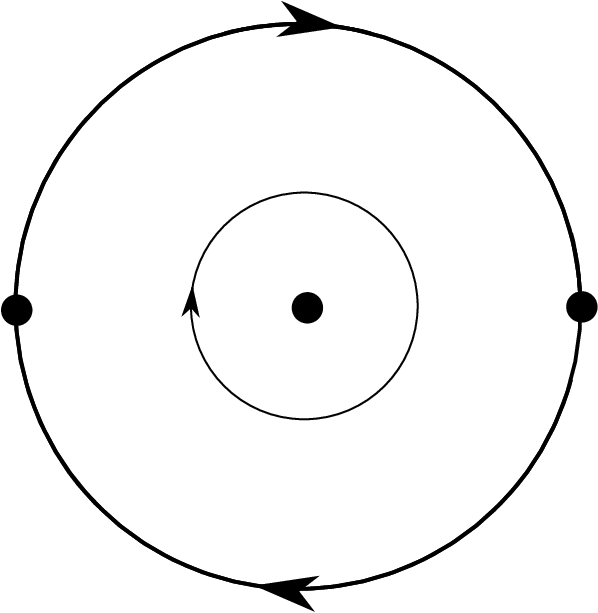}
				
				$X_{14}$, $0\leqslant\lambda$
			\end{center}
		\end{minipage}
	\end{center}
	\caption{Phase portraits of $X_{01}$, $X_{02}$, $X_{11}$, $X_{12}$, $X_{13}$ and $X_{14}$.}\label{cod0and1}
\end{figure}

\begin{figure}[ht]	
	\begin{center}
		\begin{overpic}[width=12cm]{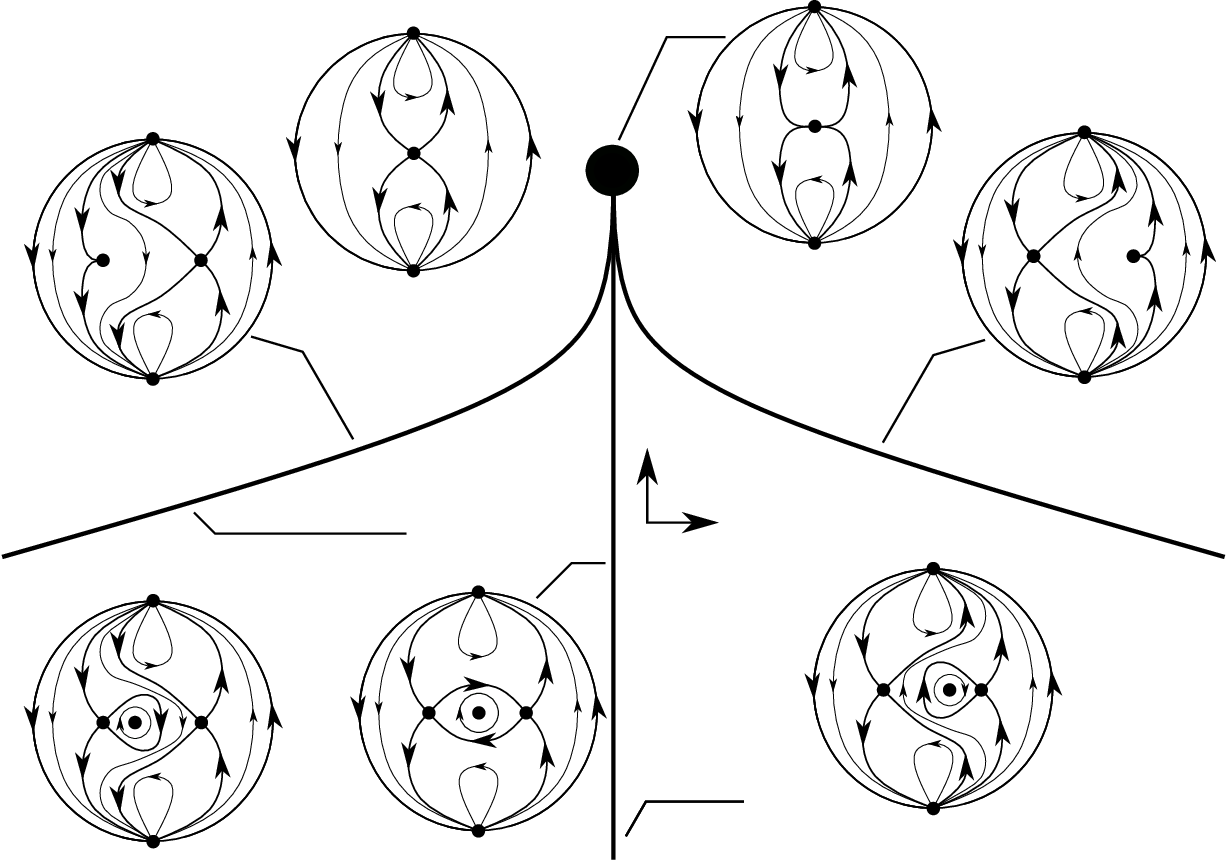} 
			\put(53.5,32){$\beta$}
			\put(58,28){$\alpha$}
			\put(54,5.5){$\alpha=0$}
			\put(20,27.5){$27\alpha^2+4\beta^3=0$}
		\end{overpic}
	\end{center}	
	\caption{Phase portrait of $X_{21}$ with $b=1$.}\label{21a}
	\vspace{0.5cm}
	\begin{center}
		\begin{overpic}[width=12cm]{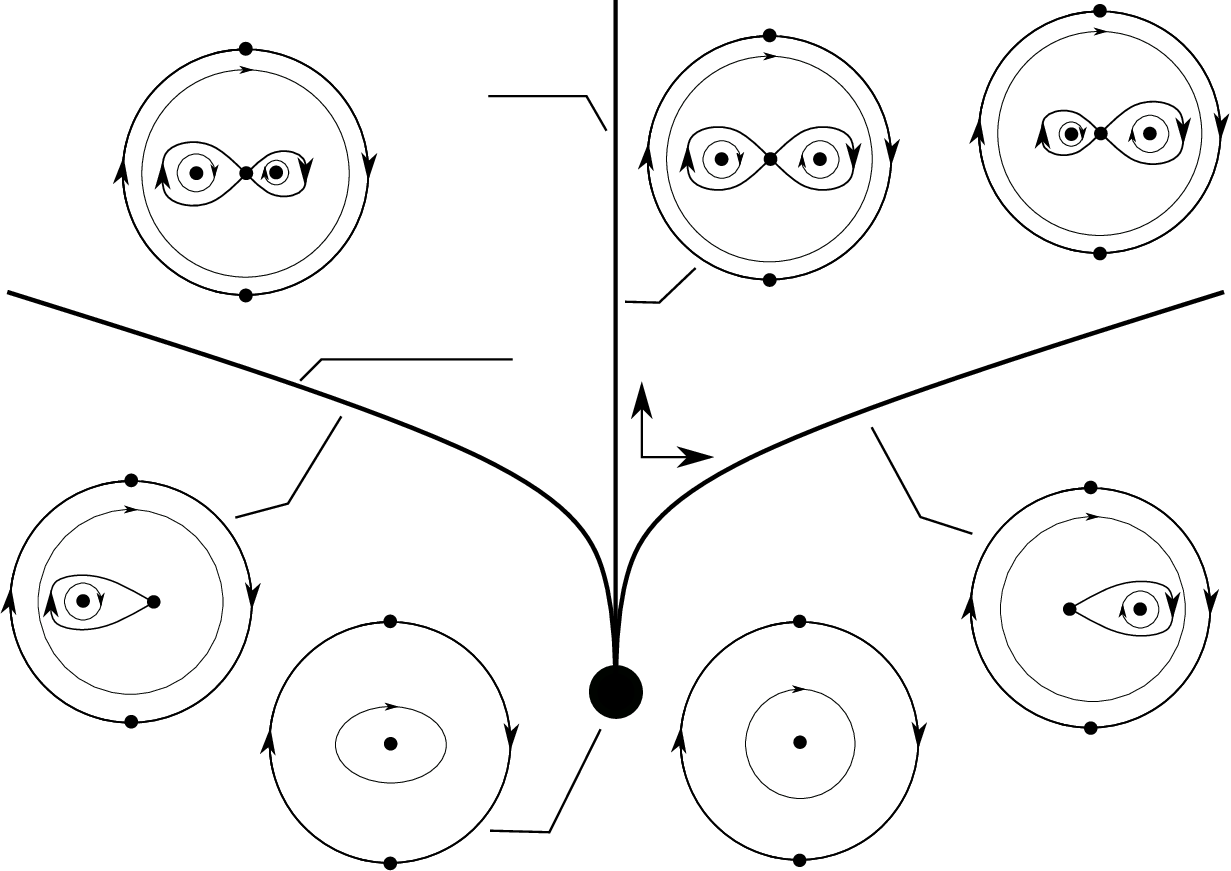} 
			\put(53,38){$\beta$}
			\put(57,34){$\alpha$}
			\put(39,63.5){$\alpha=0$}
			\put(26,42.5){$27\alpha^2+4\beta^3=0$}
		\end{overpic}
	\end{center}	
	\caption{Bifurcation diagram of $X_{21}$ with $b=-1$.}\label{21b}
\end{figure}

\begin{figure}[ht]	
	\begin{center}
	\begin{overpic}[width=12cm]{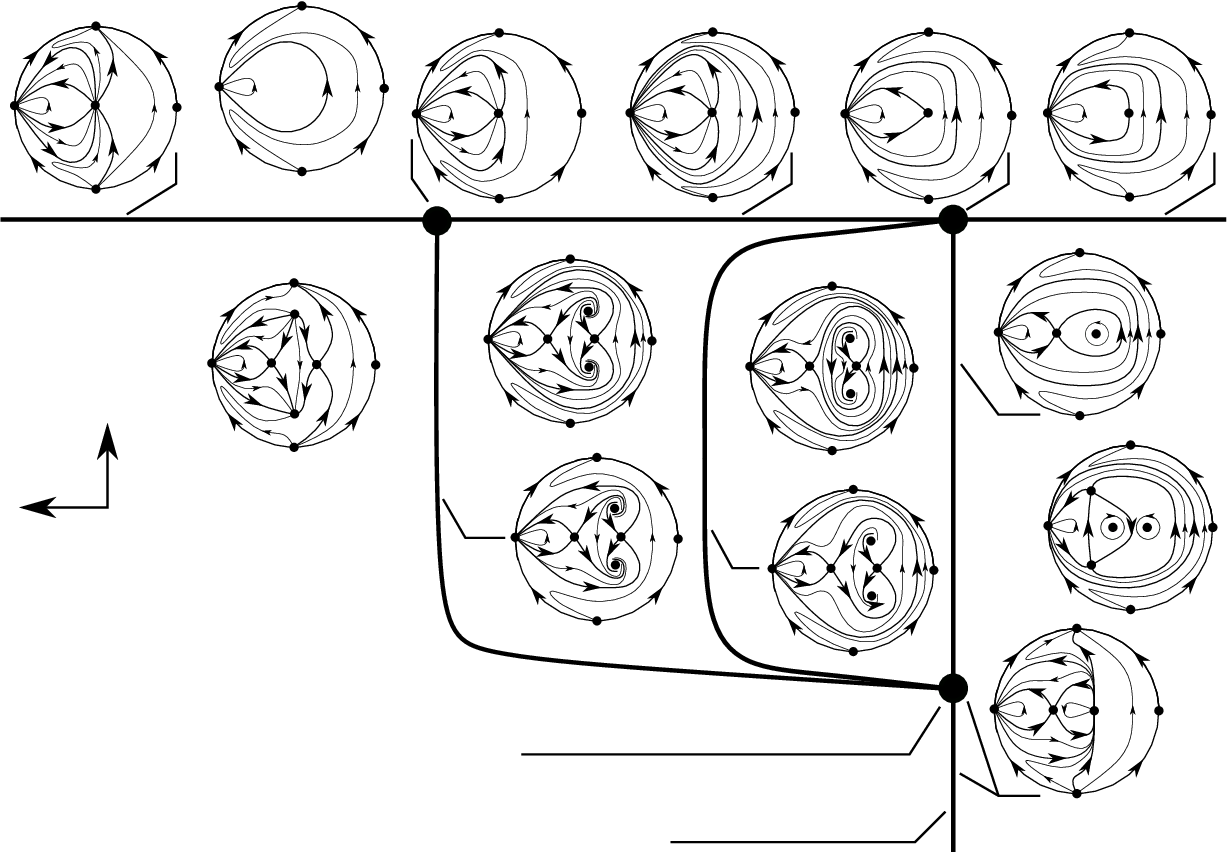} 
			\put(6,33){$\alpha$}
			\put(2,29){$\beta$}
			\put(55,2){$\beta=-1$}
			\put(1,49){$\alpha=0$}
			\put(42,9){$(\alpha,\beta)=(-16,-1)$}
		\end{overpic}
	\end{center}	
	\caption{Bifurcation diagram of $X_{22a}$ with $a=1$.}\label{22a}
	\vspace{0.5cm}
	\begin{center}
		\begin{overpic}[width=12cm]{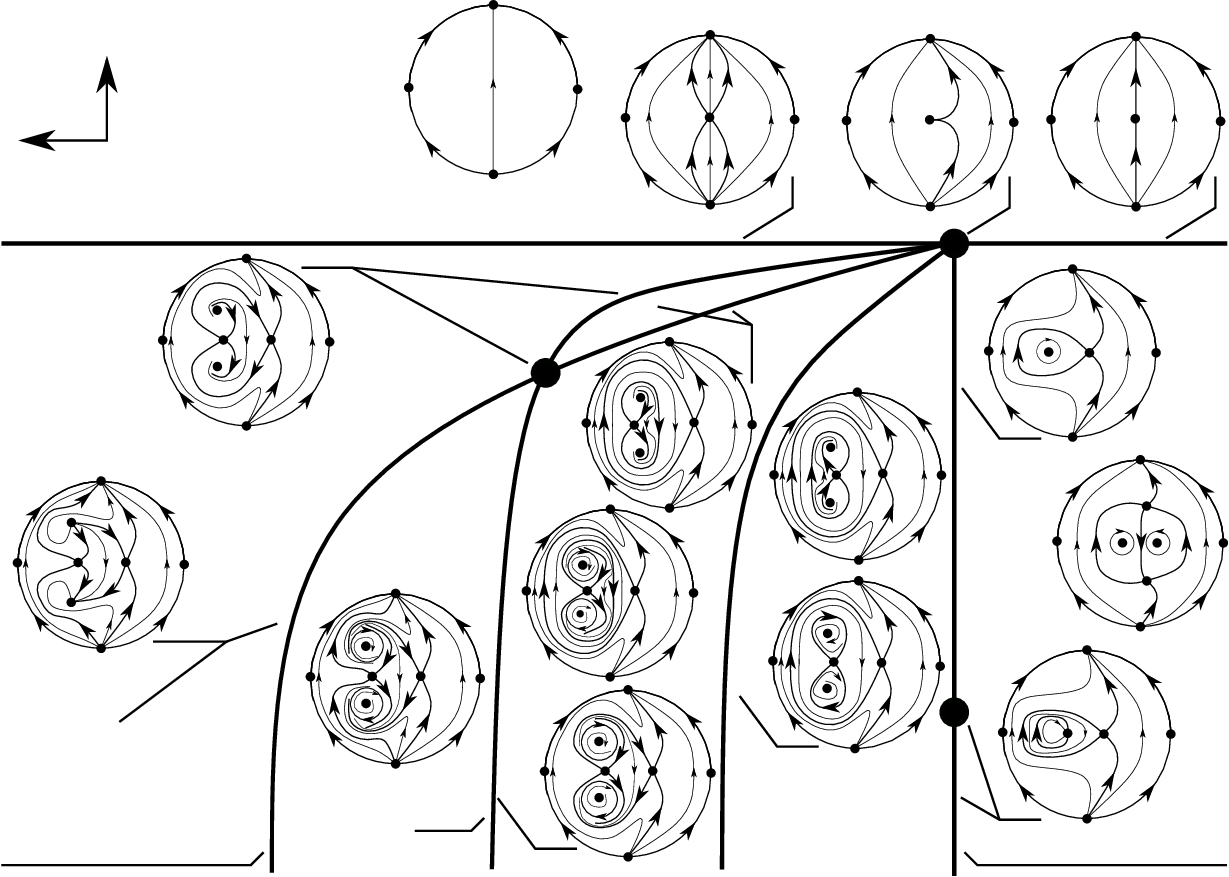} 
			\put(6,65){$\alpha$}
			\put(2,61){$\beta$}
			\put(0,2){$\beta=1+2\sqrt{-\alpha}$}
			\put(93,2){$\beta=1$}
			\put(1,49){$\alpha=0$}
			\put(25.5,3){$\beta_1(\alpha)$}
		\end{overpic}
	\end{center}	
	\caption{Bifurcation diagram of $X_{22a}$ with $a=-1$.}\label{22b}
\end{figure}

\begin{figure}[ht]	
	\begin{center}
		\begin{overpic}[width=12cm]{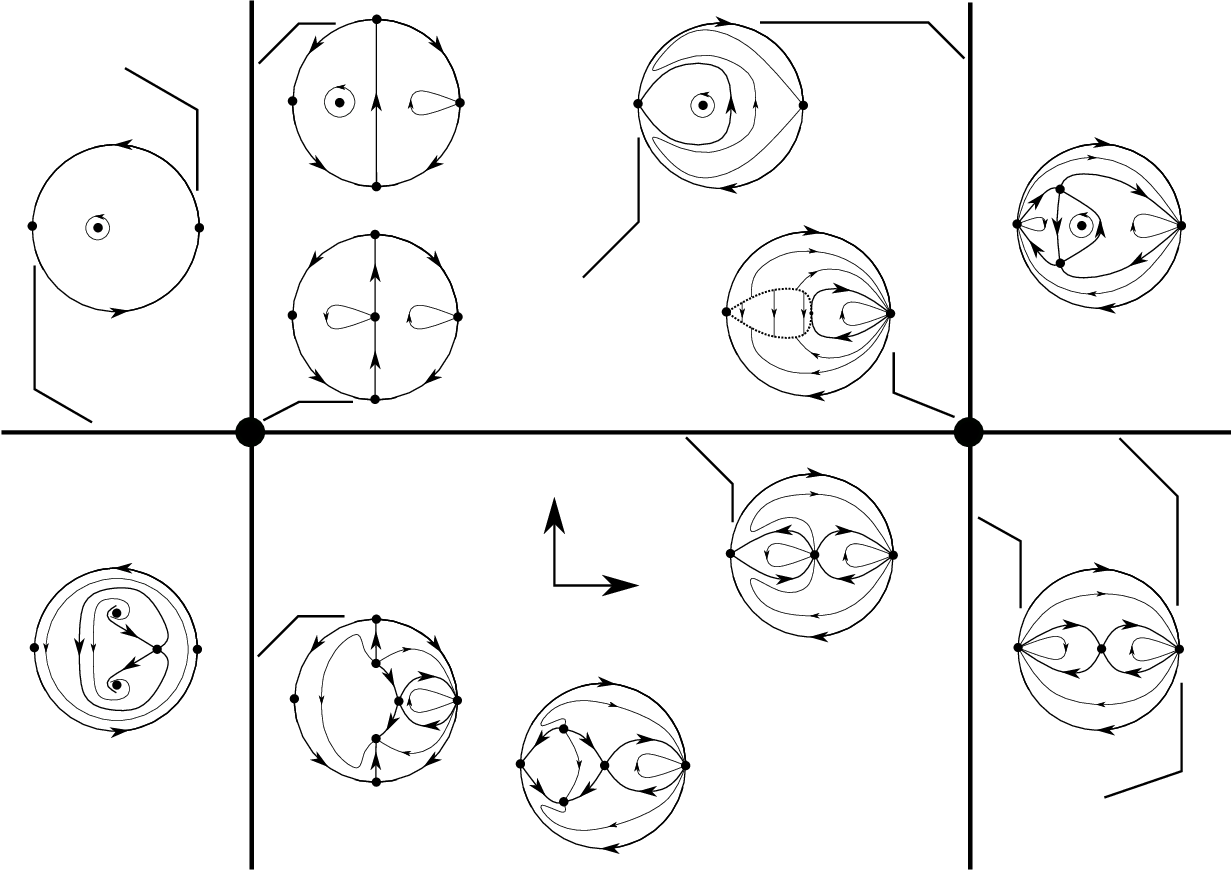} 
			\put(50,24){$\alpha$}
			\put(46,29){$\beta$}
			\put(21,1){$\alpha=0$}
			\put(79.5,1){$\alpha=1$}
			\put(90,37){$\beta=0$}
		\end{overpic}
	\end{center}	
	\caption{Bifurcation diagram of $X_{24a}$ with $a=1$.}\label{24a}
	\vspace{0.5cm}	
	\begin{center}
		\begin{overpic}[width=12cm]{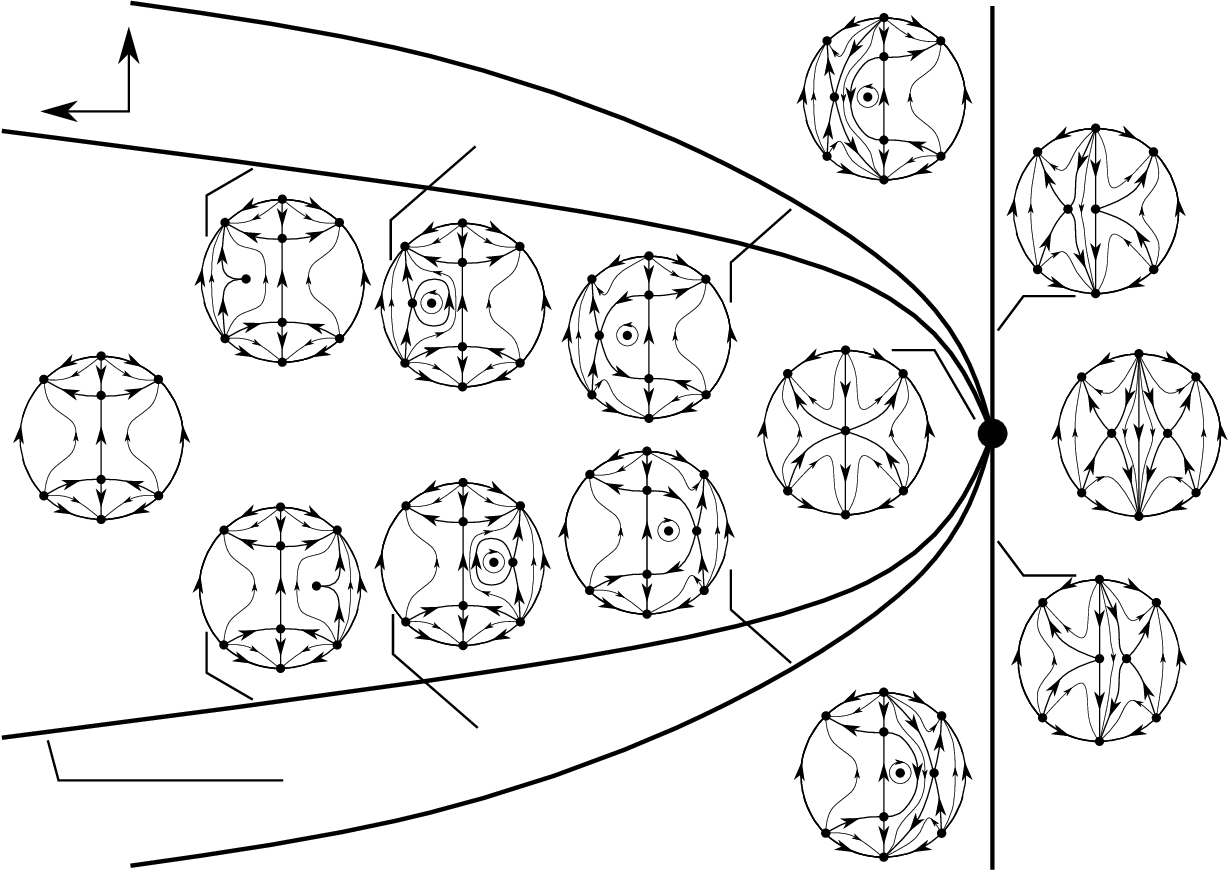} 
			\put(8,9){$\beta=\frac{1}{4}\alpha^2$}
			\put(4,63){$\beta$}
			\put(7.5,67){$\alpha$}
			\put(81,1){$\beta=0$}
		\end{overpic}
	\end{center}	
	\caption{Bifurcation diagram of $X_{25a}$ with $a=1$.}\label{25a1}
\end{figure}

\begin{figure}[ht]	
	\begin{center}
		\begin{overpic}[width=12cm]{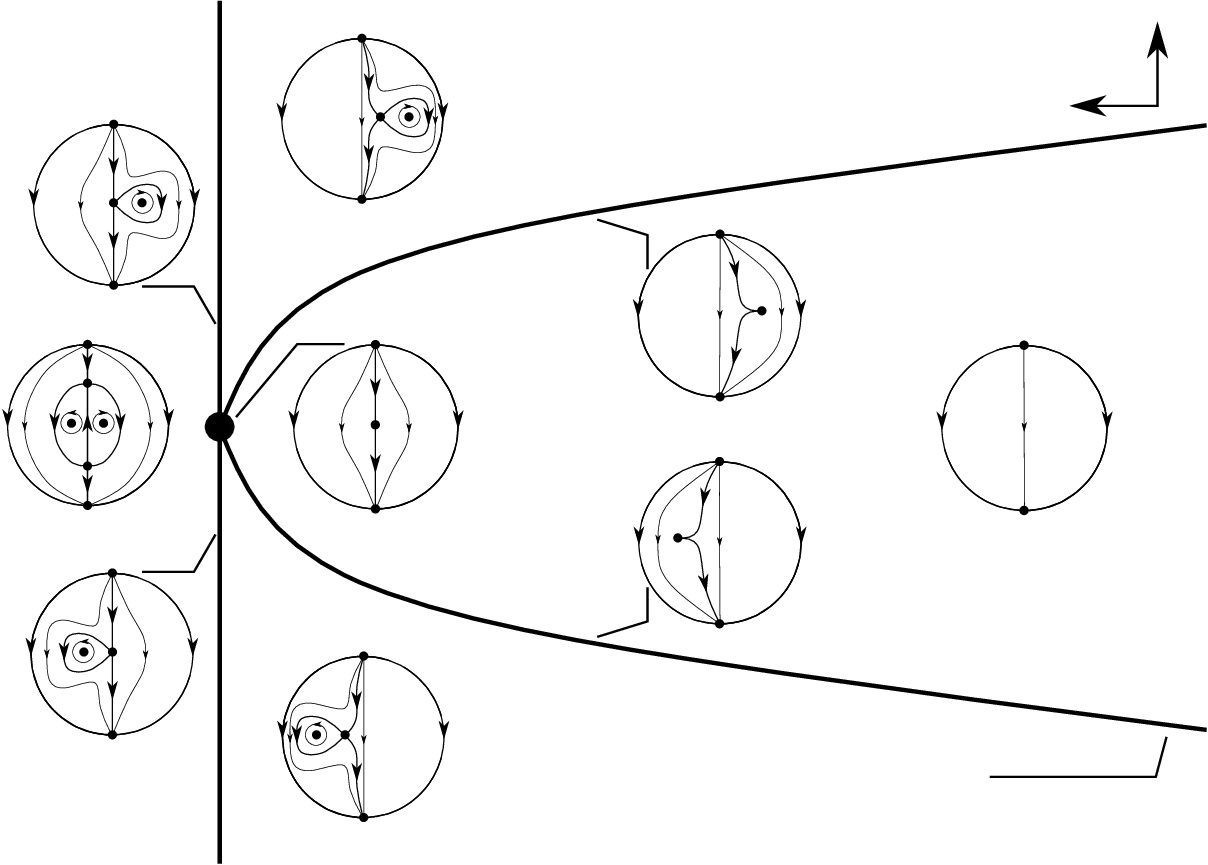} 
			\put(80,8.5){$\beta=-\frac{1}{4}\alpha^2$}
			\put(89,64){$\beta$}
			\put(93,68){$\alpha$}
			\put(10,69){$\beta=0$}
		\end{overpic}
	\end{center}	
	\caption{Bifurcation diagram of $X_{25a}$ with $a=-1$.}\label{25a2}
	\vspace{0.5cm}	
	\begin{center}
		\begin{overpic}[width=12cm]{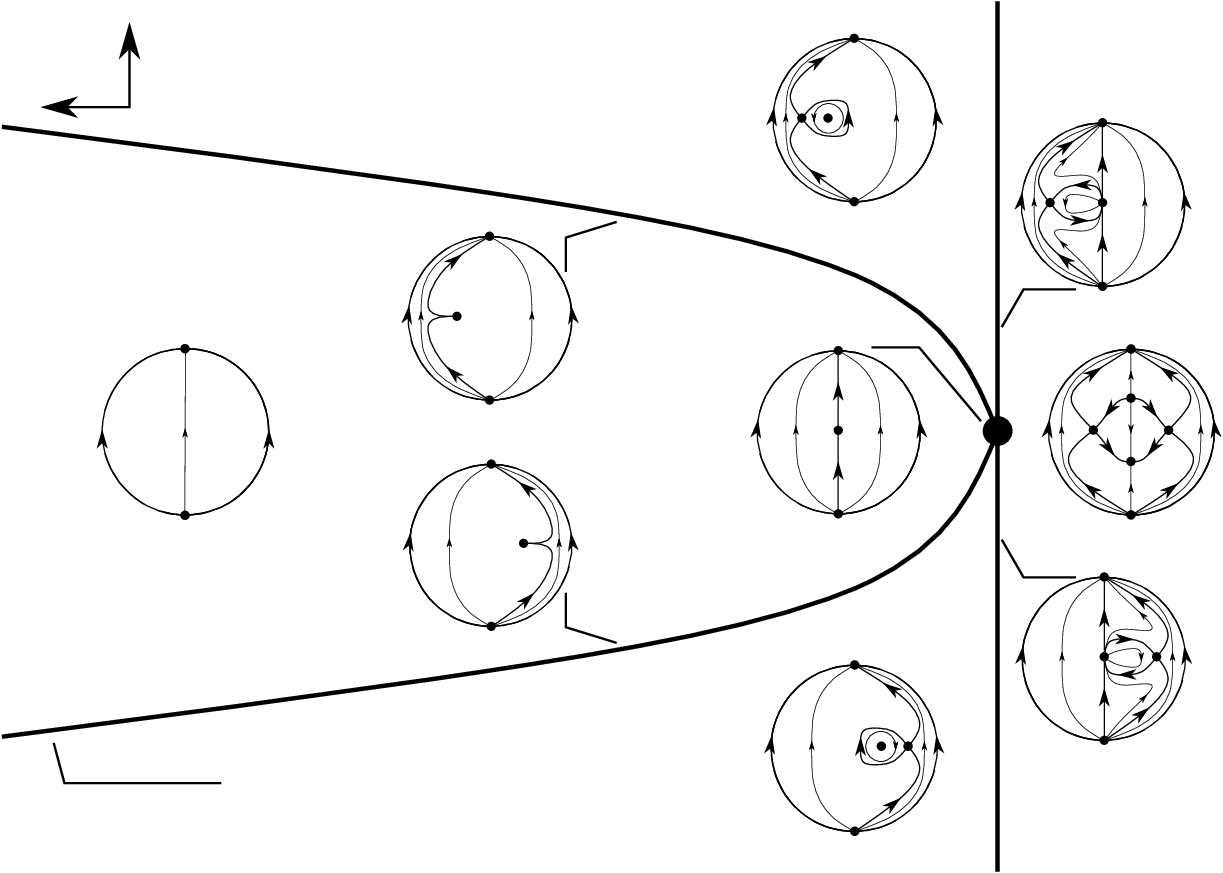} 
			\put(8,68){$\alpha$}
			\put(4,64){$\beta$}
			\put(6,9){$\beta=\frac{1}{12}\alpha^2$}
			\put(82,1){$\beta=0$}
		\end{overpic}
	\end{center}	
	\caption{Bifurcation diagram of $X_{25b}$ with $(a,b,\delta)=(1,3,3)$.}\label{25b1}
\end{figure}

\begin{figure}[ht]	
	\begin{center}
		\begin{overpic}[width=12cm]{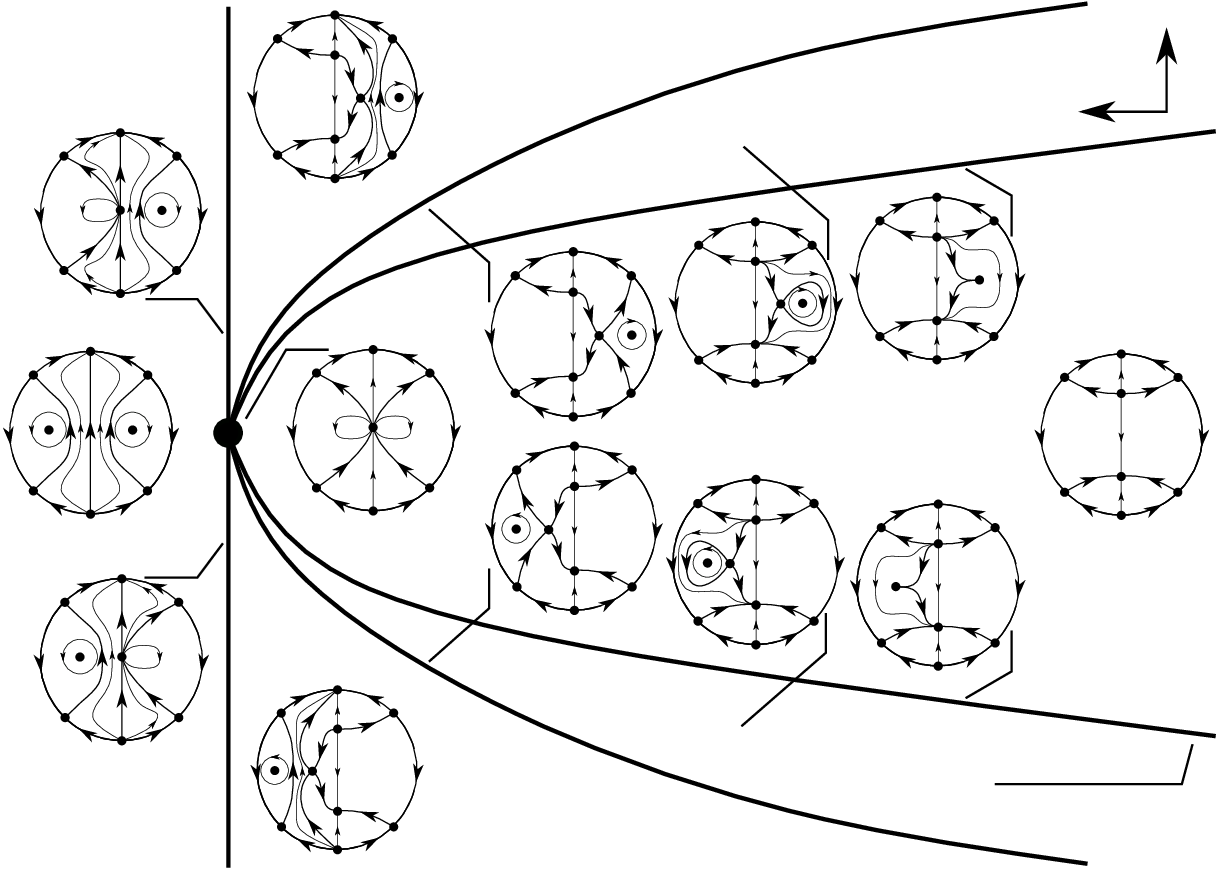} 
			\put(93,68){$\alpha$}
			\put(89,63.5){$\beta$}
			\put(80,8.5){$\beta=-\frac{1}{12}\alpha^2$}
			\put(10,69){$\beta=0$}
		\end{overpic}
	\end{center}	
	\caption{Bifurcation diagram of $X_{25b}$ with $(a,b,\delta)=(1,3,-3)$.}\label{25b2}
	\vspace{0.5cm}	
	\begin{center}
		\begin{overpic}[width=12cm]{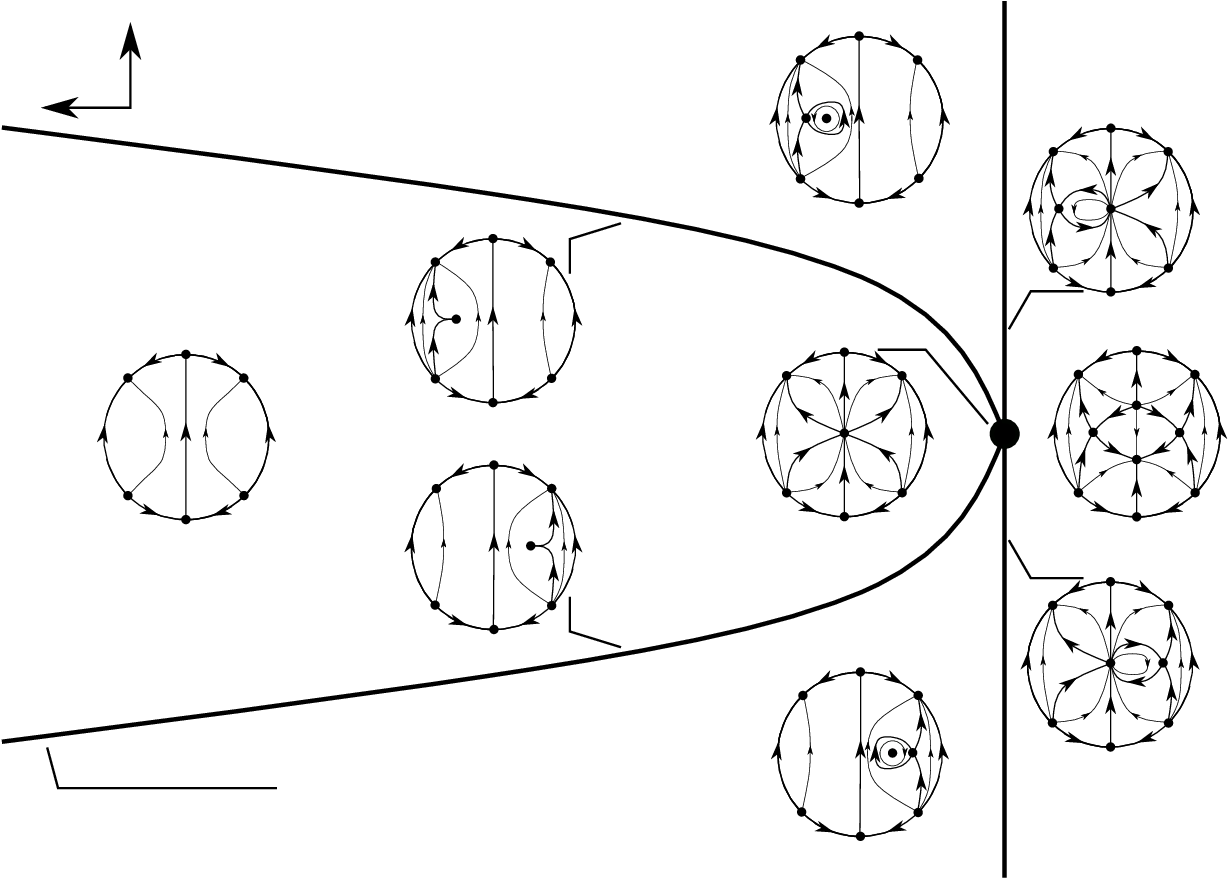} 
			\put(7.5,68){$\alpha$}
			\put(4,64){$\beta$}
			\put(8,9){$\beta=\frac{1}{12}\alpha^2$}
			\put(82,1){$\beta=0$}
		\end{overpic}
	\end{center}	
	\caption{Bifurcation diagram of $X_{25b}$ with $(a,b,\delta)=(3,1,3)$.}\label{25b3}
\end{figure}

\begin{figure}[ht]	
	\begin{center}
		\begin{overpic}[width=12cm]{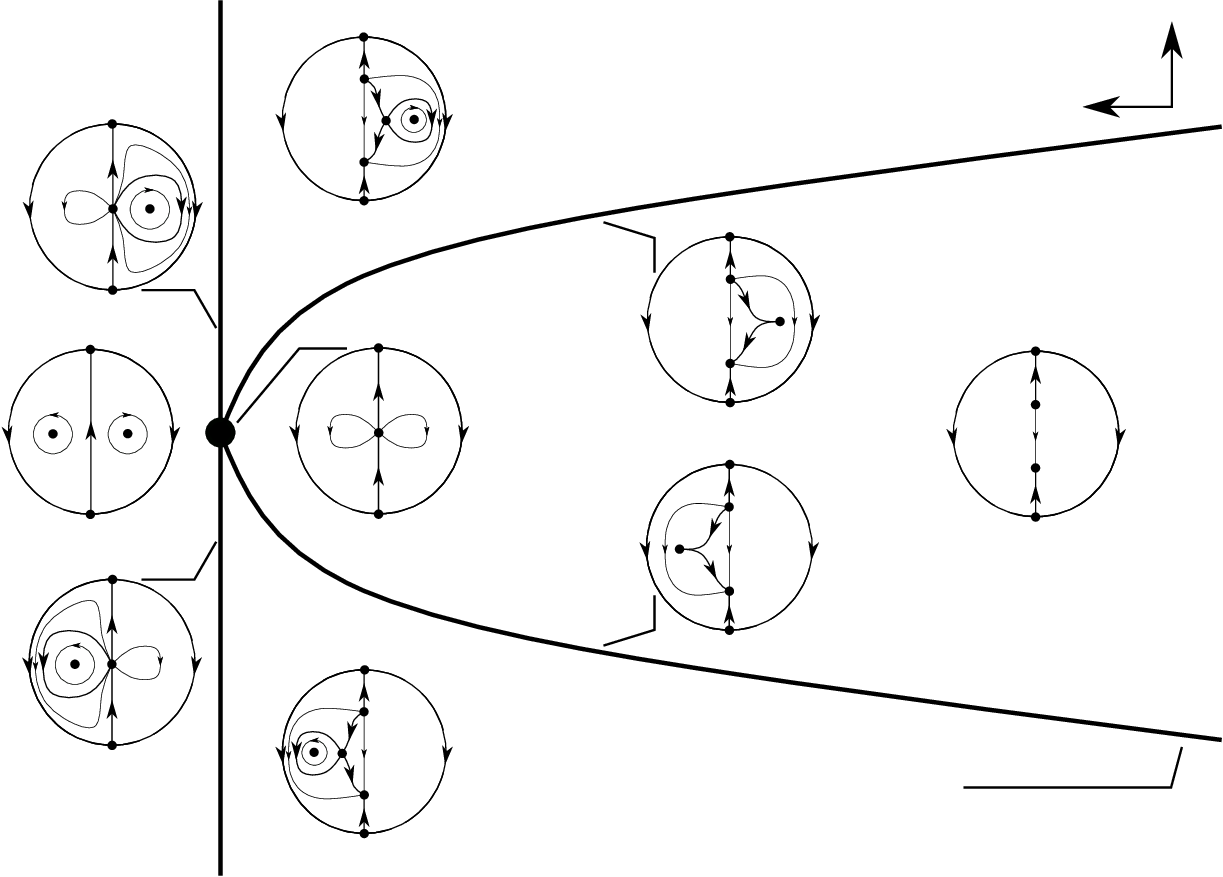} 
			\put(93,68){$\alpha$}
			\put(89,64){$\beta$}
			\put(79,8.5){$\beta=-\frac{1}{12}\alpha^2$}
			\put(9.5,69){$\beta=0$}
		\end{overpic}
	\end{center}	
	\caption{Bifurcation diagram of $X_{25b}$ with $(a,b,\delta)=(3,1,-3)$.}\label{25b4}
	\begin{center}
		\begin{overpic}[width=12cm]{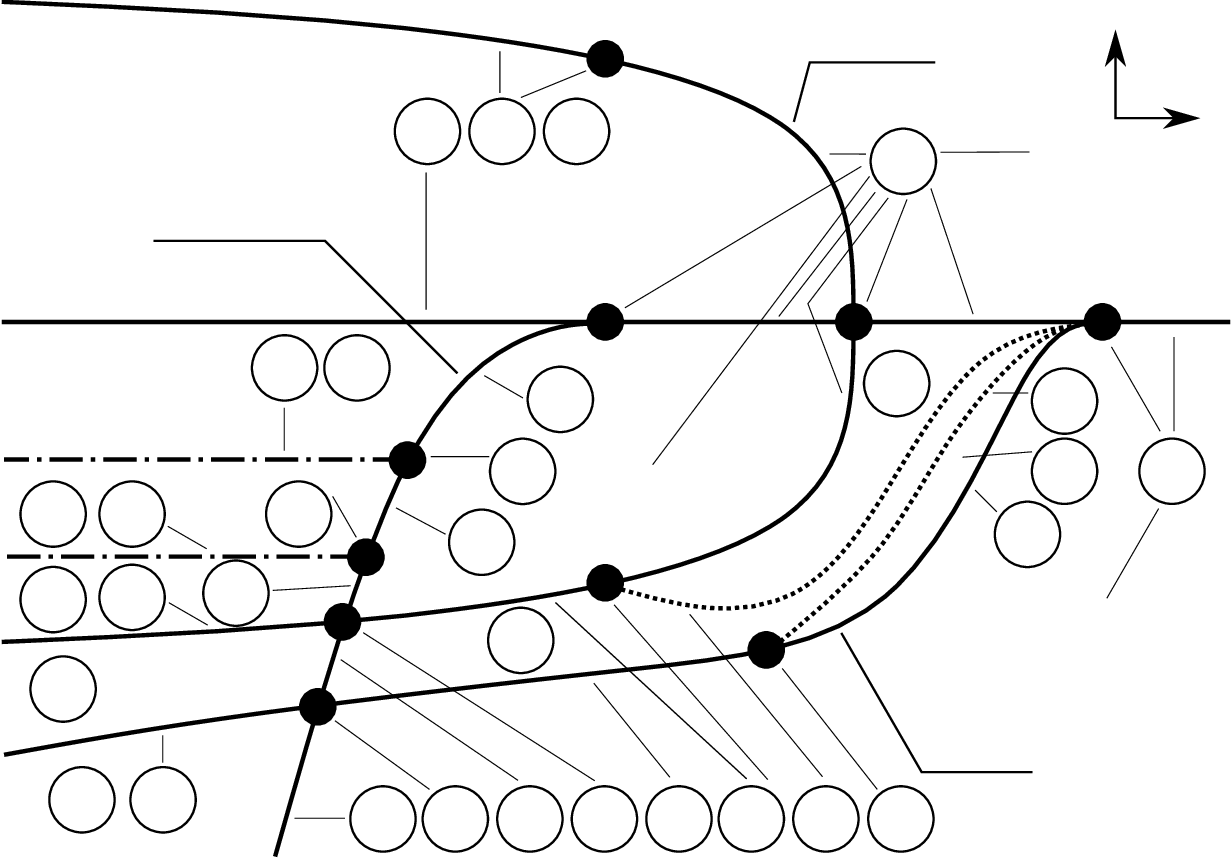} 
			\put(92,65){$\beta$}
			\put(96,61){$\alpha$}
			\put(93,44.5){$\beta=0$}
			\put(66,65.5){$\alpha+\beta^2=0$}
			\put(76,8){$\beta_+(\alpha)$}
			\put(14,51){$\beta_-(\alpha)$}
			\put(72.5,55.5){$1$}
			\put(40,58){$2$}
			\put(46,58){$3$}
			\put(34,58){$4$}
			\put(28,39){$5$}
			\put(22,39){$6$}
			\put(3.5,27){$7$}
			\put(10,27){$8$}
			\put(3.5,20){$9$}
			\put(9,20){$10$}
			\put(3.5,12.5){$11$}
			\put(11.5,3.5){$12$}
			\put(5,3.5){$13$}
			\put(44,36){$14$}
			\put(41,30.5){$15$}
			\put(37.5,24.5){$16$}
			\put(22.5,27){$17$}
			\put(17.5,20.5){$18$}
			\put(47.5,2){$19$}
			\put(41.5,2){$20$}
			\put(35.5,2){$21$}
			\put(29.5,2){$22$}
			\put(93.5,30.5){$23$}
			\put(82,25){$24$}
			\put(71.5,2){$25$}
			\put(53.5,2){$26$}
			\put(85,30.5){$27$}
			\put(85,36){$28$}
			\put(40.75,16.75){$29$}
			\put(65.5,2){$30$}
			\put(71.25,37.25){$31$}
			\put(59.5,2){$32$}
		\end{overpic}
	\end{center}	
	\caption{Bifurcation diagram of $X_{23}$ with $a=1$. We observe that it may be an intersection between the dotted lines defined by phase portraits $28$ and $30$. This is an open problem.}\label{23a1}
\end{figure}

\begin{figure}[ht]	
	\begin{center}
		\begin{minipage}{3cm}		
			\begin{center}
				\includegraphics[width=3cm]{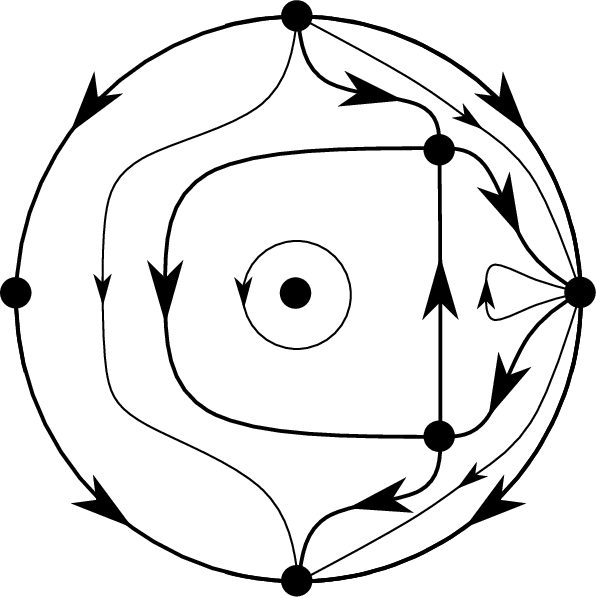}
				
				$1$
			\end{center}		
		\end{minipage}
		\begin{minipage}{3cm}		
			\begin{center}
				\includegraphics[width=3cm]{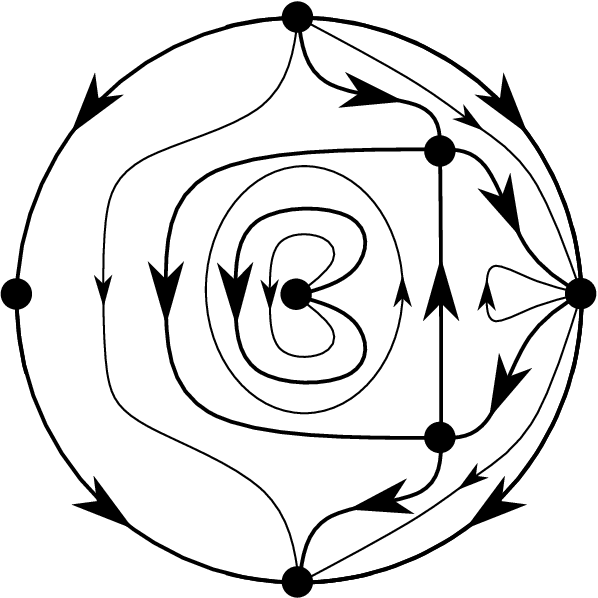}
				
				$2$
			\end{center}		
		\end{minipage}
		\begin{minipage}{3cm}		
			\begin{center}
				\includegraphics[width=3cm]{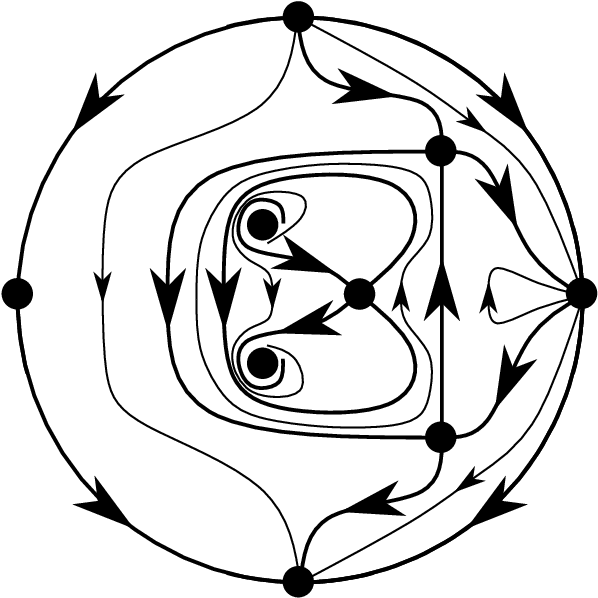}
				
				$3$
			\end{center}
		\end{minipage}
		\begin{minipage}{3cm}	
			\begin{center}
				\includegraphics[width=3cm]{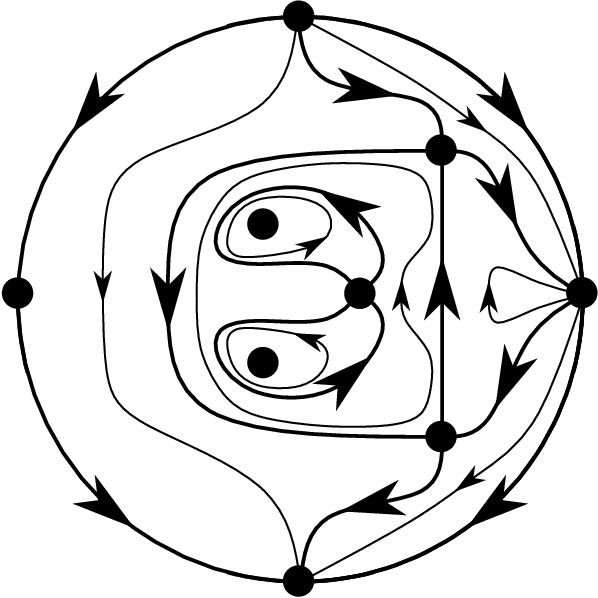}
				
				$4$
			\end{center}
		\end{minipage}
	\end{center}
	\vspace{0.2cm}
	\begin{center}
		\begin{minipage}{3cm}	
			\begin{center}
				\includegraphics[width=3cm]{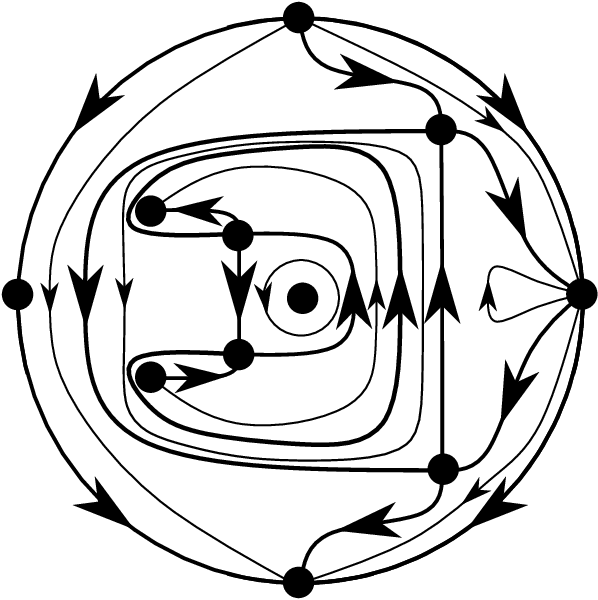}
				
				$5$
			\end{center}		
		\end{minipage}
		\begin{minipage}{3cm}		
			\begin{center}
				\includegraphics[width=3cm]{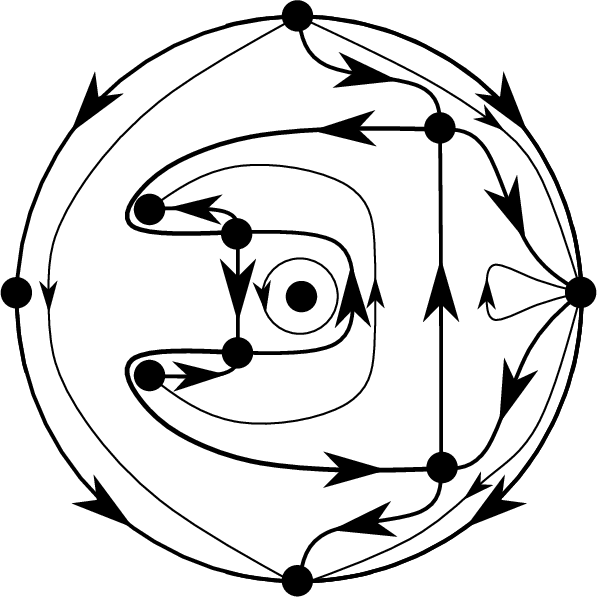}
				
				$6$
			\end{center}		
		\end{minipage}
		\begin{minipage}{3cm}	
			\begin{center}
				\includegraphics[width=3cm]{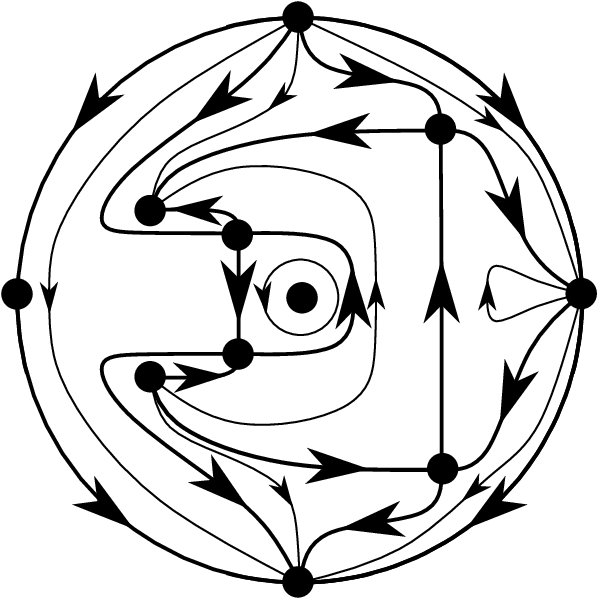}
				
				$7$
			\end{center}		
		\end{minipage}
		\begin{minipage}{3cm}			
			\begin{center}
				\includegraphics[width=3cm]{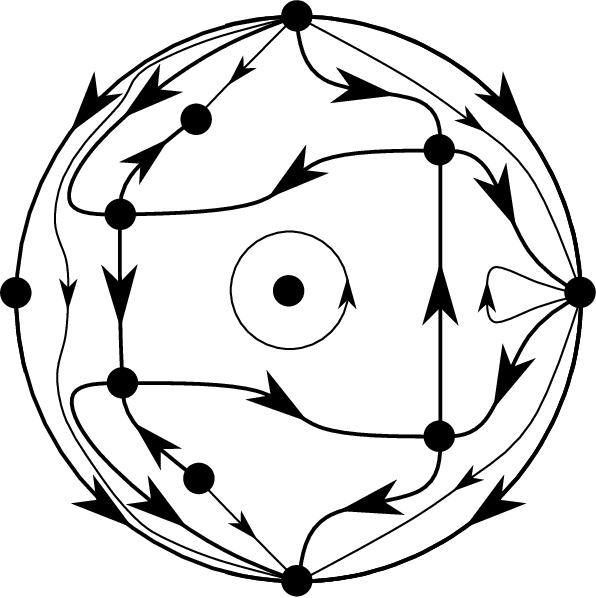}
				
				$8$
			\end{center}
		\end{minipage}
	\end{center}
	\vspace{0.2cm}
	\begin{center}	
		\begin{minipage}{3cm}	
			\begin{center}
				\includegraphics[width=3cm]{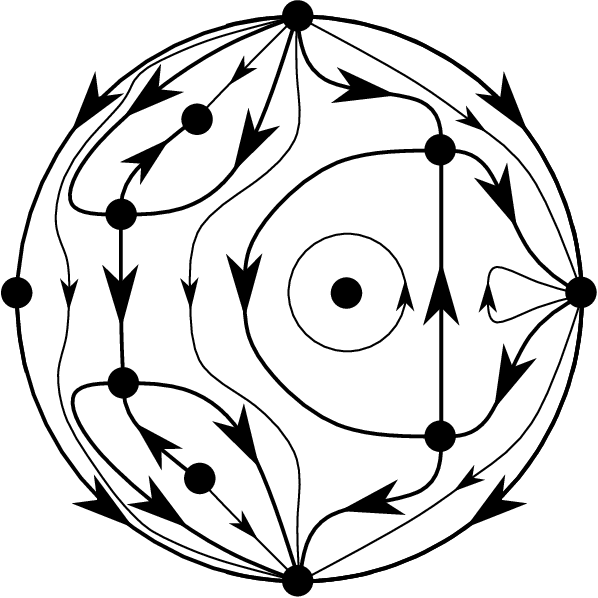}
				
				$9$
			\end{center}		
		\end{minipage}
		\begin{minipage}{3cm}			
			\begin{center}
				\includegraphics[width=3cm]{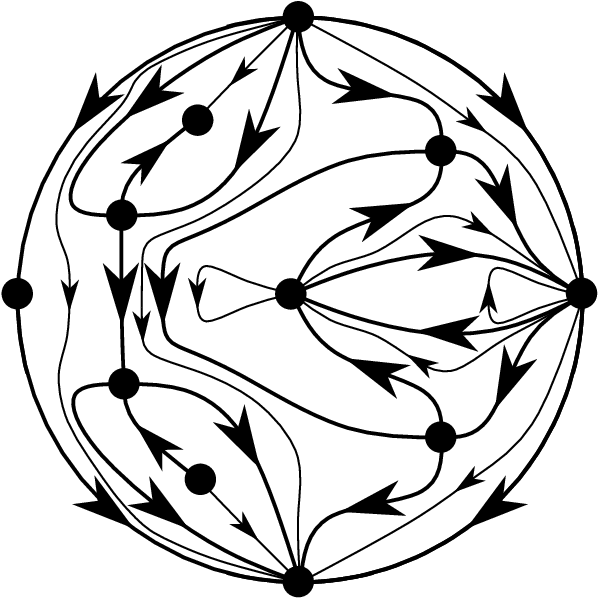}
				
				$10$
			\end{center}		
		\end{minipage}
		\begin{minipage}{3cm}			
			\begin{center}
				\includegraphics[width=3cm]{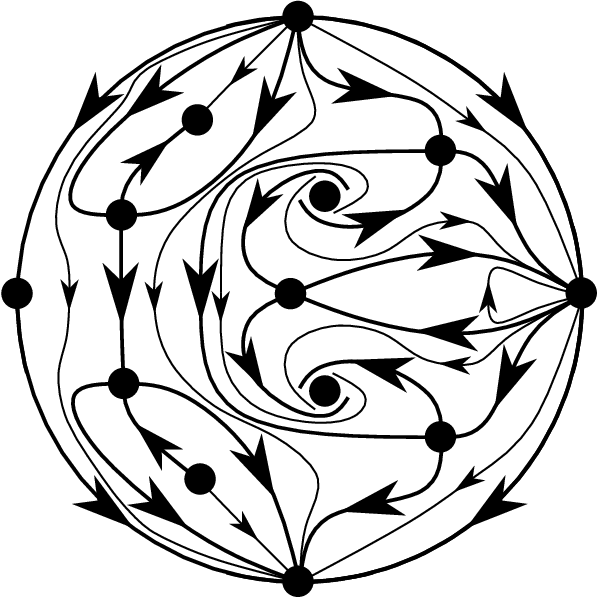}
				
				$11$
			\end{center}		
		\end{minipage}
		\begin{minipage}{3cm}			
			\begin{center}
				\includegraphics[width=3cm]{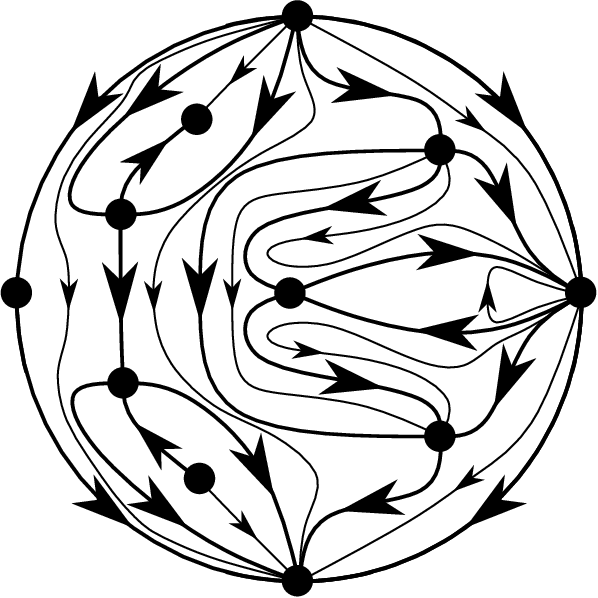}
				
				$12$
			\end{center}
		\end{minipage}
	\end{center}
	\vspace{0.2cm}
	\begin{center}
		\begin{minipage}{3cm}			
			\begin{center}
				\includegraphics[width=3cm]{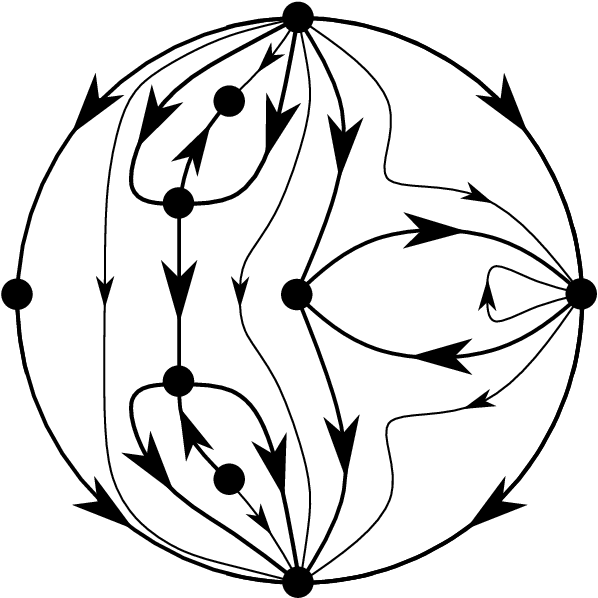}
				
				$13$
			\end{center}
		\end{minipage}
		\begin{minipage}{3cm}			
			\begin{center}
				\includegraphics[width=3cm]{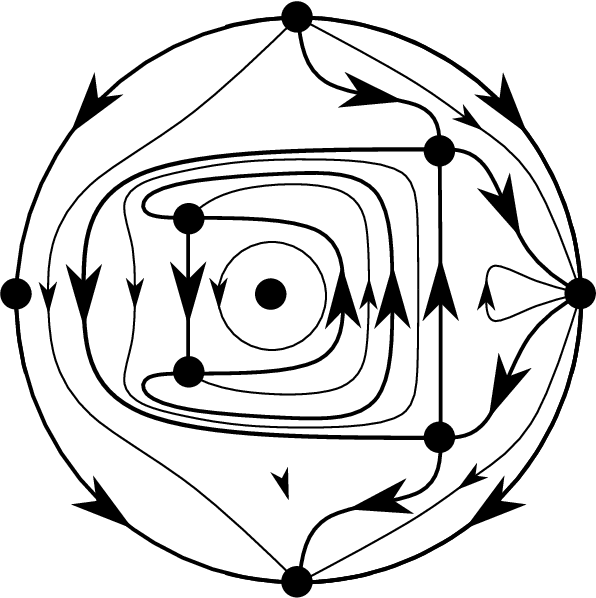}
				
				$14$
			\end{center}		
		\end{minipage}
		\begin{minipage}{3cm}		
			\begin{center}
				\includegraphics[width=3cm]{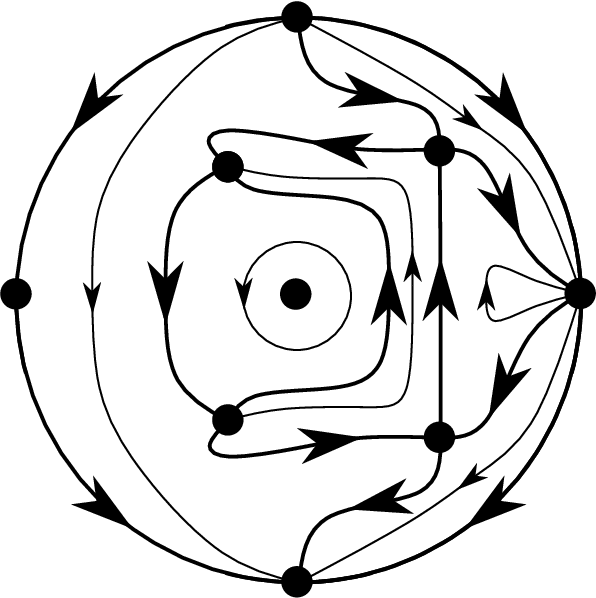}
				
				$15$
			\end{center}		
		\end{minipage}
		\begin{minipage}{3cm}	
			\begin{center}
				\includegraphics[width=3cm]{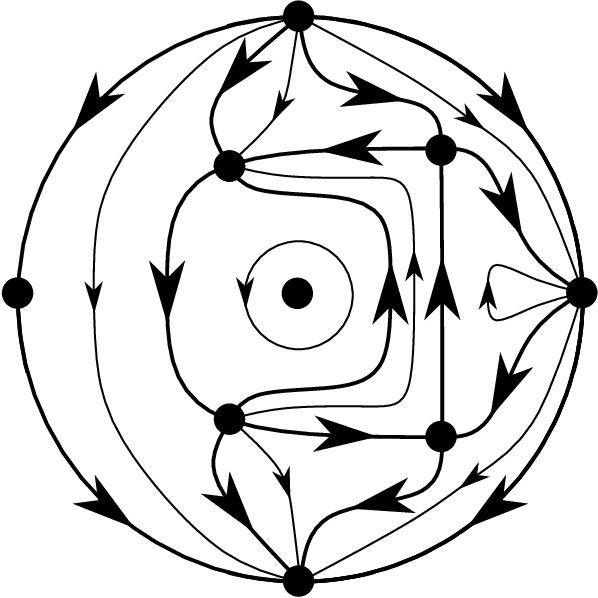}
				
				$16$
			\end{center}		
		\end{minipage}
	\end{center}
	\vspace{0.2cm}
	\begin{center}
		\begin{minipage}{3cm}	
			\begin{center}
				\includegraphics[width=3cm]{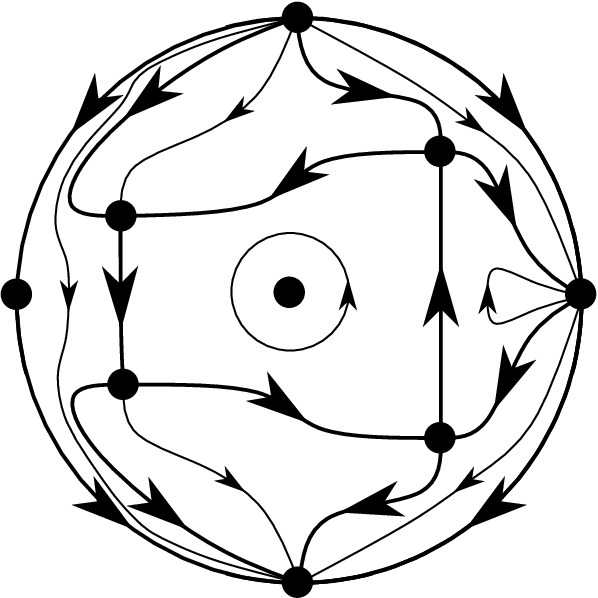}
				
				$17$
			\end{center}
		\end{minipage}
		\begin{minipage}{3cm}			
			\begin{center}
				\includegraphics[width=3cm]{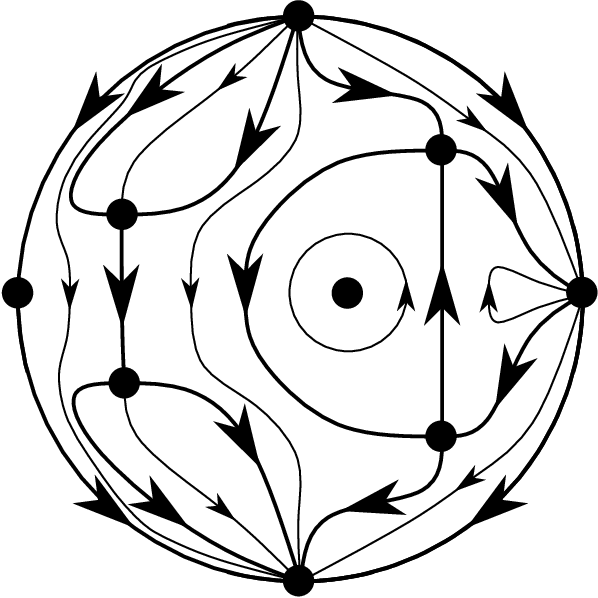}
				
				$18$
			\end{center}		
		\end{minipage}
		\begin{minipage}{3cm}			
			\begin{center}
				\includegraphics[width=3cm]{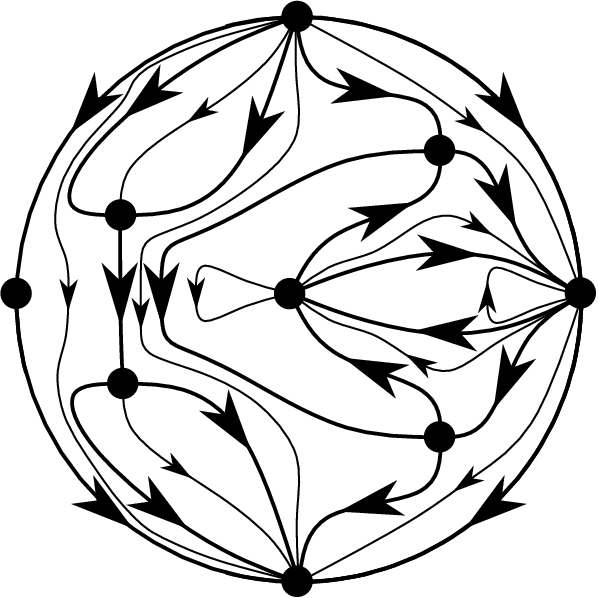}
				
				$19$
			\end{center}		
		\end{minipage}
		\begin{minipage}{3cm}			
			\begin{center}
				\includegraphics[width=3cm]{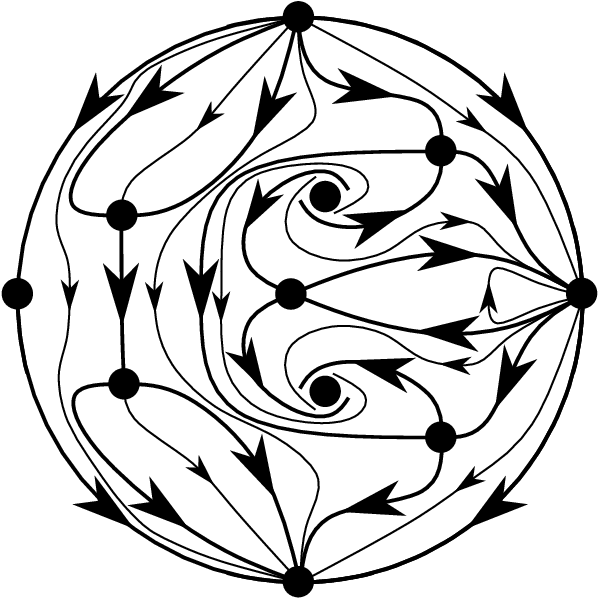}
				
				$20$
			\end{center}	
		\end{minipage}	
	\end{center}
	\caption{Phase portraits of $X_{23a}$ with $a=1$.}\label{23a2}
\end{figure}

\begin{figure}[ht]
	\begin{center}
		\begin{minipage}{3cm}
			\begin{center}
				\includegraphics[width=3cm]{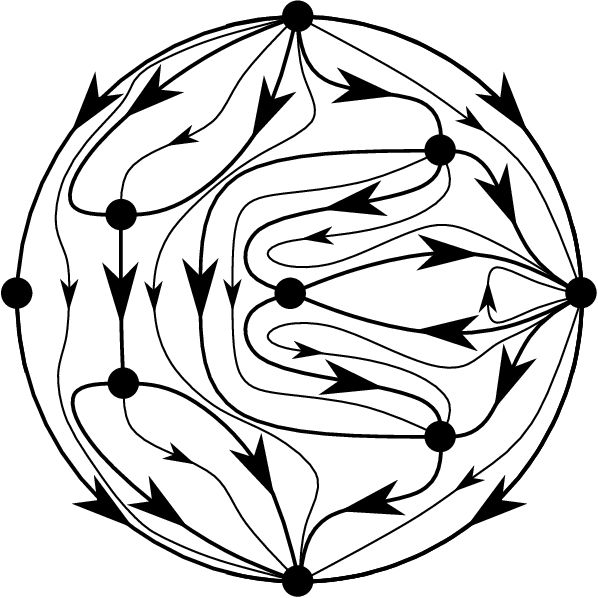}
				
				$21$
			\end{center}
		\end{minipage}
		\begin{minipage}{3cm}
			\begin{center}
				\includegraphics[width=3cm]{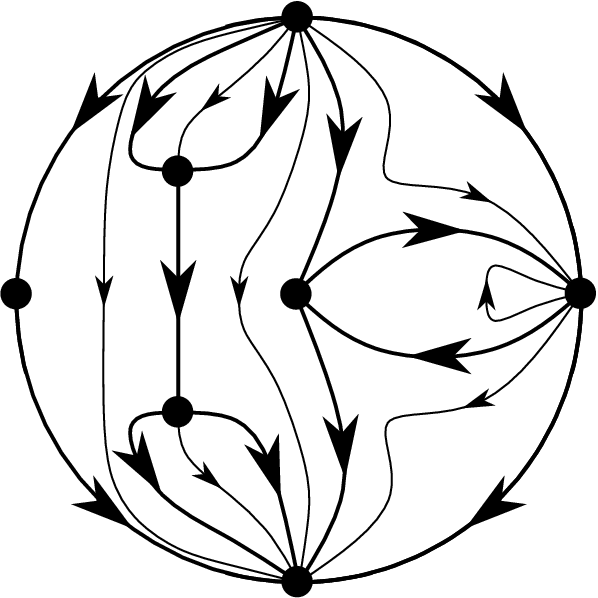}
				
				$22$
			\end{center}
		\end{minipage}
		\begin{minipage}{3cm}		
			\begin{center}
				\includegraphics[width=3cm]{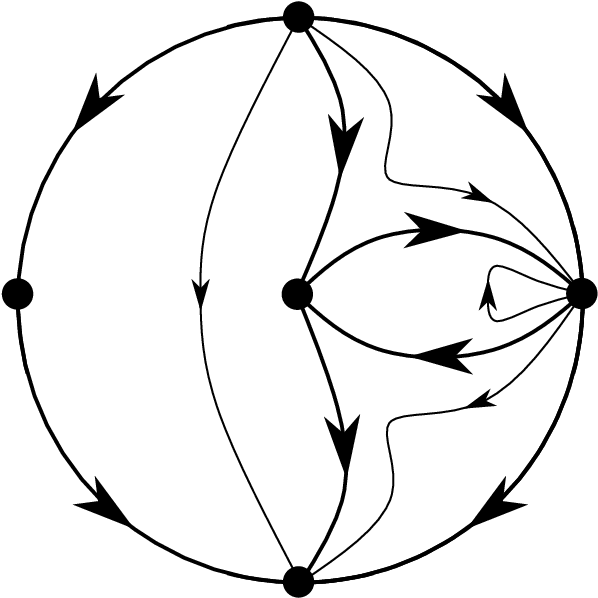}
				
				$23$
			\end{center}		
		\end{minipage}
		\begin{minipage}{3cm}
			\begin{center}
				\includegraphics[width=3cm]{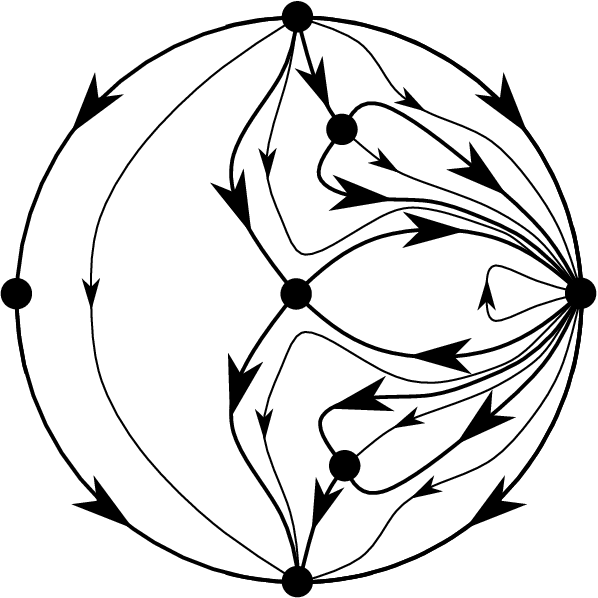}
				
				$24$
			\end{center}
		\end{minipage}
	\end{center}
	\vspace{0.2cm}
	\begin{center}
		\begin{minipage}{3cm}	
			\begin{center}
				\includegraphics[width=3cm]{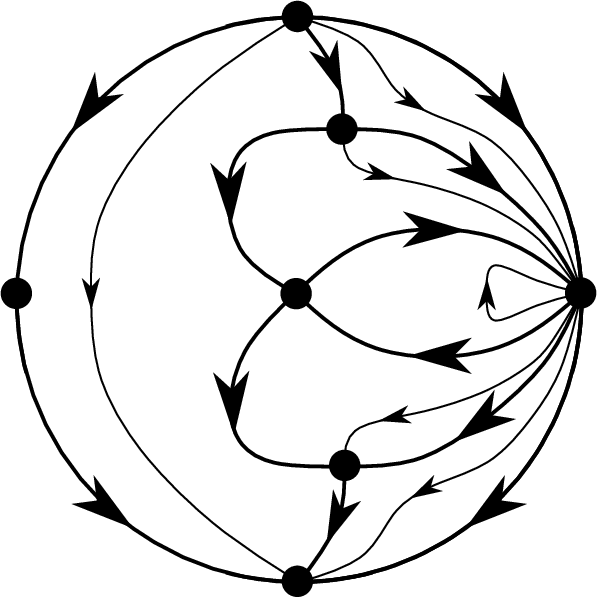}
				
				$25$
			\end{center}
		\end{minipage}
		\begin{minipage}{3cm}	
			\begin{center}
				\includegraphics[width=3cm]{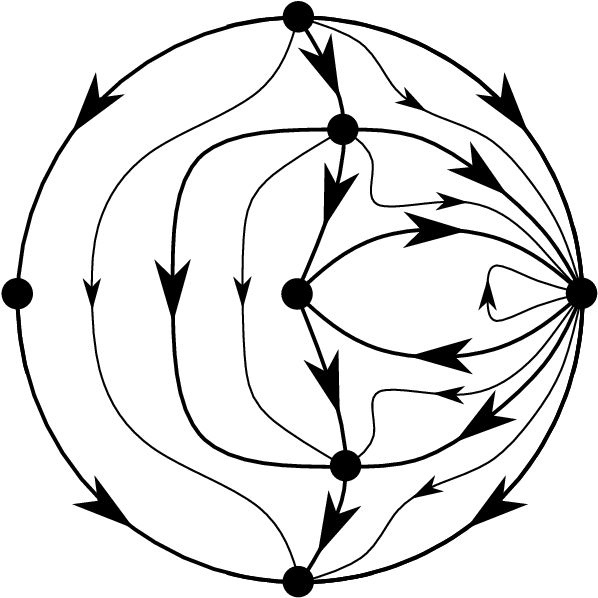}
				
				$26$
			\end{center}		
		\end{minipage}
		\begin{minipage}{3cm}	
			\begin{center}
				\includegraphics[width=3cm]{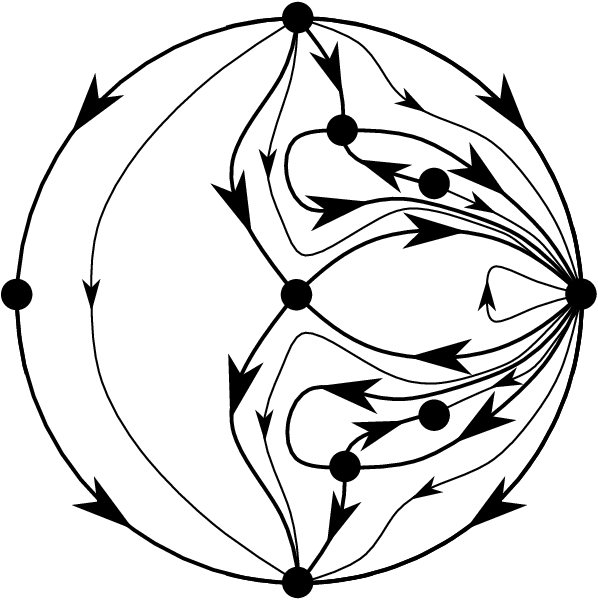}
				
				$27$
			\end{center}			
		\end{minipage}
		\begin{minipage}{3cm}		
			\begin{center}
				\includegraphics[width=3cm]{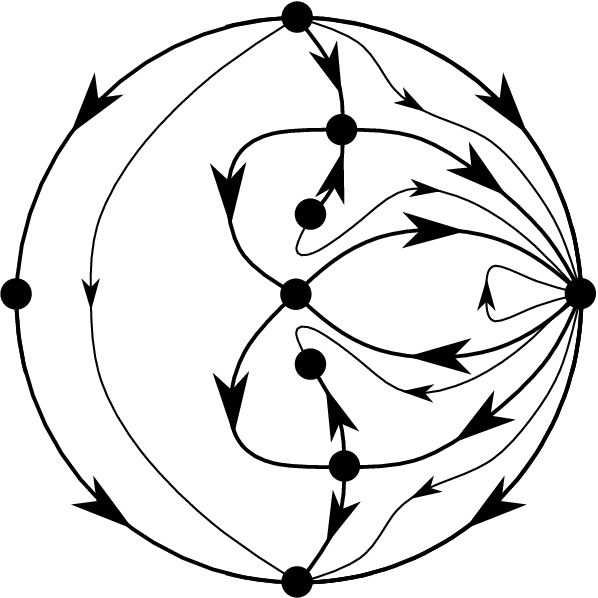}
				
				$28$
			\end{center}		
		\end{minipage}
	\end{center}
	\vspace{0.2cm}
	\begin{center}
		\begin{minipage}{3cm} 	
			\begin{center}
				\includegraphics[width=3cm]{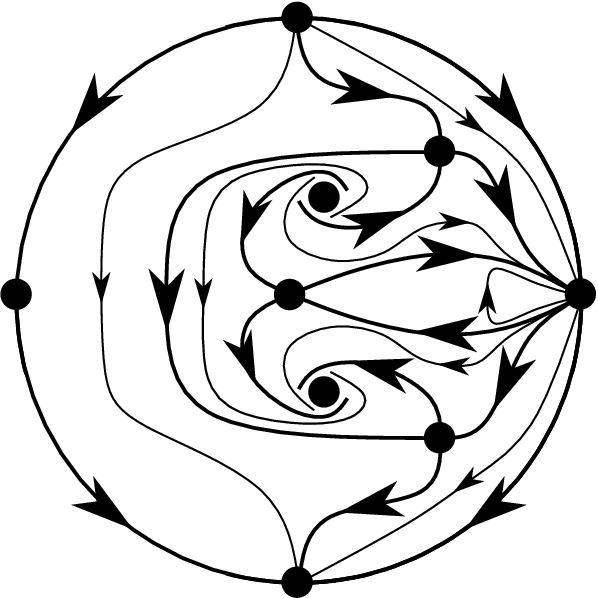}
				
				$29$
			\end{center}		
		\end{minipage}
		\begin{minipage}{3cm}		
			\begin{center}
				\includegraphics[width=3cm]{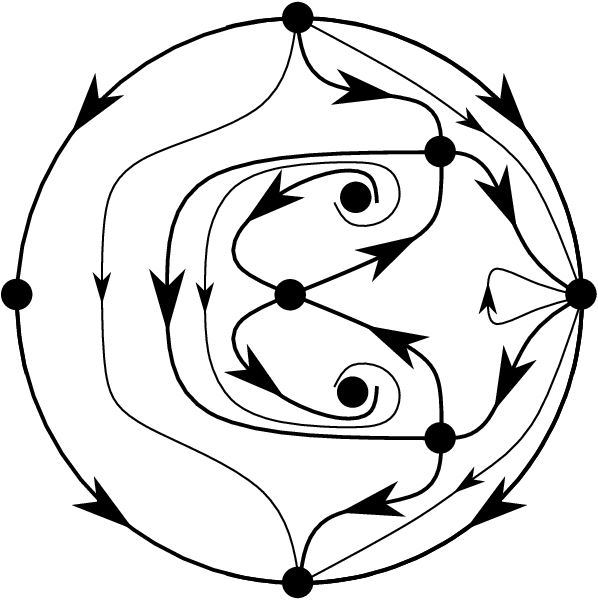}
				
				$30$
			\end{center}
		\end{minipage}
		\begin{minipage}{3cm}		
			\begin{center}
				\includegraphics[width=3cm]{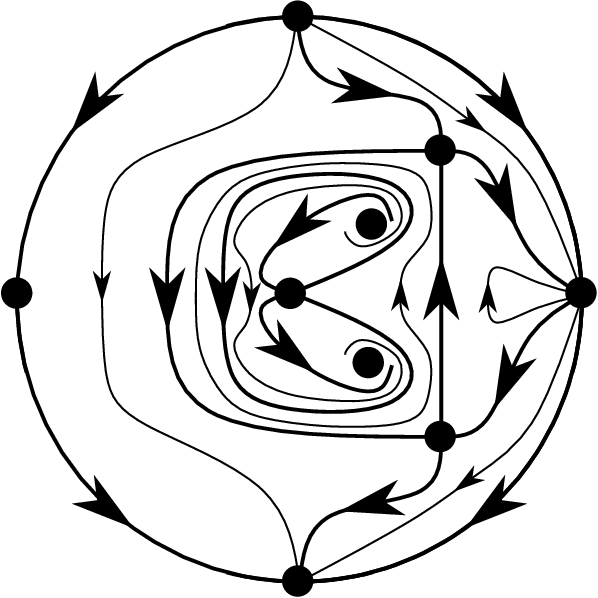}
				
				$31$
			\end{center}
		\end{minipage}
		\begin{minipage}{3cm}	
			\begin{center}
				\includegraphics[width=3cm]{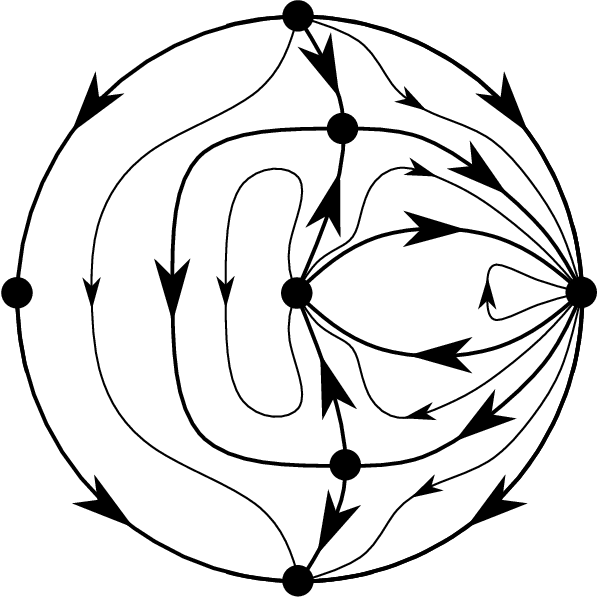}
				
				$32$
			\end{center}		
		\end{minipage}
	\end{center}
	\caption{Phase portraits of $X_{23a}$ with $a=1$.}\label{23a3}
\end{figure}

\begin{figure}[ht]	
	\begin{center}
		\begin{overpic}[width=3cm]{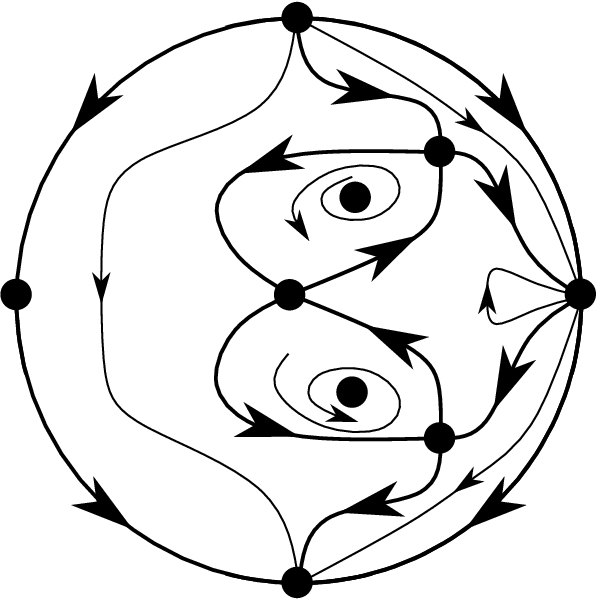} 
		\end{overpic}
	\end{center}	
	\caption{We do not know if the dotted lines defined by phase portraits $28$ and $30$ of Figu\-re~\ref{23a1} intersect each other. But if it happens, then this is the phase portrait on the intersection.}\label{23a4}
\end{figure}

\begin{figure}[ht]	
	\begin{center}
		\begin{overpic}[width=12cm]{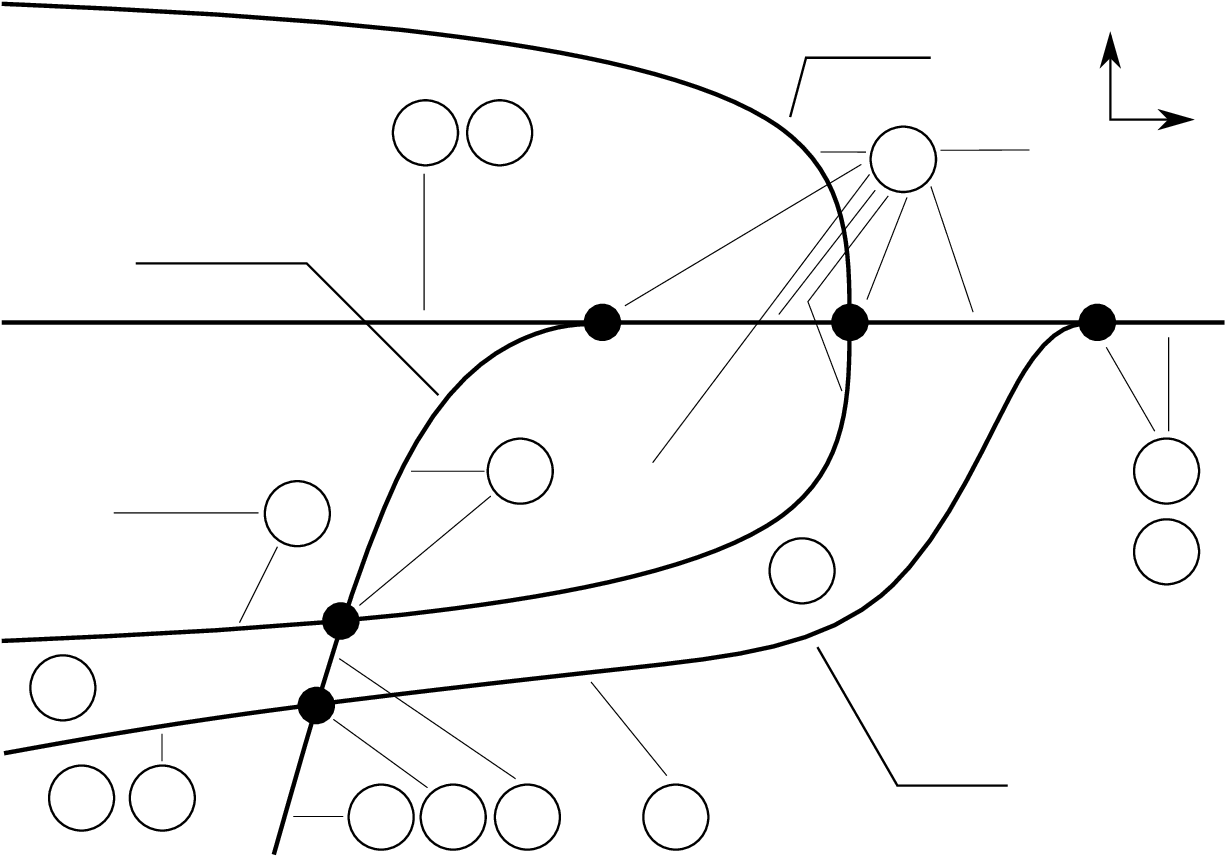} 
			\put(92,65){$\beta$}
			\put(96,61){$\alpha$}
			\put(92,44.5){$\beta=0$}
			\put(66,66){$\alpha+\beta^2=0$}
			\put(75,7){$\beta_+(\alpha)$}
			\put(13,49.5){$\beta_-(\alpha)$}
			\put(73,55.5){$1$}
			\put(94.5,30.25){$2$}
			\put(94.5,23.5){$3$}
			\put(40,57.75){$4$}
			\put(34,57.75){$5$}
			\put(23.5,26.75){$6$}
			\put(4.5,12.75){$7$}
			\put(12.5,3.5){$8$}
			\put(5.75,3.5){$9$}
			\put(41,30.25){$10$}
			\put(41.5,2){$11$}
			\put(35.5,2){$12$}
			\put(29.5,2){$13$}
			\put(63.75,22){$14$}
			\put(53.5,2){$15$}
		\end{overpic}
	\end{center}	
	\caption{Bifurcation diagram of $X_{23b}$ with $a=-1$.}\label{23b1}
\end{figure}

\begin{figure}[ht]	
	\begin{center}
		\begin{minipage}{3cm}	
			\begin{center}
				\includegraphics[width=3cm]{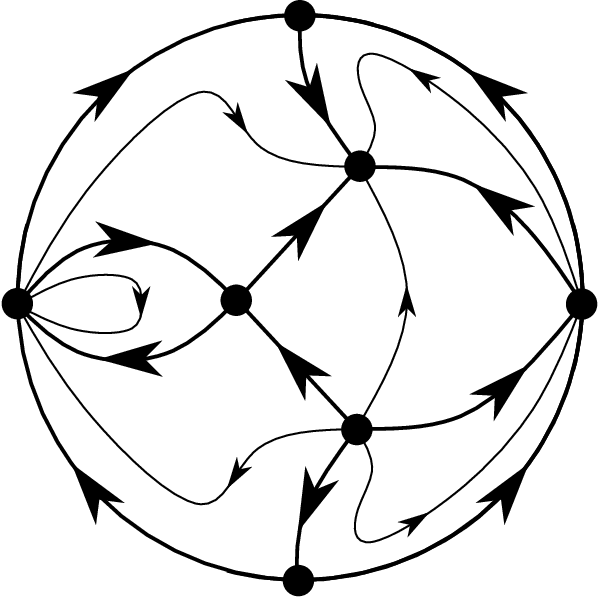}
				
				$1$
			\end{center}		
		\end{minipage}
		\begin{minipage}{3cm} 		
			\begin{center}
				\includegraphics[width=3cm]{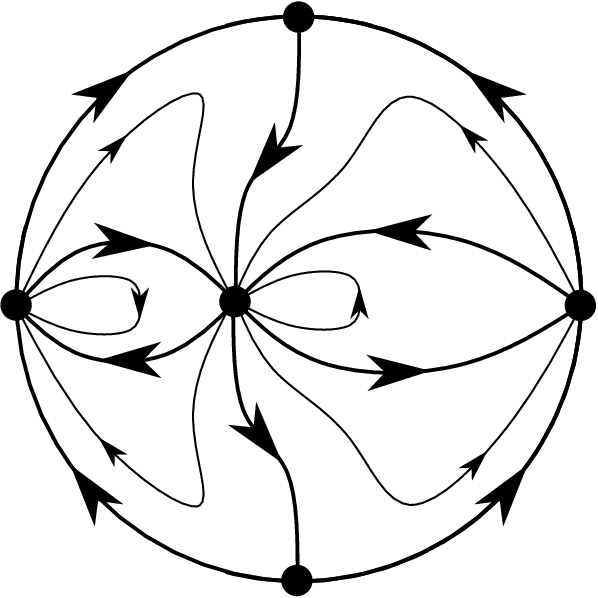}
				
				$2$
			\end{center}		
		\end{minipage}
		\begin{minipage}{3cm} 		
			\begin{center}
				\includegraphics[width=3cm]{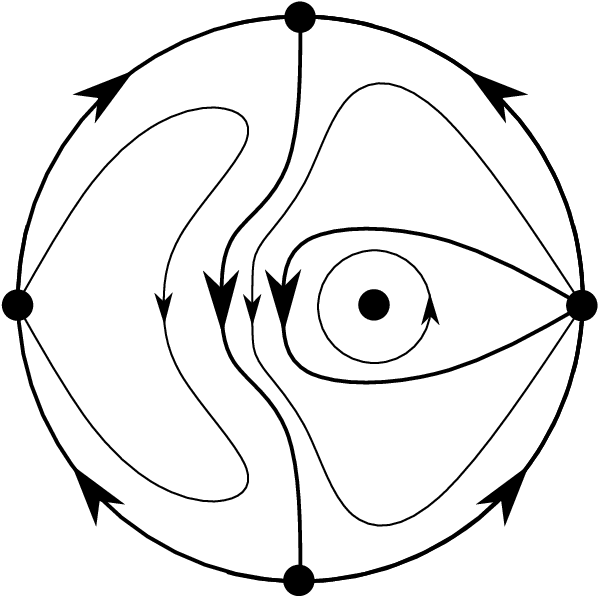}
				
				$3$
			\end{center}
		\end{minipage}
		\begin{minipage}{3cm} 			
			\begin{center}
				\includegraphics[width=3cm]{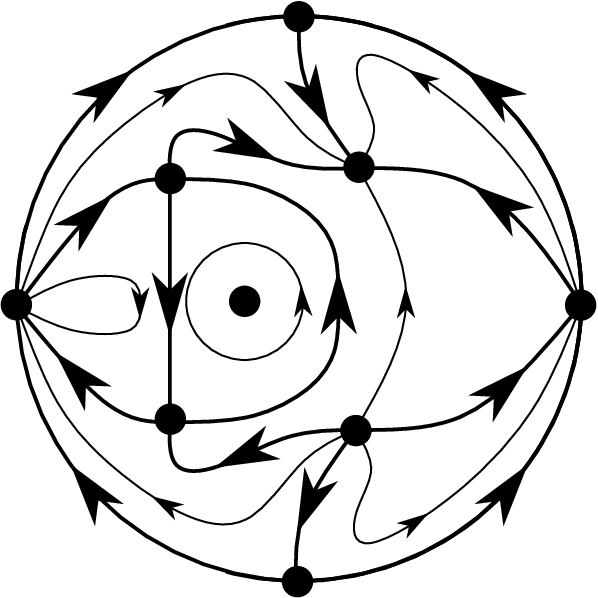}
				
				$4$
			\end{center}
		\end{minipage}
	\end{center}
	\vspace{0.2cm}
	\begin{center}
		\begin{minipage}{3cm} 		
			\begin{center}
				\includegraphics[width=3cm]{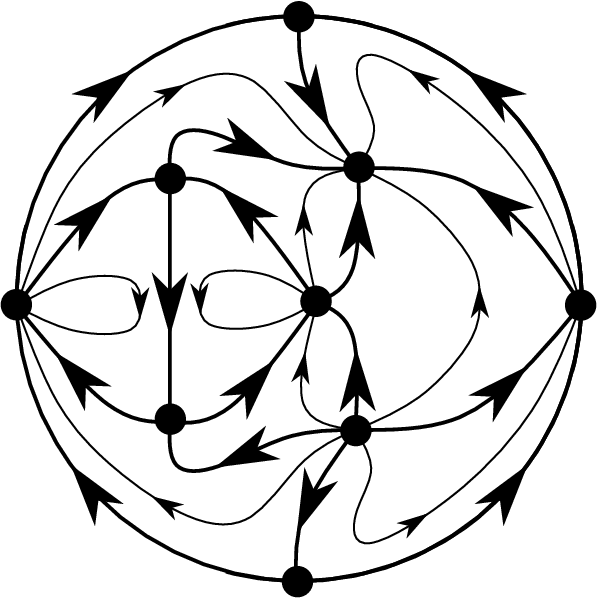}
				
				$5$
			\end{center}		
		\end{minipage}
		\begin{minipage}{3cm} 	
			\begin{center}
				\includegraphics[width=3cm]{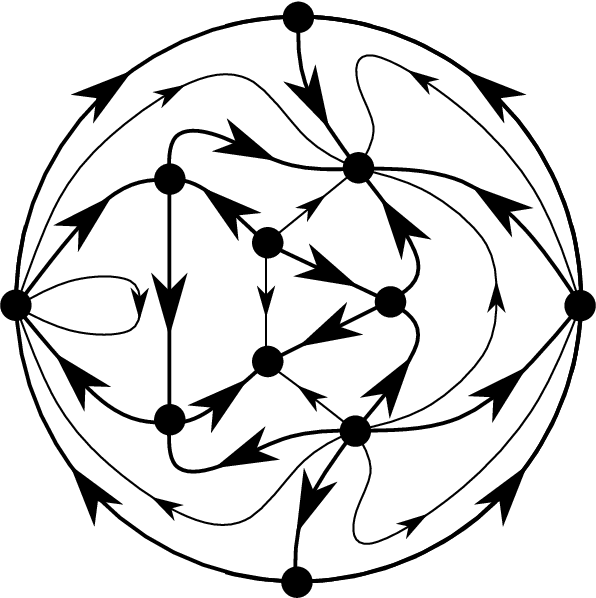}
				
				$6$
			\end{center}		
		\end{minipage}
		\begin{minipage}{3cm} 	
			\begin{center}
				\includegraphics[width=3cm]{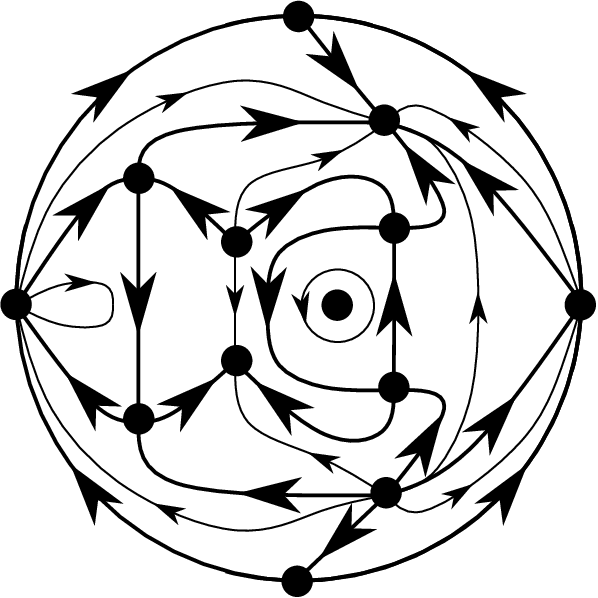}
				
				$7$
			\end{center}		
		\end{minipage}
		\begin{minipage}{3cm} 		
			\begin{center}
				\includegraphics[width=3cm]{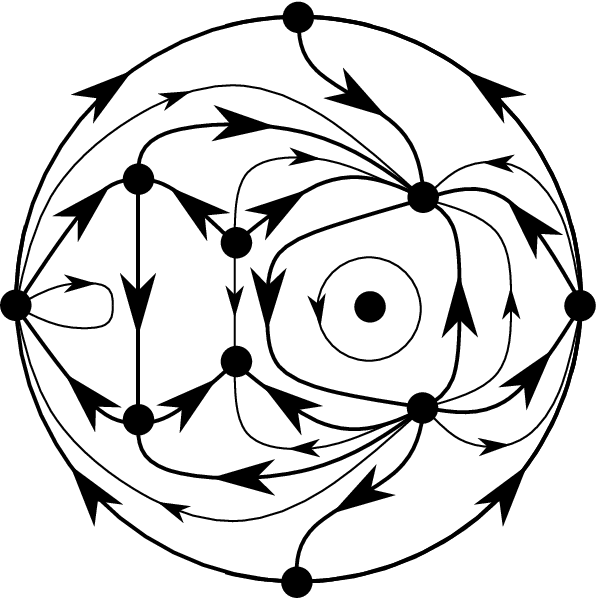}
				
				$8$
			\end{center}
		\end{minipage}
	\end{center}
	\vspace{0.2cm}
	\begin{center}
		\begin{minipage}{3cm} 			
			\begin{center}
				\includegraphics[width=3cm]{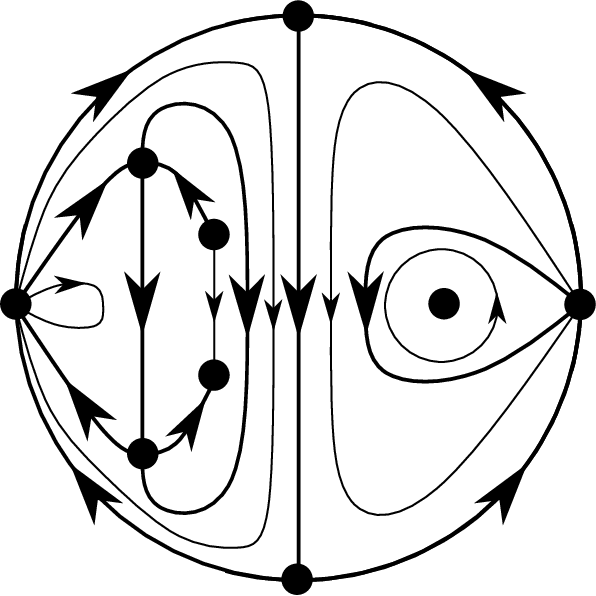}
				
				$9$
			\end{center}		
		\end{minipage}
		\begin{minipage}{3cm} 
			\begin{center}	
				\includegraphics[width=3cm]{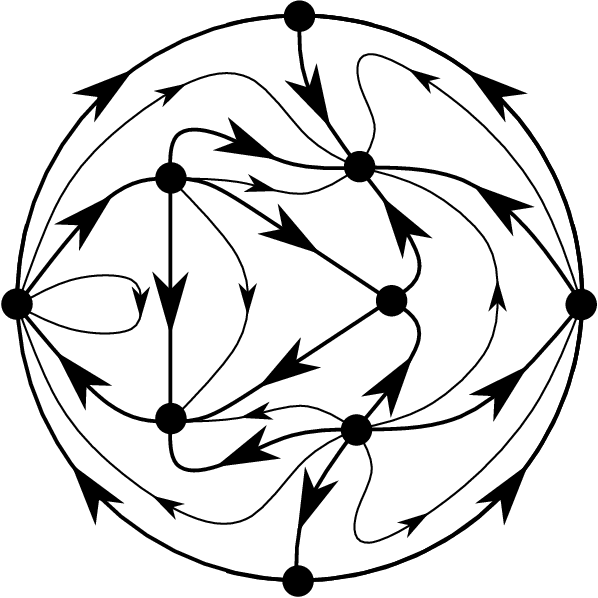}
				
				$10$
			\end{center}		
		\end{minipage}
		\begin{minipage}{3cm} 			
			\begin{center}
				\includegraphics[width=3cm]{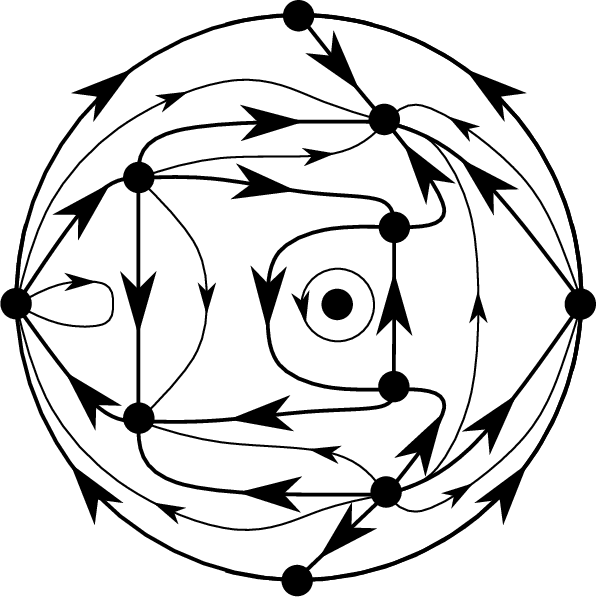}
				
				$11$
			\end{center}		
		\end{minipage}
		\begin{minipage}{3cm} 			
			\begin{center}
				\includegraphics[width=3cm]{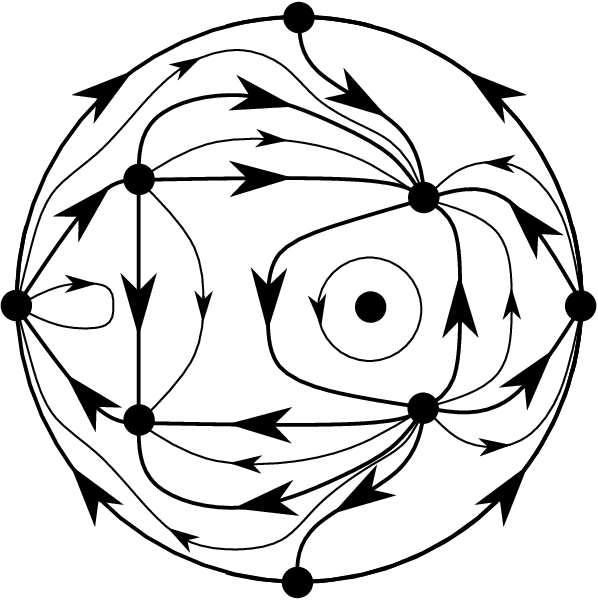}
				
				$12$
			\end{center}
		\end{minipage}
	\end{center}
	\vspace{0.2cm}
	\begin{center}
		\begin{minipage}{3cm} 			
			\begin{center}
				\includegraphics[width=3cm]{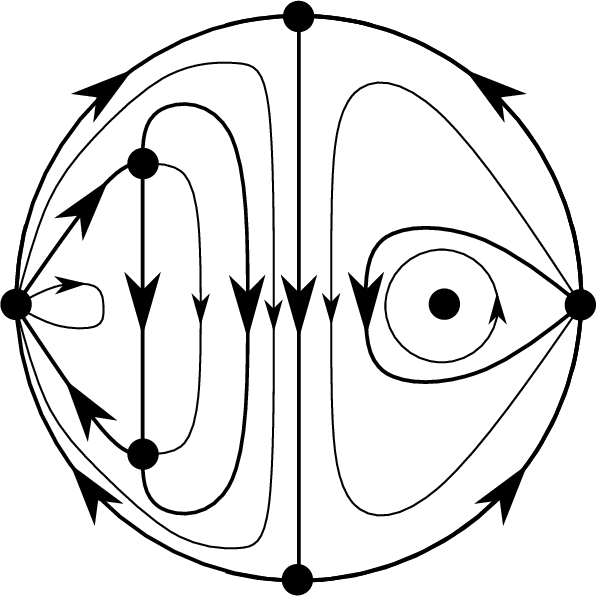}
				
				$13$
			\end{center}
		\end{minipage}
		\begin{minipage}{3cm} 		
			\begin{center}
				\includegraphics[width=3cm]{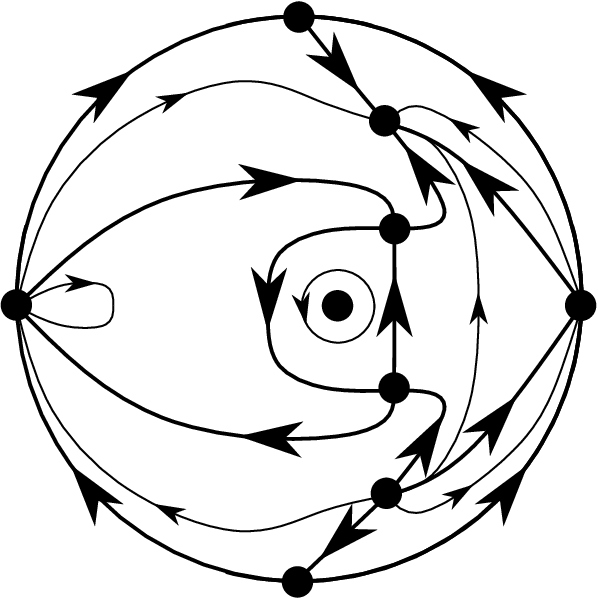}
				
				$14$
			\end{center}		
		\end{minipage}
		\begin{minipage}{3cm} 	
			\begin{center}
				\includegraphics[width=3cm]{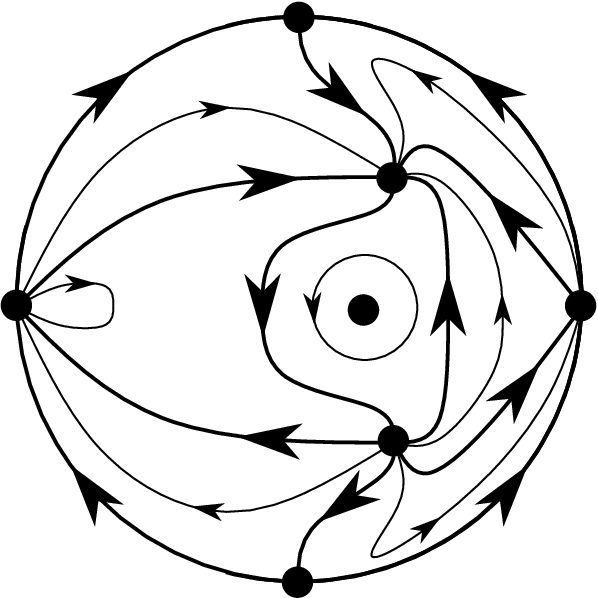}
				
				$15$
			\end{center}		
		\end{minipage}
	\end{center}
	\caption{Phase portraits of $X_{23}$ with $a=-1$ and $(x,y)\mapsto(-x,y)$.}\label{23b2}
\end{figure}

\end{document}